\documentclass[11pt]{imsart}
\usepackage{geometry}
\usepackage{graphicx} 
\usepackage{amsfonts}
\usepackage{amsmath}
\usepackage{amssymb, amsthm}
\usepackage{url}
\usepackage{comment}
\usepackage{appendix}
\usepackage{footnote}
\usepackage{bbm}
\usepackage{natbib}
\usepackage{mathtools}

\newtheorem {Proposition}{Proposition}[section]
\newtheorem {Lemma}[Proposition] {Lemma}
\newtheorem {Theorem}[Proposition]{Theorem}
\newtheorem {Corollary}[Proposition]{Corollary}
\newtheorem {Remark}[Proposition]{Remark}

\def\log{\mathop{\rm log}\nolimits}

\def\Var{\mathrm{Var}}

\def\x{\mathbf{x}}
\def\y{\mathbf{y}}

\newcommand{\E}{\mathbb{E}}

\newcommand{\N}{\mathbb{N}}
\newcommand{\NN}{\mathbb{N}}

\renewcommand{\P}{\mathbb{P}}

\newcommand{\R}{\mathbb{R}}
\newcommand{\RR}{\mathbb{R}}

\newcommand{\U}{\mathbf{U}} 

\newcommand{\X}{\mathbf{X}} 
\newcommand{\Y}{\mathbf{Y}} 

\newcommand{\Zb}{\mathbf{Z}}

\newcommand{\rD}{\mathrm{D}}

\newcommand{\rF}{\mathrm{F}}

\newcommand{\rK}{\mathrm{K}}

\newcommand{\rM}{\mathrm{M}}

\newcommand{\rS}{\mathrm{S}}
\newcommand{\rT}{\mathrm{T}}

\newcommand{\cB}{\mathcal{B}}

\newcommand{\cK}{\mathcal{K}}

\newcommand{\cM}{\mathcal{M}}
\newcommand{\cN}{\mathcal{N}}
\newcommand{\cO}{\mathcal{O}}
\newcommand{\cP}{\mathcal{P}}

\newcommand{\cU}{\mathcal{U}}

\newcommand{\bi}{\mathbf{i}}
\newcommand{\bj}{\mathbf{j}}
\newcommand{\nnk}{[[n]]^k}
\newcommand{\nn}[1]{[[n]]^{#1}}
\newcommand{\mmk}{[[m]]^k}
\newcommand{\mm}[1]{[[m]]^{#1}}

\def\one{\mathbf{1}}

\def\mP{\mathrm{P}}
\def\mQ{\mathrm{Q}}
\def\mPn{\mathrm{P}_n}
\def\mQm{\mathrm{Q}_m}
\def\mQn{\mathrm{Q}_m} 

\def\AP{\mathcal{A}_{\mP}}
\def\APn{\mathcal{A}_{\mP_n}}
\def\AQ{\mathcal{A}_{\mQ}}
\def\AQm{\mathcal{A}_{\mQ_m}}
\def\AQn{\mathcal{A}_{\mQ_m}}

\def\DPn{\mathcal{D}_{\mP_n}}

\def\DQm{\mathcal{D}_{\mQm}}

\def\MP{\cM_{\mP}}

\def\MQ{\cM_{\mQ}}

\def\bAP{ \bar{\mathcal{A}}_{\mP}}
\def\bAPn{\bar{\mathcal{A}}_{\mPn}}

\def\pipq{\pi_{\mP,\mQ}} 
\def\pinm{\pi_{\mPn,\mQm}} 
\def\xipq{\xi_{\mP,\mQ}}    
\def\ximn{\xi_{\mPn,\mQm}}  

\def\bfpq{\bar{f}_{\mP,\mQ}}   
\def\bfnm{\bar{f}_{\mPn,\mQm}} 
\def\bfmn{\bar{f}_{\mPn,\mQm}} 
\def\bgpq{\bar{g}_{\mP,\mQ}}   
\def\bgmn{\bar{g}_{\mPn,\mQm}} 

\def\Nnm{\left\lceil \log_+^{1/2}(\frac{nm}{n+m})\right\rceil}

\def\bAP{\bar{\mathcal{A}}_{\mP}}
\usepackage{enumerate}
\usepackage{natbib}
\usepackage{bbm}
\usepackage[T1]{fontenc}    
\usepackage{lmodern}        
\usepackage{hyperref}
\usepackage{xurl}
\hypersetup{breaklinks=true}     
\usepackage{url}            
\usepackage{booktabs}       
\usepackage{amsfonts}       
\usepackage{nicefrac}       
\usepackage{microtype}      
\usepackage{graphicx}
\usepackage{color}
\usepackage{xcolor}
\graphicspath{ {./images/} }
\DeclareSymbolFontAlphabet{\amsmathbb}{AMSb}%
\newcommand{\norm}[1]{\left\|#1\right\|}

\newcommand{\coloneqq}{:=}

\newcommand{\convW}{\xrightarrow{\;\;w\;\;}}
\newcommand{\convP}{\xrightarrow{\;\;p\;\;}}
\newcommand{\convAS}{\xrightarrow{\;\;a.s.\;\;}}
\newcommand{\Op}{\cO_p}
\newcommand{\op}{o_p}
\newcommand{\opnm}{\op\left(\sqrt{\frac{n+m}{nm}}\right)}

\begin{document}
\title{Weak Limits for Empirical Entropic Optimal Transport: Beyond Smooth Costs}
\date{\today}
\runauthor{Gonz\'alez-Sanz and Hundrieser}
\runtitle{Weak limits for Empirical Entropic Optimal Transport}
\begin{aug}
\author[A]{\fnms{Alberto} \snm{Gonz\'alez-Sanz}\ead[label=e1]{alberto.gonzalez.sanz@uva.es}}
\and
\author[B]{\fnms{Shayan} \snm{Hundrieser}\ead[label=e2]{s.hundrieser@math.uni-goettingen.de}}

\address[A]{IMUVA, Universidad de Valladolid, Spain; IMT, Universit\'e de Toulouse III, France, \printead{e1}}

\address[B]{IMS, University of G\"ottingen, Germany, \printead{e2}}
\end{aug}

\begin{abstract}
 We establish weak limits for the empirical entropy regularized optimal transport cost, the expectation of the empirical plan and the conditional expectation. Our results require only uniform boundedness of the cost function and no smoothness properties, thus emphasizing the far-reaching regularizing nature of entropy penalization. To derive these results, we employ a novel technique that sidesteps the intricacies linked to empirical process theory and the control of suprema of function classes determined by the cost. Instead, we perform a careful linearization analysis for entropic optimal transport with respect to an empirical $L^2$-norm, which enables a streamlined analysis. As a consequence, our work gives rise to new implications for a multitude of transport-based applications under general costs, including pointwise distributional limits for the empirical entropic optimal transport map estimator, kernel methods as well as regularized colocalization curves. Overall, our research lays the foundation for an expanded framework of statistical inference with empirical entropic optimal transport.
\end{abstract}



\maketitle
\section{Introduction}\label{sec:intro}

For a Polish space $\Omega$ and a cost function $c:\Omega\times\Omega\to \R$ the optimal transport problem between two probability measures $\mP,{\rm Q}\in \mathcal{P}(\Omega)$ is defined~as
\begin{equation}
    \label{eq:OT}
    \rT(\mP,\mQ)=\min_{\pi\in \Pi(\mP,\mQ)} \int_{\Omega\times \Omega} c(\x,\y) d\pi(\x,\y),
\end{equation}
where $\Pi(\mP,\mQ)$ represents the set of couplings between $\mP$ and $\mQ$. 
Over the last several decades, there has been substantial progress in the understanding of optimal transport theory, with comprehensive monographs detailing analytical  \citep{vil03,villani2008optimal,santambrogio2015optimal}, computational \citep{cuturi18}, and statistical aspects \citep{panaretos2020an}. In the context of data science,  optimal transport offers two key entities: optimizers for \eqref{eq:OT}, referred to as \emph{optimal transport plans}, and the optimal value  $\rT(\mP,\mQ)$, known as the \emph{optimal transport cost}. 
The optimal transport plan has been used for domain adaptation \citep{Courty2017OptimalTF}, in multivariate rank based inference \citep{DELBARRIO2020104671,Hallin2020DistributionAQ,deb2019multivariate}, for the definition of a kernel on the space of probability measures \citep{bachoc2020gaussian}, to define a robust classifier \citep{Serrurier_2021_CVPR,BethumePay}, in the context of counterfactual explanations and causal inference \citep{delara2021transportbased,torous2021optimal}, and  in multiple output quantile regression \citep{Barrio2022NonparametricMC}.
Since, for certain cost functions \citep[Chapter~6]{villani2008optimal}, the optimal transport cost serves to define a distance between probability measures it has been used as a discriminator in generative adversarial networks \citep{WGANS}, as a  test statistic \citep{Hallin2020MultivariateGT,Javi2,sommerfeld2018} and as a penalty term in fair learning \citep{Risser2019TacklingAB}. 

Nonetheless, there are two primary drawbacks of optimal transport that impede its utilization for sizable and high-dimensional datasets: its \emph{computational complexity} -- exact solvers generally scale cubically in the number of support points of the two underlying measures, and its \emph{statistical complexity} -- plagued by the curse of dimensionality. %
Indeed, although the statistical complexity of estimating the optimal transport cost is generically determined by the minimum intrinsic dimension of the two measures as well as the regularity of the cost function, as outlined by the recently discovered lower complexity adaptation principle \citep{weed2019sharp,hundrieser2022empirical}, if both population measures have high intrinsic dimensions slow convergence occur, which deteriorate exponentially in the size of the dimension. 

For instance, if the cost $c$ is the squared Euclidean distance and $\Omega=\mathbb{R}^d$, the minimax rate of estimation for the optimal transport cost through a sample $\X_1, \dots, \X_n$ and $\Y_1, \dots, \Y_m$ of i.i.d. random variables  which follow the probability distributions $\mP$ and $\mQ$, respectively, is of order $(n\wedge m)^{-2/(4\vee d)}$ \citep{manole2021sharp}.  
Similar observations have been made for the optimal transport map, i.e., when the optimal transport plan is concentrated the graph of a map \citep{HutterRigollet, manole2021plugin}. This last point is important to provide statistical guarantees for the aforementioned applications of optimal transport, since in practice, practitioners only have access to the mentioned samples and not to the underlying probabilities $\mP$ and $\mQ$. That is, they are only able to calculate the optimal plan and costs between  the empirical measures $\mP_n=\frac{1}{n}\sum_{i=1}^n\delta_{\X_i} $ and  $\mQ_m=\frac{1}{m}\sum_{j=1}^m\delta_{\Y_j} $. 

As a solution to the computational challenges, \cite{cuturi13} proposed adding relative entropy ${\rm KL}$ -- defined for measures $\alpha$ and $\beta$ as ${\rm KL}(\alpha \mid \beta) = \int \log(\frac{d\alpha}{d\beta}(x))d\alpha(x)$ if $\alpha$ is absolutely continuous with respect to $\beta$, $\alpha\ll \beta$, and $+\infty$ otherwise -- to the original transport problem as a penalty term, i.e.,
\begin{equation}\label{kanto_entrop}
\rS_{\epsilon}(\mP,\mQ)=\inf_{\pi\in \Pi(\mP,\mQ)} \int_{\Omega\times \Omega} c(\x,\y)d\pi(\x,\y)+\epsilon {\rm KL}(\pi \mid \mP \otimes \mQ).
\end{equation}
The value $\rS_{\epsilon}(\mP,\mQ)$ is referred to as the \emph{entropy regularized optimal transport cost} and a unique optimizer $\pi_{\mP,\mQ}\in \Pi(\mP,\mQ)$ for \eqref{kanto_entrop} always exists under bounded costs (see, e.g., \cite{nutz2022entropic} or Proposition \ref{prop:structureEOT}). This optimizer is referred to as the \emph{entropy regularized optimal transport plan}. 
Being a strictly convex optimization problem, the entropic optimal transport problem \eqref{kanto_entrop} also admits a dual formulation which is the basis for the Sinkhorn algorithm, offering a computational complexity that is an order of magnitude faster compared to the original linear programming problem  \citep{altschuler2017near, dvurechensky2018computational}.

Moreover, entropy regularization is not only computationally beneficial, but also in terms of statistical complexity (see literature review in Section \ref{subsec:related_work}), with convergence parametric rates of order $(n\wedge m)^{-1/2}$ under fixed regularization parameter $\epsilon>0$ provided that the cost function is sufficiently smooth \citep{Genevay2018SampleCO, mena2019, Chizat2020}. Only recently, it has been discovered by \cite{rigollet2022sample} that this behavior is generically true for any kind of bounded costs, which can be considered as a starting point of this work. 

Motivated by these developments, this work is concerned with providing an exhaustive analysis on distributional limits for empirical entropic optimal transport quantities. This extends a previous line of research by \cite{klatt2020empirical, hundrieser2021limit} for discrete settings and \cite{ del2022improved,gonzalez2022weak, goldfeld2022statistical,Goldfeld2022LimitTF} for smooth costs.
To simplify as much as possible the notation, we assume throughout this work that $\epsilon=1$. This is not a genuine restriction since $\rS_\epsilon(\mP, \mQ) = \epsilon \rS_1'(\mP, \mQ)$ for any $\epsilon>0$, 
where $\rS_1'$ denotes the entropic transport cost between $\mP$ and $\mQ$ for the cost $c/\epsilon$. Further, for notation we  suppress the regularization parameter and write $\rS(\mP, \mQ)\coloneqq\rS_1(\mP, \mQ)$. 

The main results of this paper are stated as follows and are proven in Section \ref{sec:main_results}. \newpage
\begin{Theorem} \label{Theo:MainResults} For a  Polish space $\Omega$  consider a bounded and Borel-measurable cost function $c\in L^{\infty}(\Omega\times \Omega)$. Then, for $n,m\rightarrow \infty$  with $\frac{m}{n+m}\to \lambda \in (0,1)$ the following holds.
\begin{enumerate}%
    \item\label{item:main_costs} The empirical entropic optimal transport cost satisfies asymptotically
    \begin{align*}\sqrt{\frac{n\, m}{n+m}}\Big( \rS(\mPn, \mQm) - \rS(\mP, \mQ) \Big)\convW \mathcal{N}\big(0, \sigma^2_{\lambda}\big)
    \end{align*}
    with asymptotic variance $\sigma^2_{\lambda}\coloneqq\lambda \operatorname{Var}_{\X\sim\mP}(f_{\mP,\mQ}(\X))+(1-\lambda) \operatorname{Var}_{\Y\sim\mQ}(g_{\mP,\mQ}(\Y))$, where  $(f_{\mP,\mQ},g_{\mP,\mQ})$ denotes a  solution for the dual entropic optimal transport problem \eqref{eq:dual_entrop}.
    \item\label{item:main_plan} The expectation of $\eta\in L^\infty(\Omega\times 
\Omega)$ with respect to the empirical entropic optimal transport plan satisfies asymptotically
\begin{align*}
   \sqrt{\frac{n\, m}{n+m}}\left(\E_{(\X,\Y)\sim \pinm}[\eta(\X,\Y)] - \E_{(\X,\Y)\sim \pipq}[\eta(\X,\Y)] \right)%
\convW \cN(0,\sigma^2_{\lambda}(\eta))
\end{align*}
    with asymptotic variance $\sigma^2_{\lambda}(\eta)$ given in \eqref{eq:VarianceEvalPlan} and characterized by $\mP, \mQ$, $c$, $\eta$.
\item\label{item:main_condplan} The conditional expectation of $\eta\in L^\infty(\Omega\times 
\Omega)$ with respect to the empirical entropic optimal transport plan given that $\X = \x\in\Omega$ satisfies asymptotically
\begin{multline*}
        \sqrt{\frac{n\, m}{n+m}}\left(\E_{(\X,\Y)\sim \pinm}[\eta(\X,\Y) | \X = \x] - \E_{(\X,\Y)\sim \pipq}[\eta(\X,\Y)|\X = \x] \right)\\
         \convW \cN(0,\sigma^2_{\lambda}(\eta, \x))
\end{multline*}
    with asymptotic variance 
$   \sigma^2_{\lambda}(\eta, \x)$ given in \eqref{eq:VarianceEvalCondPlan} and characterized by $\mP, \mQ$, $c$, $\eta(\x,\cdot)$.
   \end{enumerate}
\end{Theorem}

It is worth pointing out that the distributional limits are not distribution free, i.e., the variances $\sigma^2_{\lambda}$, $\sigma^2_{\lambda}(\eta)$ and $\sigma^2_{\lambda}(\eta, \x)$ depend on $\mP$ and $\mQ$. For this purpose we provide in Section \ref{sec:main_results} consistent estimators for the asymptotic variances. Moreover, for practical purposes we only focus on the two-sample case since entropic optimal transport quantities are essentially only computable for discrete measures.  %

Remarkably, the weak limits given in Theorem \ref{Theo:MainResults} require no smoothness or regularity conditions on the underlying cost function which paves the way for transport-based inference for metric based costs on metric spaces. In addition, smoothness of the cost is a somewhat obtrusive assumption as it prohibits the use of prominent Euclidean ground costs $(\x,\y)\mapsto |\x-\y|^p$ for $p$ not even. Moreover, Theorem \ref{Theo:MainResults} considerably broadens the perspective of underlying ground spaces and is not limited to bounded sets within (finite dimensional) Euclidean spaces. 
Indeed, current data science is not limited to smooth spaces; for example, time series data is frequently represented as functional data (in a Banach or infinite-dimensional Hilbert space, see  \cite{Horvth2012} and references therein) and images can be represented as probability measures \citep{feydy2017optimal,DeLaraDiffeo}.

Overall, Theorem \ref{Theo:MainResults} essentially just leaves open the statistical analysis of empirical entropic optimal transport for unbounded cost functions. While it is plausible to expect that similar assertions persist given adequate concentration of the underlying measures, formalizing these claims proves difficult due to the absence of appropriate quantitative limits for the relevant objects of interest (Remark \ref{rmk:unboundedCosts}).

The proof of Theorem \ref{Theo:MainResults} is based on a careful linearization analysis of optimal solutions for the dual entropic optimal transport problem (see Theorem \ref{Theo:PotIntApprox}). We pursue this approach by relying on $L^2$-norms with respect to empirical measures and suitably controlling the arising the approximating errors. This way, we succeed in sidestepping issues from empirical process theory such as ensuring that the class of optimal transport potentials is a Donsker\footnote{In the literature of empirical processes, a class of functions $\mathcal{F}$ is Donsker if the central limit theorem is valid in $\ell^{\infty}(\mathcal{F})$ (see e.g., \citealt{van1996weak}).} class, which would inevitably require certain regularity conditions imposed on the cost function. Notably, this empirical norm was also utilized by \cite{rigollet2022sample} for the derivation of convergence rates~and non-asymptotic concentration bounds for empirical entropic optimal transport quantities. 

\subsection{Summary of applications}

By transitioning from smooth to bounded costs, our assumptions now facilitate asymptotically valid inferences concerning both entropic optimal transport cost and plan in non-Euclidean spaces and general costs. These insights hold particular importance for a variety of applications, such as in computational biology \citep{evans2012phylogenetic,Schiebinger19, tameling2021Colocalization, wang2021optimal} where metric-based costs are commonly employed, transport-based signal analysis incorporating non-smooth signal functions \citep{thorpe2017transportation}, registration invariant object discrimination in metric measure spaces \citep{chowdhury2019gromov,weitkamp2022distribution} utilizing costs associated with the underlying metrics, and general transport-based measures for dependency on metric spaces \citep{nies2021transport, liu2022entropy, wiesel2022measuring}.

As tangible examples of our theoretical framework, we present in Section \ref{sec:Applications} distributional limits for the empirical Sinkhorn cost as initially introduced by \cite{cuturi13}, along with the divergence \citep{genevay2018}, entropic optimal transport maps \citep{seguy2018large}, Sinkhorn kernels for Gaussian processes \citep{bachoc2023gaussian}, and transport-based colocalization curves \citep{klatt2020empirical}.

Overall, our findings emphasize that entropy regularization considerably enhances the statistical complexity and streamlines transport inference in general Polish spaces, paving the way for further research and applications across a range of fields.

\subsection{Related works}\label{subsec:related_work}
Several recent works have studied the statistical behavior of the regularized transport problem. A first contribution was made by \cite{Genevay2018SampleCO} who showed for compactly supported measures and smooth costs that
\begin{equation}\label{eq:sampleCO}
\E\left(\left\vert \rS_{\epsilon}(\mP_n, \mQ_m ) -\rS_{\epsilon}(\mP, \mQ )\right\vert\right) \leq \cK(c,\epsilon)\sqrt{\frac{n+m}{nm}}
\end{equation}
with a constant $\cK(c,\epsilon)$ that depends exponentially on the ratio of the cost and the regularization parameter $\epsilon$. For squared Euclidean costs and sub-Gaussian measures the constant was then improved to a polynomial dependence by \cite{mena2019} and further refined for compactly supported measures by \cite{Chizat2020}. The strategy used in these works is based on the observation that the degree of H\"older-regularity of a function $f(\x)=\int e^{-c(\x,\y)} d\mu(\y)$, where $\mu$ is a positive measure, is the same as that of the cost $c$. Thus, under a smooth cost function, the collection of feasible dual entropic optimal transport potentials is included in the set of $\alpha$-H\"older continuous functions $\mathcal{C}^{\alpha}(\Omega)$ for $\alpha>0$ arbitrarily large. Therefore, upon selecting  $\alpha$ sufficiently large, the class $\mathcal{C}^{\alpha}(\Omega)$ is Donsker (see, e.g., \citealt{van1996weak}), i.e.,
\begin{equation*}
\sup_{f\in \mathcal{C}^{\alpha}(\Omega)} \left\vert \int f d(\mP_n-\mP) \right\vert =\mathcal{O}_{p}(n^{-1/2}).
\end{equation*}
Notably, the arguments underlying \eqref{eq:sampleCO} require a finite-dimensional Euclidean domain and the cost function to be sufficiently smooth. %

For sub-Gaussian measures \cite{mena2019} and squared Euclidean costs also showed that the fluctuations of the empirical regularized cost around its expected value are asymptotically Gaussian, i.e., 
\begin{equation}\label{eq:fluctuations}
    \sqrt{\frac{n\, m}{n+m}}(\rS(\mP_n,{\rm Q}_m)-\E(\rS(\mP_n,{\rm Q}_m)))\convW \mathcal{N}(0, \lambda \operatorname{Var}_{\mP}(f_{\mP,\mQ})+(1-\lambda) \operatorname{Var}_{\mQ}(g_{\mP,\mQ})),
\end{equation}
    for $n,m\to \infty$ and $\frac{m}{n+m}\to \lambda\in (0,1)$. Let us point out that \eqref{eq:sampleCO} does not permit replacement of the expectation $\E(\rS(\mP_n,{\rm Q}_m))$ by its population counterpart $\rS(\mP, \mQ)$ in \eqref{eq:fluctuations}, hindering further statistical inference with the entropic optimal transport cost. 
    
    Initial contributions, asserting the validity of such a replacement  were made by \cite{bigot2019CentralLT} and \cite{klatt2020empirical} for finitely supported measures, and later extended by \cite{hundrieser2021limit} for suitably dominated costs for measures with countable support and sub-Exponential concentration. Their techniques rely on a differentiability analysis of the entropic optimal transport cost and the functional delta method (see, e.g., \citealt{Roemisch04}). 

    Moreover, using a refined analysis of the bias, \cite{del2022improved} proved for sub-Gaussian measures that the bias convergence rate is $n^{-1}$, allowing interchangeability of $\E(\rS(\mP_n,{\rm Q}_m))$ in \eqref{eq:fluctuations} with $\rS(\mP,{\rm Q})$. In parallel, \cite{goldfeld2022statistical} %
    confirmed this result using a unified framework for regularized optimal transport by relying on the functional delta method. Their approach also asserts consistency of the bootstrap and asymptotic efficiency of the empirical plug-in estimator. Notably, both techniques by \cite{del2022improved} and \cite{goldfeld2022statistical} crucially rely on smoothness of the cost function in order for entropic optimal transport potentials to lie in a Donsker class.

In addition to distributional limits for the empirical entropic optimal transport cost, multiple works also analyzed weak limits of the corresponding entropic optimal plan. First contributions to this subject were done by \cite{klatt2020empirical} and \cite{hundrieser2021limit} for a discrete ground space $\Omega$ in the Banach space $\ell^1(\Omega\times \Omega)$. For more general spaces similar results cannot hold since the $\ell^1$-norm is too strong; instead results for continuous settings only consider expectations of single functions with respect to the entropic optimal transport plan. 
The two works \cite{gonzalez2022weak,Goldfeld2022LimitTF} that provide such weak limits start with the study of the empirical  entropic optimal transport potentials, i.e., dual optimizers of \eqref{eq:dual_entrop}. To this end, if one follows the path of empirical processes, it must be ensured that the difference between the empirical and population potential is described in terms of a Hadamard differentiable function of the empirical process in a Donsker class. As a consequence, this approach is not suitable for non-continuous costs.

A paradigm shift in the smoothness assumptions for the cost function is carried out by \cite{rigollet2022sample}, who show for uniformly bounded and measurable costs that empirical entropic optimal transport quantities generically converge with parametric order $n^{-1/2}$. Further, they show the bias of the empirical entropy regularized optimal transport cost to converge at the rate $n^{-1}$, which previously was only known for squared Euclidean costs by \cite{del2022improved}. 
In their analysis, \cite{rigollet2022sample}, similar to \cite{del2022improved}, begin by studying the convergence rates of the potentials. However, they linearize these potentials with respect to the empirical $L^2(\mPn)\times L^2(\mQm)$ norm instead of a fixed norm in a Banach space (H\"older continuous functions in \citealt{del2022improved}). In conjunction with independence of random variables, their approach asserts appropriate upper bounds on the mean absolute deviation for entropic potentials and costs, as well as expectations of functions with respect to entropic plans without reliance on suprema over function classes.

Moreover, let us also point out that \cite{Harchaoui2020} proposed a different estimator, namely $\pi'_{\mPn,\mQm}$, of $\pipq$ based on a regularization procedure inspired by Schr\"odinger’s lazy gas experiment, which enables an explicit expression for the optimizer.  In their work, they also derived a pointwise weak limit for $\sqrt{\frac{n\, m}{n+m}}(\pi'_{\mPn,\mQm}-\pipq)$. The asymptotic variance of their limit matches that in Theorem~\ref{Theo:MainResults}.\ref{item:main_plan} and they conjecture the validity of such a limit for the {\it ``common''} empirical entropy regularized  plan, $\pi_{\mPn,\mQm}$. For smooth costs, \cite{gonzalez2022weak} proved this conjecture. %

  In the context of unregularized optimal transport, these results cannot be obtained in the same manner. Neither using our technique nor using the empirical processes based technique. Generally, there is no closed form to determine the solutions of the dual problem, and the optimization problem cannot be written as a Z-estimation problem, which is the fundamental basis of our reasoning. Furthermore, it is well-known that the regularity of the cost and densities does not imply the regularity of the potentials in general \citep{Ma2005,Loeper2009}, so that the optimization class of the dual formulation cannot be reduced to a Donsker class. Additionally, the transport problem suffers from the curse of dimensionality, and one cannot expect limits with parametric  rate  in dimensions greater than 4 \citep{manole2021sharp}. However, the convergence rate of the fluctuations (difference between empirical cost and its expected value) is parametric, irrespective of the dimension, and they turn out to be asymptotically Gaussian if the potentials are unique up to additive constants (see \citealt{delbarrio2019,delbarrio2021central,Javi2} for weak limits and \citealt{staudt2022uniqueness} for the uniqueness conditions of the potentials).

Due to the aforementioned curse of dimensionality, we cannot replace the expected value of the empirical cost by its population value except in cases where at least one probability is supported on a set of sufficiently low dimension \citep{hundrieser2022unifying, Hundrieser23_OT_estcost,hundrieser2022empirical}, such as the semi-discrete case \citep{del2022semi,sadhu2023limit}, the discrete case \citep{sommerfeld2018,tameling18}, or the well-known univariate problem \citep{Munk98,del1999central,BarrioGine2005,delBarrioGordaliza2019, hundrieser2022statistics}. It is worth pointing out that weak limits for the empirical (unregularized) optimal transport cost are not always centered normal, which is in strict contrast to the entropy regularized setting. Indeed, if the collection of population dual potentials is non-unique, normal limits occur -- an observation which is known for general supremum type functionals \citep{carcamo2020directional}.

\subsection{Notation}\label{notation} The random variables considered in this work are measurable functions on the same underlying  probability  space $(\boldsymbol{\Omega}, \mathcal{F}, \P)$.  Throughout this work, $\Omega$ is a Polish space and probability measures ${\rm P}$ and ${\rm Q}$ are assumed to be Borel measures, which we denote by ${\rm P},{\rm Q}\in \mathcal{P}(\Omega)$. 
For a general Banach space $\mathcal{B}$, the norm is denoted as $\|\cdot\|_{\mathcal{B}}$. The operator norm of an operator $M:\mathcal{B}\to \mathcal{B}$ is denoted as 
$ \norm{M}_{\mathcal{B}}=\sup_{\norm{x}_{\mathcal{B}}\leq 1}\norm{Mx}_{\mathcal{B}.}$
Given a random sequence $\{w_n\}_{n\in}\subset \mathcal{B}$  and a real random sequence $\{a_n\}_{n\in\N}$,  the notation  $w_n=o_p(a_n)$ means that the sequence $w_n/a_n$ tends to $0$ in probability, and  $w_n=\mathcal{O}_p(a_n)$ that $w_n/a_n$ is stochastically bounded. %
 In the proofs of the results, we will use the notation $\cK(a_1, \dots, a_k)$ to indicate a constant that depends exclusively on $a_1, \dots, a_k$ and which may vary from line to line. 
 Further, given positive sequences $\{a_n\}_{n\in \N}, \{b_n\}_{n\in \N}\subset (0,\infty)$ the notation $ a_n\asymp b_n$ indicates that $ a_n/ b_n \to (0, \infty)$ as $n\to \infty$. 
Weak convergence of random variables is denoted by \smash{$\convW$},  convergence in probability by \smash{$\convP$}, and almost sure convergence by \smash{$\convAS$}.
For a Borel measure $\mu$ on $\Omega$, the space $L^2(\mu)$ denotes the vector space of square-integrable functions, while $L^2_0(\mu)$ denotes the subset of functions in $L^2(\mu)$ which admit expectation zero.  Further, we denote by $L^\infty(\Omega)$ the space of uniformly bounded and Borel-measurable functions on $\Omega$. The uniform norm of a function $f:\Omega\to \R$ is denoted by 
$\|f\|_{\infty}=\sup_{\x\in \Omega} \vert f(x)\vert$. Moreover, for ease of notation we write 
$ \norm{\cdot}_n=\norm{\cdot}_{L^2(\mPn)}$, $ \norm{\cdot}_m=\norm{\cdot}_{L^2(\mQn)}$ and $ \norm{\cdot}_{n\times m}=\norm{\cdot}_{L^2(\mPn\times \mQm)}$.
The expected value of a function $f$ with respect to a probability  measure $\mP$ is denoted as 
$ \int f d\mP= \int f(\x) d\mP(\x)=\E_{\X\sim \mP}[f(\X)],$
while its variance as
$ \operatorname{Var}_{\X\sim \mP}[f(\X)].$ Finally,  we write $\log_+(t) \coloneqq \max(1,\log(t))$.

\section{Preliminaries}\label{sec:preliminaries}

In this section we outline structural insights about entropy regularized optimal transport and introduce relevant operators which arise in the formulation of our weak limits. 

\begin{Proposition}[{\citealt[Theorems 4.2 and 4.7, Equation 4.11]{nutz2021introduction}}]\label{prop:structureEOT}
    Let $\mP, \mQ\in \cP(\Omega)$ be probability measures on a Polish space $\Omega$ and consider a uniformly bounded, Borel measurable cost function $c\colon \Omega \times \Omega \rightarrow \RR$. Define the kernel $C(\x,\y) = \exp(-c(\x,\y))$. Then, the following assertions hold.
    \begin{enumerate}%
        \item Strong duality is satisfied in the sense that 
            \begin{multline}    
            \label{eq:dual_entrop}
            \rS(\mP,{\rm Q})= \sup_{\substack{f\in L^\infty(\mP) \\ g\in L^\infty(\mQ)}}\bigg[ \int f(\x) d\mP(\x) + \int g(\y) d\mQ(\y)\\
            -\, \int  C(\x,\y) e^{f(\x)+g(\y)} d\mP(\x) d\mQ(\y) +1\bigg].
            \end{multline}
        \item There exists a unique optimizer $\pi_{\mP, \mQ}\in \Pi(\mP, \mQ)$ for \eqref{kanto_entrop}, referred to as the entropic optimal transport plan, given by \begin{subequations}\label{eq:optimallityCodPlan}
  \begin{align}
      d\pi_{\mP,\mQ}(\x,\y)&=\xi_{\mP,\mQ}(\x,\y) d\mP(\x) d\mQ(\y),\label{eq:optimallityCodPlan_density}
      \intertext{where the density is characterized by}
     \xi_{\mP,\mQ} (\x,\y)&= C(\x,\y) e^{f_{\mP,\mQ}(\x)+g_{\mP,\mQ}(\y)},\label{eq:optimallityCodPlan_representation}
  \end{align}
  \end{subequations}
        for a pair of potentials $(f_{\mP,\mQ},g_{\mP,\mQ})\in L^\infty(\mP)\times L^\infty(\mQ)$, referred to as entropic optimal transport potentials, which solve \eqref{kanto_entrop}.
\item \label{item:optimalityCondition}Any pair $(f_{\mP,\mQ},g_{\mP,\mQ})$ of entropy regularized optimal transport potentials can be written one in terms of each other, 
\begin{equation}
\begin{aligned}
    \label{eq:optimallityCodt0}
       f_{\mP,\mQ}&=-\log\left(\int C(\cdot,\y) e^{g_{\mP,\mQ}(\y)}d\mQ(\y)\right), \quad \mP-\text{a.s.}\\
    g_{\mP,\mQ}&=-\log\left(\int C(\x,\cdot)  e^{f_{\mP,\mQ}(\x)}d\mP(\x)\right), \quad \mQ-\text{a.s.}
\end{aligned}
\end{equation}
These equalities are necessary and sufficient for optimality of potentials, i.e., a~pair of potentials $(f,g)\in L^\infty(\mP)\times L^\infty(\mQ)$ solves \eqref{eq:dual_entrop} if and only if \eqref{eq:optimallityCodt0} is met.
    \end{enumerate}
\end{Proposition}
The optimality conditions in \eqref{eq:optimallityCodt0} also enable a canonical extension of the potentials beyond the supports of $\mP$ and $\mQ$, so that they are well-defined everywhere on $\Omega$.
For our purposes, we will frequently consider the unique pair of potentials
$(f_{\mP}, g_{\mQ})\in L^\infty(\Omega)\times L^\infty(\Omega)$ such that everywhere on $\Omega$ it holds that
\begin{equation}\label{eq:optimalityCriterionNicePotentials}
\begin{aligned}
 f_{\mP,\mQ}=-\log\left(\int C(\cdot,\y) e^{g_{\mP,\mQ}(\y)}d\mQ(\y)\right),\quad 
    g_{\mP,\mQ}&=-\log\left(\int C(\x,\cdot)e^{f_{\mP,\mQ}(\x)}d\mP(\x)\right), \\
    \int g_{\mP,\mQ}(\y) d\mQ(\y) &= 0.
\end{aligned}
\end{equation}
This also uniquely determines the density $\xipq$ everywhere on $\Omega\times \Omega$, even beyond the support of the underlying measures. Moreover, given empirical measures $\mPn, \mQm$ we define the empirically shifted population potentials $(\bfpq, \bgpq)\in L^\infty(\Omega)\times L^\infty(\Omega)$ by  
$$ \bfpq(\x) := f_{\mP,\mQ}(\x) + \int g_{\mP,\mQ}(\y) d\mQn(\y), \quad  \bgpq(\y) := g_{\mP,\mQ}(\y) - \int g_{\mP,\mQ}(\y) d\mQn(\y).$$
Based on Proposition \ref{prop:structureEOT}.\ref{item:optimalityCondition} this still defines a pair of entropic optimal transport potentials which now depends on the empirical measure $\mQn$ and are suitably shifted to satisfy $\int \bgpq(\y) d\mQm(\y)=0$. 
With these conventions at our disposal we formulate a suitable bound for entropic optimal transport potentials in terms of uniform bound for the cost function.

\begin{Lemma}[{\citealt[Theorem 1.2]{marino2020optimal}}]\label{lem:regularity}
   Consider a uniformly bounded cost function $c\in L^\infty(\Omega^2)$, then for any $\mP,\mQ\in 
    \mathcal{P}(\Omega)$ it holds that
    $$\inf_{v\in \R}\{\|f_{\mP,\mQ}-v\|_{\infty}+\| g_{\mP,\mQ}+v\|_{\infty}\}\leq \, \frac{3}{2} \norm{c}_\infty.$$
    In particular, for the pair of optimal potentials $(\bfpq, \bgpq)$ such that $\int \bgpq d \mQm = 0$, it holds (deterministically) that $\max(\|\bfpq\|_\infty, \|\bgpq\|_\infty) \leq 3\norm{c}_\infty$. 
    Further, it holds  that
    $$ \exp(-3\norm{c}_\infty) \leq \xi_{\mP, \mQ}(\x, \y) \leq \exp(3\norm{c}_\infty) \quad \text{ for any } \,\x, \y \in \Omega.$$
\end{Lemma}

For the formulation of the asymptotic variances in our weak limit limits, we need to introduce the following operators 
 \begin{equation}\label{APop}
     \mathcal{A}_{\mP}f = \frac{\int e^{f_{\mP,\mQ}(\x)}C(\x,\cdot)f(\x)d\mP(\x)}{\int e^{f_{\mP,\mQ}(\x)}C(\x,\cdot)d\mP(\x)},
         \quad  \mathcal{A}_{\mQ}g = \frac{\int e^{g_{\mP,\mQ}(\y)}C(\cdot,\y)g(\y)d\mQ(\y)}{\int e^{g_{\mP,\mQ}(\x)}C(\cdot,\y)d\mQ(\y)}
 \end{equation}
 and $\bAP f=\mathcal{A}_{\mP}f-\int \mathcal{A}_{\mP}f d\mQ$. Additionally, we define the centering operators $$ \MP f = f - \int f d \mP, \quad \MQ g = g - \int g d \mQ.$$
 The following result establishes important properties for the operators $\bAP\AQ$ and $\AQ\bAP$ which arise in the asymptotic variance for the weak limit for the (conditional) expectation with respect to the empirical entropic optimal transport plan. The proof is provided in Appendix \ref{app:preliminaries}.

\begin{Lemma}\label{lemma:empiricalInverse}
Consider a uniformly bounded cost function $c\in L^\infty(\Omega^2)$.
Then, the operators $\AQ \bAP:  L^2(\mP)\to  L^2(\mP)$ and $\bAP \AQ:  L^2_0(\mQ)\to  L^2_0(\mQ)$, defined in \eqref{APop}, have the following properties.
    \begin{enumerate}
        \item \label{lem:ApEqualAp0} It holds that $\bAP = \MQ \AP = \AP \MP $ and  $\MP \AQ = \AQ \MQ$. In particular, $$\bAP \AQ = \AP \AQ \MQ \quad \text{ and } \quad \AQ \bAP = \AQ \AP \MP.$$ 
        Further, $\bAP \AQ=\AP\AQ$ on $L^2_0(\mQ)$, and  $\AQ \bAP=\AQ\AP$ on $L^2_0(\mP)$.
        \item \label{1} The eigenvalues of the operators $\mathcal{A}_{\mQ} \bAP$ and $\bAP \mathcal{A}_{\mQ}$ belong to $[0,1]$.
        \item \label{3} The eigenvalues of $\AQ\bAP: L^2(\mP)\to  L^2(\mP) $ and $\bAP\AQ: L^2_0(\mQ)\to  L^2_0(\mQ) $, namely resp. $\lambda_{ \mQ, \mP}$  and $\lambda_{\mP, \mQ}$, are upper bounded by %
        \begin{equation}
            \label{eq:boundeig}
            \lambda_{ \mQ, \mP}, \ \lambda_{\mP, \mQ}\leq \delta(c) \coloneqq 1- \exp(3 \norm{c}_\infty) <1\, . 
        \end{equation}
        \item \label{2} The operators $(I_{L^2(\mP)}-\AQ\bAP)^{-1}: L^2(\mP)\to  L^2(\mP)$ and $(I_{L^2_0(\mQ)}-\bAP\AQ)^{-1}: L^2_0(\mQ)\to  L^2_0(\mQ)$ are well-defined. 
    \end{enumerate}
\end{Lemma}

\begin{Remark}[Unbounded costs]\label{rmk:unboundedCosts}
Our weak limits crucially rely on the uniform boundedness of the cost function, which asserts quantitative uniform bound for the potentials. Extending the weak limits to unbounded costs requires $(i)$ the derivation of quantitative bound for entropic optimal transport potentials as in Lemma \ref{lem:regularity}, but more importantly $(ii)$ a similar kind of eigenvalue gap as in Lemma \ref{lemma:empiricalInverse}. Although it is reasonable to expect these aspects remain valid under sufficient concentration of the underlying measures, its formalization is a challenge and left as future work.
\end{Remark}

\section{Weak Limits}\label{sec:main_results}

In this section we state the distributional limits for different empirical quantities from entropic optimal transport. We consider throughout probability measures $\mP, \mQ\in \cP(\Omega)$ and a bounded, measurable cost function $c\in L^\infty(\Omega\times \Omega)$. We stress that the subsequent results require no smoothness or continuity assumptions of the cost. For independent random variables $\X_1, \dots, \X_n \sim \mP$ and $\Y_1, \dots, \Y_m\sim \mQ$ we define empirical measures $\mPn \coloneqq \frac{1}{n} \sum_{i = 1}^{n} \delta_{\X_i}$ and $\mQm \coloneqq \frac{1}{m} \sum_{j= 1}^{n} \delta_{\Y_j}$. All subsequent asymptotic %
two-sample results are stated for $n,m\rightarrow \infty$ with $\frac{m}{n+m}\rightarrow \lambda \in (0,1)$. %

\subsection{Entropic optimal transport potentials}
A central element for our subsequent weak limits is the following approximation result for the integral of a bounded function $\eta \in L^\infty(\Omega\times \Omega)$ multiplied by an entropic optimal transport potential. %
 Before stating the result we introduce the notation
\begin{align}\label{eq:DefEtaX}
    &\eta_{\x}:\x\to \int \eta(\x,\y)\xi_{\mP,\mQ}(\x,\y)d\mQ(\y)\ \ \text{and}\ \   \eta_{\y}:\y\to \int \eta(\x,\y)\xi_{\mP,\mQ}(\x,\y)d\mP(\x),
\end{align}

\begin{Theorem}\label{Theo:PotIntApprox}
For a bounded function $\eta \in L^\infty(\Omega\times \Omega)$ the entropic optimal transport potentials $(\bfmn, \bgmn), (\bfpq, \bgpq)$ such that $\int \bgmn d\mPn =\int \bgpq d\mPn =0$   satisfy asymptotically  
\begin{align*}
    \int \eta(\x,\y)&(\bfmn(\x)  - \bfpq(\x))\xipq(\x,\y) d(\mPn \otimes \mQm)(\x,\y)\\
    =& \int \Big((I_{L^2_0(\mP)}-\AQ\AP)^{-1}\big(\AQ\AP \eta_{\x}-\E_{\X\sim \mP}[\eta_{\x}(\X)]\big)\Big) d (\mPn - \mP) \\
&-\int \Big((I_{L^2_0(\mQ)}-\AP\AQ)^{-1}\big(\AP \eta_{\x}-\E_{\X\sim \mP}[\eta_{\x}(\X)]\big)\Big) d (\mQm - \mQ)+ \op\left(\sqrt{\frac{n+m}{nm}}\right),
    \\
     \int \eta(\x,\y)&(\bgmn(\y)  - \bgpq(\y))\xipq(\x,\y) d(\mPn \otimes \mQm)(\x,\y)\\
    =&\int \Big((I_{L^2_0(\mP)}-\AP\AQ)^{-1}\big(\AP\AQ \eta_{\y}-\E_{\Y\sim \mQ}[\eta_{\y}(\Y)]\big)\Big) d (\mQm - \mQ)\\
    &-\int \Big((I_{L^2_0(\mQ)}-\AQ\AP)^{-1}\big(\AQ \eta_{\y}-\E_{\Y\sim \mQ}[\eta_{\y}(\Y)]\big)\Big) d (\mPn - \mP)+ \op\left(\sqrt{\frac{n+m}{nm}}\right).
\end{align*}

\end{Theorem}

The proof is deferred to Section \ref{sec:proofs}. 
It is worth noting that the operators that appear on the right-hand side of the asymptotic equality are well-defined (in the sense of the existence of the inverse) thanks to Lemma~\ref{lemma:empiricalInverse}. More precisely, Lemma~\ref{lemma:empiricalInverse}.\ref{2} implies the existence and continuity of $(I_{L^2_0(\mP)}-\AQ\bAP)^{-1}:L^2_0(\mP)\to  L^2_0(\mP)$ and $(I_{L^2_0(\mQ)}-\bAP\AQ)^{-1}: L^2_0(\mQ)\to  L^2_0(\mQ)$.  Lemma~\ref{lemma:empiricalInverse}.\ref{lem:ApEqualAp0}  states that $\bAP\AQ=\AP\AQ$ in $L^2_0(\mQ)$ and $\AQ\bAP=\AQ\AP$ in $L^2_0(\mP)$ so that $(I_{L^2_0(\mP)}-\AQ\AP)^{-1}:L^2_0(\mP)\to  L^2_0(\mP)$ and $(I_{L^2_0(\mQ)}-\AP\AQ)^{-1}: L^2_0(\mQ)\to  L^2_0(\mQ)$ are continuous operators.

By choosing a function $\eta$ that is constant in one component we also obtain an approximation result for the integral with respect to only one measure. The proof of the following result is relegated to Appendix \ref{app:main}.

\begin{Corollary}\label{cor:PotIntApprox}
For a bounded function $\tilde \eta \in L^\infty(\Omega)$ the entropic optimal transport potentials as in Theorem \ref{Theo:PotIntApprox}  satisfy asymptotically 
\begin{subequations}
\begin{align}
     &\int \tilde \eta  (\bfmn  - \bfpq)d\mPn\label{stat1CoroApproxPoint}
   \\ = &\int \Big((I_{L^2_0(\mP)}-\AQ\AP)^{-1}\big(\AQ\AP \tilde\eta-\E_{\X\sim \mP}[\tilde\eta(\X)]\big)\Big) d (\mPn - \mP)\notag \\
&-\int \Big((I_{L^2_0(\mQ)}-\AP\AQ)^{-1}\big(\AP \tilde\eta-\E_{\X\sim \mP}[\tilde\eta(\X)]\big)\Big) d (\mQm - \mQ)+ \op\left(\sqrt{\frac{n+m}{nm}}\right),\notag\\
      &\int \tilde \eta  (\bgmn  - \bgpq)d\mPn\label{stat2CoroApproxPoint}
    \\=&\int \Big((I_{L^2_0(\mQ)}-\AP\AQ)^{-1}\big(\AP\AQ \tilde\eta-\E_{\Y\sim \mQ}[\tilde\eta(\Y)]\big)\Big) d (\mQm - \mQ)\notag\\
    &-\int \Big((I_{L^2_0(\mP)}-\AQ\AP)^{-1}\big(\AQ \tilde\eta-\E_{\Y\sim \mQ}[\tilde\eta(\Y)]\big)\Big) d (\mPn - \mP)+ \op\left(\sqrt{\frac{n+m}{nm}}\right).\notag
\end{align}
\end{subequations}
As a consequence, if $\tilde \eta$ is constant, 
it holds that 
$$ \bigg\vert  \int \tilde \eta(\bfmn  - \bfpq)d\mPn\bigg\vert+\bigg\vert\int \tilde \eta(\bgmn  - \bgpq)d\mQm\bigg\vert=\op\left(\sqrt{\frac{n+m}{nm}}\right).$$
\end{Corollary}

\subsection{Entropic optimal transport cost}
Based on our characterization of the asymptotic fluctuation for empirical potentials, we now derive  our main result for the empirical entropic optimal transport cost, whose proof we detail at the end of the section. 

\begin{Theorem} \label{Theo:TCLcost} The empirical entropic optimal transport cost satisfies asymptotically
    \begin{align*} \rS(\mPn, \mQm) - \rS(\mP, \mQ)=  \int f_{\mP,\mQ} d(\mP_n-\mP) + \int g_{\mP,\mQ} d(\mQ_m-\mQ) +\op\left(\sqrt{\frac{n+m}{nm}}\right).
    \end{align*}
\end{Theorem}
The univariate central limit theorem gives the weak limit of Theorem~\ref{Theo:MainResults}.\ref{item:main_costs}. The variance of such a limit is 
${\sigma}^2_{\lambda}$ which is not distribution free, i.e., the variance of the Gaussian limit depends on the population probability measures $\mP$ and $\mQ$. To provide greater practical usefulness to the obtained result, we propose (as in \citealt{del2022improved}) for  ${\sigma}^2_{\lambda}$ the estimator 
\begin{align*}
    \label{estimated_variance2}
\hat{\sigma}_{n,m}^2:=%
\frac{m}{n+m}\operatorname{\Var}_{\X\sim \mP}[f_{\mP_n,\mQ_m}(\X)] + \frac{n}{n+m}\operatorname{\Var}_{\Y\sim \mQ}[g_{\mP_n,\mQ_m}(\Y)].
\end{align*}
The following result shows the consistency of $\hat{\sigma}_{n,m}^2$. The proof is given in Appendix \ref{app:main}.

\begin{Proposition}\label{prop:TCLcost}
    The proposed variance estimator satisfies asymptotically
    $$ \hat{\sigma}_{n,m}^2 \xrightarrow{\;\;p\;\;} {\sigma}_{\lambda}^2. $$
    As a consequence, provided that ${\sigma}_{\lambda}^2>0$, it follows that
   \begin{equation*}
    \sqrt{\frac{nm}{n+m}}\frac{\rS(\mPn,\mQm)-\rS(\mP,\mQ)}{\hat{\sigma}_{n,m}}\convW \cN(0,1).
\end{equation*}
\end{Proposition}

\subsection{Expectation under entropic optimal transport plan} 
We now focus on weak limits for the expectation of a bounded function $\eta \in L^\infty(\Omega\times \Omega)$ with respect to the entropic optimal transport plan, whose proof is stated at the end of the section.

\begin{Theorem}\label{Theo:TCLPlan} The expectation of $\eta\in L^\infty(\Omega\times 
\Omega)$ with respect to the entropic optimal transport plan satisfies asymptotically
\begin{align*}
    \E_{(\X,\Y)\sim \pinm}&[\eta(\X,\Y)] - \E_{(\X,\Y)\sim \pipq}[\eta(\X,\Y)] \\
    =&{\frac{1}{{n}}}\sum_{i=1}^n (I_{L^2_0(\mP)}-\AQ\AP)^{-1}\big(\eta_{\x}-\AQ\eta_{\y}\big)(\X_i)\\
&+{ \frac{1}{{m}}}{\sum_{j=1}^m (I_{L^2_0(\mQ)}-\AP\AQ)^{-1}\big(\eta_{\y}-\AP\eta_{\x}\big)(\Y_j)}
    +o_p\left(\sqrt{\frac{n+ m}{n\,m}}\right).
\end{align*}
\end{Theorem}

Recall from Lemma \ref{lemma:empiricalInverse} that the operators $(1 - \AQ\bAP)^{-1}$ and $(1 - \bAP\AQ)^{-1}$ are well-defined due to the fact that $(\eta_{\x}-\AQ\eta_{\y})(\X_i)$ and $(\eta_{\y}-\AP\eta_{\x})(\Y_j)$ are centered for all $i,j$. We prove the claim only for the first one and $i=1$, i.e., $\E[(\eta_{\x}-\AQ\eta_{\y})(\X_1)]=0$. To do so, set $u(\Y_2)=\E[\xipq(\X_2,\Y_2)\eta(\X_2,\Y_2)\vert \Y_2] $ and note by $\pipq\in \Pi(\mP,\mQ)$ that
\begin{equation}\label{eq:ExpectationExchange}
\begin{aligned}
    \E[(\AQ\eta_{\y})(\X_1)]&=\E[\xipq(\X_1,\Y_2)u(\Y_2)]=\E[u(\Y_2)]=\E[\xipq(\X_2,\Y_2)\eta(\X_2,\Y_2)]\\
    &=\E[\E[\xipq(\X_2,\Y_2)\eta(\X_2,\Y_2)\vert \X_2]]=\E[\eta_{\x}(\X_2)]=\E[\eta_{\x}(\X_1)],
\end{aligned}
\end{equation}
which proves the claim. 

Further, from Theorem~\ref{Theo:TCLPlan}, we conclude  Theorem~\ref{Theo:MainResults}.\ref{item:main_plan}, where the variance of the limit is 
\begin{equation}\label{eq:VarianceEvalPlan}
\begin{aligned}
   \sigma^2_{\lambda}(\eta)\coloneqq & \lambda \operatorname{Var}_{\X\sim \mP}\left[ (I_{L^2_0(\mP)}-\AQ\AP)^{-1}\big(\eta_{\x}-\AQ\eta_{\y}\big)(\X)\right]\\
  &+(1-\lambda )\operatorname{Var}_{\Y\sim \mQ}\left[ (I_{L^2_0(\mQ)}-\AP\AQ)^{-1}\big(\eta_{\y}-\AP\eta_{\x}\big)(\Y)\right].
\end{aligned}
\end{equation} 
It is worth noting that \cite{Harchaoui2020}, using regularization based on the Schr\"odinger bridge, also obtained the same result and conjectured the truth of Theorem~\ref{Theo:TCLPlan}. In \cite{gonzalez2022weak} the conjecture was proven for smooth costs and compactly supported $\mP$ and $\mQ$. With similar assumptions,   \cite{Goldfeld2022LimitTF}  obtained Gaussian limits but with unspecified variances.   In contrast, our result establishes the veracity of the conjecture for bounded costs without any assumptions on the supports or smoothness. 

As an estimator for the asymptotic variance $\sigma^2_\lambda(\eta)$ we define for $N\in \NN\cup\{\infty\}$,
\begin{align*}
 \hat\sigma^2_{n,m,N}(\eta)\coloneqq &\frac{m}{n+m}\operatorname{Var}_{\X\sim \mPn} \left[\sum_{k = 0}^N(\AQn\APn)^{k}\big(\hat\eta_{\x,n,m}-\AQn\hat\eta_{\y,n,m}\big)(\X)\right]\\
  &+ \frac{n}{n+m} \operatorname{Var}_{\Y\sim \mQm} \left[ \sum_{k = 0}^N(\APn\AQm)^{k}\big(\hat\eta_{\y,n,m}-\APn\hat \eta_{\x,n,m}\big)(\Y)\right],
\end{align*}
where $(\hat\eta_{\x,n,m}, \hat\eta_{\y,n,m})$ denote the empirical counterparts to $(\eta_{\x},\eta_{\y})$, defined by
\begin{align*}
    &\hat\eta_{\x,n,m}:\x\!\to\!\int \!\eta(\x,\y)\xi_{\mPn,\mQm}(\x,\y)d\mQn(\y)\ \ \text{and}\ \   \hat\eta_{\y,n,m}:\y\!\to\!\int \!\eta(\x,\y)\xi_{\mPn,\mQn}(\x,\y)d\mPn(\x).
\end{align*}
The following result shows consistency of this estimator and is proven in Appendix \ref{app:main}.
\begin{Proposition}\label{Proposition:VariancePlans}
The proposed variance estimator satisfies asymptotically for any choice of $N = N(n,m)\in \NN\cup\{\infty\}$ with $N\rightarrow \infty$ for $n,m\to \infty$ and $m/(n+m)\to\lambda$ that
    $$  \hat\sigma^2_{n,m,N}(\eta) \xrightarrow{\;\;p\;\;}  \sigma^2_{\lambda}(\eta). $$
    As a consequence, provided that $\sigma^2_{\lambda}(\eta)>0$, it follows  that
    $$ \sqrt{\frac{nm}{n+m}}\frac{\E_{(\X,\Y)\sim \pinm}[\eta(\X,\Y)] - \E_{(\X,\Y)\sim \pipq}[\eta(\X,\Y)]}{ \hat\sigma_{n,m,N}(\eta)} \convW \cN(0,1).$$
\end{Proposition}

\subsection{Conditional expectation under entropic optimal transport plan}
We also derive the weak limit for the conditional expectation of a bounded function $\eta \in L^\infty(\Omega\times \Omega)$ under the entropic optimal transport plan. The proof is given at the end of the section.

\begin{Theorem}\label{Theo:TCLCondPlan} The conditional expectation of $\eta\in L^\infty(\Omega\times 
\Omega)$ with respect to the entropic optimal transport plan given that $\X = \x\in\Omega$ satisfies asymptotically
\begin{align*}
    \E_{(\X,\Y)\sim \pinm}&[\eta(\X,\Y) | \X = \x] - \E_{(\X,\Y)\sim \pipq}[\eta(\X,\Y)|\X = \x] \\
    &\!\!\!\!\!\!\!\!\!\!\!\!\!\!\!\!\!\!\!\!\!=\frac{1}{n}\sum_{i = 1}^n\Big((I_{L^2_0(\mP)}-\AQ\AP)^{-1}\AQ\big\{ [\eta(\x,\cdot)-\eta_{\x}(\x)] \xipq(\x,\cdot) \big\} \Big)(\X_i)\\
    &\!\!\!\!\!\!\!\!\!\!\!\!\!\!\!+\frac{1}{m}\sum_{j = 1}^m\Big((I_{L^2_0(\mQ)}-\AP\AQ)^{-1}\big\{ [\eta(\x,\cdot)-\eta_{\x}(\x)] \xipq(\x,\cdot) \big\} \Big)(\Y_j)  + \opnm .
\end{align*}
\end{Theorem}
Note  by Lemma \ref{lemma:empiricalInverse} that the right-hand side is well-defined  since $[\eta(\x, \Y) - \eta_\x(\x)]\xipq(\x,\Y)$ is a centered random variable, which is a consequence of \eqref{eq:optimallityCodt0}. Moreover, from Theorem \ref{Theo:TCLCondPlan} we conclude Theorem \ref{Theo:MainResults}.\ref{item:main_condplan}, where variance of the limit is given by
\begin{equation}\label{eq:VarianceEvalCondPlan}
\begin{aligned}
   \sigma^2_{\lambda}(\eta, \x)\coloneqq&
\lambda\operatorname{Var}_{\X\sim \mP}\left[ (I_{L^2_0(\mP)}-\AQ\AP)^{-1}\AQ\big\{ [\eta(\x,\cdot)-\eta_{\x}(\x)] \xipq(\x,\cdot) \big\} (\X)\right]\!\!\!\!\!\!\!\!\!\!\\
  &+(1-\lambda )\operatorname{Var}_{\Y\sim \mQ}\left[ (I_{L^2_0(\mQ)}-\AP\AQ)^{-1}\big\{ [\eta(\x,\cdot)-\eta_{\x}(\x)] \xipq(\x,\cdot) \big\} (\Y)\right].\!\!\!\!\!\!\!\!\!\!
\end{aligned}
\end{equation}
As an estimator for the asymptotic variance $\sigma^2_{\lambda}(\eta, \x)$ we consider for $N\in \NN\cup\{\infty\}$,
\begin{align*}
    \hat\sigma^2_{n,m,N}(\eta, \x)\coloneqq& 
    \frac{m}{n+m}\operatorname{Var}_{\X\sim \mPn}\left[\sum_{k = 0}^N (\AQm\APn)^k\AQm\big\{ [\eta(\x,\cdot)-\hat\eta_{\x,n,m}(\x)] \ximn(\x,\cdot) \big\} (\X)\right]\\
    &+\frac{n}{n+m}\operatorname{Var}_{\Y\sim \mQm}\left[\sum_{k = 0}^N (\APn\AQm)^k\big\{ [\eta(\x,\cdot)-\hat\eta_{\x,n,m}(\x)] \ximn(\x,\cdot) \big\} (\Y)\right],
\end{align*}
where $\hat\eta_{\x,n,m}$ is defined just as before. The subsequent result asserts consistency of the variance estimator for a suitable choice of $N$. Its proof is given in Appendix \ref{app:main}.

\begin{Proposition}\label{Proposition:VarianceCondPlans}
The proposed variance estimator satisfies asymptotically for any choice of $N = N(n,m)\in \NN\cup\{\infty\}$ with $N\rightarrow \infty$ for $n,m\to \infty$ and $m/(n+m)\to\lambda$ that
     $$  \hat\sigma^2_{n,m,N}(\eta, \x) \xrightarrow{\;\;p\;\;}  \sigma^2_{\lambda}(\eta, \x). $$
    As a consequence, provided that $\sigma^2_{\lambda}(\eta, \x)>0$, it follows  that
    $$ \sqrt{\frac{nm}{n+m}}\frac{\E_{(\X,\Y)\sim \pinm}[\eta(\X,\Y)|\X = \x] - \E_{(\X,\Y)\sim \pipq}[\eta(\X,\Y)|\X = \x]}{ \hat\sigma_{n,m,N}(\eta, \x)} \convW \cN(0,1).$$
\end{Proposition}

\subsection{Proofs for weak limits}
We now provide the proofs for Theorems \ref{Theo:TCLcost}, \ref{Theo:TCLPlan} and \ref{Theo:TCLCondPlan}. All of them crucially rely on the linearization result for the empirical entropic optimal transport potentials. 
\begin{proof}[Proof of Theorem \ref{Theo:TCLcost}]
By the dual formulation \eqref{eq:dual_entrop} and invariance of the dual objective under shifts of the form $(-1, 1)$ it holds 
\begin{align*}
     \rS(\mP_n,{\rm Q}_m)-\rS(\mP,{\rm Q})=& \int\bfmn  d\mP_n + \int \bgmn d\mQ_m- \int \bfpq  d\mP - \int \bgpq d\mQ\\
    = &\int (\bfmn -  \bfpq  )d\mP_n + \int (\bgmn - \bgpq) d\mQ_m\\
 &+ \int \bfpq  d(\mPn-\mP) + \int \bgpq d(\mQm-\mQ).
\end{align*}
Invoking Corollary \ref{cor:PotIntApprox} for $\tilde \eta \equiv 1$ asserts that $\int (\bfmn -  \bfpq  )d\mP_n + \int (\bgmn - \bgpq) d\mQ_m = \op\left(\sqrt{\frac{n+m}{nm}}\right)$, which proves the claim.  
\end{proof}

\begin{proof}[Proof of Theorem \ref{Theo:TCLPlan}]
We perform the following decomposition 
\begin{align*}
    \int \eta d \Big( \pinm - \pipq \Big)   = \int \eta (\xi_{\mP_n,\mQ_m}-\xi_{\mP,\mQ})d {\mP_n \otimes \mQ_m}+ \int \eta \xi_{\mP,\mQ}d ({\mP_n \otimes \mQ_m}-{\mP\otimes \mQ})
\end{align*}
We denote the first expression on the right-hand side by $A_{n,m}$ and the second by $B_{n,m}$.
     
For the analysis of $A_{n,m}$ we define $h_{n,m}=f_{\mP_n,\mQ_m}+g_{\mP_n,\mQ_m}$ and $h=f_{\mP,\mQ}+g_{\mP,\mQ}$. Then we note that the inequality $\vert e^x-e^y-e^y(x-y)\vert \leq e^{x+y}\vert x-y\vert^2$ for $x,y\in \RR$ implies 
\begin{align*}
   &\| \xi_{\mP_n,\mQ_m}-\xi_{\mP,\mQ}- \xi_{\mP,\mQ}(h_{n,m}-h)\|_{n\times m}\\
\leq  & e^{\|h_{n,m}\|_{\infty}+\|h\|_{\infty}}\|h_{n,m}-h \|_{n\times m}^2=\opnm,
\end{align*}
where the last equality is consequence of Lemma \ref{lem:regularity} and \cite[Inequality (4.3) and Lemma 14]{rigollet2022sample}. Then $A_{n,m}= A_{n,m}'+o_p(\sqrt{(n+m)/nm})$, where 
   \begin{align*}
    A_{n,m}'=\int \eta \xi_{\mP,\mQ}\,\left( f_{\mP_n,\mQ_m}+g_{\mP_n,\mQ_m}-f_{\mP,\mQ}-g_{\mP,\mQ}\right)  d {\mP_n \mQ_m}.
\end{align*}
Invoking Theorem \ref{Theo:PotIntApprox} for the function $\eta $ it follows for $\eta_{\x}$ and $\eta_{\y}$ from \eqref{eq:DefEtaX} that 
\begin{align*}
 A_{n,m}'=& \int \Big((I_{L^2_0(\mP)}-\AQ\AP)^{-1}\big(\AQ\AP \eta_{\x}-\AQ\eta_{\y}-\E_{\X\sim \mP}[\eta_{\x}(\X)-\eta_{\y}(\Y)]\big)\Big) d (\mPn - \mP) \\
&+\int \Big((I_{L^2_0(\mQ)}-\AP\AQ)^{-1}\big(\AP\AQ \eta_{\y}-\AP\eta_{\x}-\E_{\Y\sim \mQ}[\eta_{\y}(\Y)-\eta_{\x}(\X)]\Big) d (\mQm - \mQ)
\end{align*}
holds up to additive $\op\left(\sqrt{\frac{n+m}{nm}}\right)$ terms. Since $\E_{\X\sim \mP}[\eta_{\x}(\X)-\eta_{\y}(\Y)]=0$ (see relation \eqref{eq:ExpectationExchange}), we can rewrite the previous asymptotic equality as 
\begin{align*}
 A_{n,m}'=& \int \Big((I_{L^2_0(\mP)}-\AQ\AP)^{-1}\big(\AQ\AP \eta_{\x}-\AQ\eta_{\y}\Big) d (\mPn - \mP) \\
&+\int \Big((I_{L^2_0(\mQ)}-\AP\AQ)^{-1}\big(\AP\AQ \eta_{\y}-\AP\eta_{\x}\Big) d (\mQm - \mQ)+\op\left(\sqrt{\frac{n+m}{nm}}\right).
\end{align*}

Moreover, using the theory of $U$-statistics \cite[Theorem 12.6]{vaart_1998} it follows for $B_{n,m}$ that 
\begin{align*}B_{n,m}&=\int \eta \xi_{\mP,\mQ} d ((\mP_n - \mP) \mQ + \mP(\mQ_m - \mQ)) + o_p\left(\sqrt{\frac{n+m}{nm}}\right)\\
&=\int \eta_{\x} d (\mP_n - \mP)  + \int \eta_{\y}  d(\mQ_m - \mQ) + \op\left(\sqrt{\frac{n+m}{nm}}\right).
\end{align*}
By the same computation as from page 18 of \cite{gonzalez2022weak} it follows that %
\begin{align*}
    A_{n,m}' + B_{n,m} &= \int \Big((I_{L^2_0(\mP)}-\AQ\AP)^{-1}\big( \eta_{\x}-\AQ \eta_{\y}\big)\Big) d (\mPn - \mP) \\
     &+ \int \Big((I_{L^2_0(\mQ)}-\AP\AQ)^{-1}\big( \eta_{\y}-\AP \eta_{\x}\big)\Big) d (\mQm - \mQ) + \op\left(\sqrt{\frac{n+m}{nm}}\right).\qedhere
\end{align*}
\end{proof}

For the proof of Theorem \ref{Theo:TCLCondPlan} we will rely on the following result on pointwise convergence of empirical potentials whose proof is deferred to Appendix \ref{app:main}.

\begin{Lemma}\label{lem:pointwiseConvergence}
    Fix $(\x, \y)\in \Omega^2$. Then, it holds for $n,m \rightarrow \infty$ with $\frac{m}{n+m} \rightarrow \lambda \in (0,1)$, 
     $$ \vert \bfmn(\x)- \bfpq(\x)\vert +\vert  \bgmn(\y)-\bgpq(\y)\vert =\Op\left(\sqrt{\frac{n+m}{n\, m}}\right).$$
    In particular, it follows that
    $$ \bfmn(\x)\convP \bfpq(\x) \quad \text{ and } \quad  \bgmn(\y)\convP \bgpq(\y).$$
\end{Lemma}

\begin{proof}[Proof of Theorem \ref{Theo:TCLCondPlan}]
   For $\x\in \Omega$ fixed we perform the decomposition 
    \begin{align*}
        & \int \eta(\x,\y) \Big( \ximn(\x, \y) d \mQm(\y)  - \xipq(\x,\y)d\mQ(\y) \Big) \\
        =  &  \int \eta(\x,\y) \Big( \ximn(\x, \y) - \xipq(\x,\y)\Big)d\mQm(\y) +  \int \eta(\x,\y) \xipq (\x,\y) d \Big(\mQm(\y)  - d\mQ(\y) \Big).
    \end{align*}
    We denote the first time by $A_{n,m}$ and the latter term by $B_{n,m}$. 
    For the analysis of the term $A_{n,m}$ recall for any $\x\in \Omega$ by the optimality condition \eqref{eq:optimallityCodt0} that 
    \begin{equation*}
    \begin{aligned}
    \exp(\bfpq(\x)) &= \left(\int C(\x,\y) \exp(\bgpq(\y)) d \mQ(\y)\right)^{-1},\\
    \exp(\bfmn(\x)) &= \left(\int C(\x,\y) \exp(\bgmn(\y)) d \mQm(\y)\right)^{-1}.
\end{aligned}
\end{equation*}
Hence, for any $(\x, \y)\in \Omega\times \Omega$ it follows that 
\begin{align*}
 & (\ximn(\x, \y) - \xipq(\x,\y))\\
=&\Big( \exp(\bfmn(\x) + \bgmn(\y)) - \exp(\bfpq(\x) + \bgpq(\y))  \Big)C(\x,\y) \\
=&  \exp(\bfmn(\x))  \Big(  \exp(\bgmn(\y)) - \exp(\bgpq(\y))\Big)C(\x,\y)\\
& - \left(\frac{\int C(\x,\y') \exp(\bgmn(\y')) d \mQm(\y') - \int C(\x,\y') \exp(\bgpq(\y')) d \mQ(\y')}{\int C(\x,\y') \exp(\bgmn(\y')) d \mQm(\y') \int C(\x,\y') \exp(\bgpq(\y')) d \mQ(\y')} \right)\\
& \quad \quad \times  \exp\Big( \bgpq(\y)\Big)C(\x,\y).
\end{align*}
The second summand on the right-hand side can be rewritten as 
\begin{align*}
& - \Bigg[\left(\frac{\int C(\x,\y') \big(\exp(\bgmn(\y')) -  \exp(\bgpq(\y'))  \big) d \mQm(\y') }{\int C(\x,\y') \exp(\bgmn(\y')) d \mQm(\y') \int C(\x,\y') \exp(\bgpq(\y')) d \mQ(\y')} \right) \\
&{\color{white}-\Bigg[}+ \left(\frac{\int C(\x,\y') \exp(\bgpq(\y')) d (\mQm - \mQ)(\y')}{\int C(\x,\y') \exp(\bgmn(\y')) d \mQm(\y') \int C(\x,\y') \exp(\bgpq(\y')) d \mQ(\y')} \right)\Bigg] \\
& \quad \quad \times \exp\Big( \bgpq(\y) \Big)C(\x,\y)\\
=& - \Bigg[\left(\int C(\x,\y') \big(\exp(\bgmn(\y')) -  \exp(\bgpq(\y'))  \big) d \mQm(\y')  \right)\\
&  {\color{white}-\Bigg[}+\left(\int C(\x,\y') \exp(\bgpq(\y')) d (\mQm - \mQ)(\y')  \right)\Bigg]\exp\Big( \bfmn(\x) \Big)\xipq(\x,\y).
\end{align*}
Invoking Lemma \ref{lem:pointwiseConvergence} it follows that  $\bfmn(\x)\convP \bfpq(\x)$. Moreover, by our quantitative bounds $\norm{\bgmn}_\infty, \norm{\bgpq}_\infty \leq \cK(c)$ it follows using the inequality $|e^x-e^y - e^y(x-y)|\leq e^{x+ y} (x-y)^2$ for $x,y\in \RR$, in conjunction with  Slutsky's Lemma that
    \begin{align*}
     (\ximn(\x, \y)& - \xipq(\x,\y)) =  \xipq(\x,\y) \Big(  \bgmn(\y) - \bgpq(\y)\Big)  \\
&-  \Big(\int \xipq(\x,\y') \left\{\bgmn(\y') - \bgpq(\y') \right\} d \mQm(\y')  \Big) \xipq(\x,\y)\\
&- \Big(\int \xipq(\x,\y') d (\mQm - \mQ)(\y')   \Big)\xipq(\x,\y) \\
&+ \cK(c)[\bgmn(\y) - \bgpq(\y)]^2 + \op(\norm{\bgmn - \bgpq}_{m}).
    \end{align*}
We thus conclude that $A_{n,m}$ equals
\begin{align*}
    & \int \eta(\x,\y)  (\ximn(\x, \y)- \xipq(\x,\y)) d\mQm(\y)\\
    =& \int \eta(\x,\y) \xipq(\x,\y) \Big(  \bgmn(\y) - \bgpq(\y)\Big) d\mQm(\y)\\
    &- \int \xipq(\x,\y')  \Big(  \bgmn(\y') - \bgpq(\y')\Big) d\mQm(\y') \int \eta(\x,\y) \xi(\x,\y) d \mQm(\y)\\
    &-\int \xipq(\x,\y')  d (\mQm - \mQ)(\y')  \int \eta(\x,\y) \xi(\x,\y) d \mQm(\y) + \op\left(\sqrt{\frac{n+m}{nm}}\right).
\end{align*}
By the strong law of large numbers it follows that  $\int \eta(\x,\y) \xi(\x,\y) d \mQm(\y) \convAS \eta_{\x}(\x)$ for $\eta_{\x}$ defined in \eqref{eq:DefEtaX}. Further, note by \eqref{eq:optimallityCodt0} that  \begin{align*}
    \E_{\Y\sim \mQ}\left[(\eta(\x, \Y) - \eta_\x(\x))\xipq(\x,\Y)\right] = 0 . %
\end{align*}
Hence, invoking Corollary \ref{cor:PotIntApprox} it follows that $A_{n,m} + B_{n,m}$ can be represented as 
\begin{align*}
    & \int \Big((I_{L_0^2(\mQ)} -\AP\AQ)^{-1}\AP\AQ\big\{ [\eta(\x,\cdot)-\eta_{\x}(\x)] \xipq(\x,\cdot) \big\} \Big) d (\mQm - \mQ)\\
    &-\int\Big((I_{L_0^2(\mP)}-\AQ\AP)^{-1}\AQ\big\{ [\eta(\x,\cdot)-\eta_{\x}(\x)] \xipq(\x,\cdot) \big\}\Big) d (\mPn - \mP)\\
    &+\int  [\eta(\x,\cdot)-\eta_{\x}(\x)] \xipq(\x,\cdot) d (\mQm - \mQ) +
   \op\left(\sqrt{\frac{n+m}{nm}}\right)\\
   =& \int \Big((I_{L_0^2(\mQ)}-\AP\AQ)^{-1}\big\{ [\eta(\x,\cdot)-\eta_{\x}(\x)] \xipq(\x,\cdot) \big\} \Big) d (\mQm - \mQ) \\
    &+\int \Big((I_{L_0^2(\mP)}-\AQ\AP)^{-1}\AQ\big\{ [\eta(\x,\cdot)-\eta_{\x}(\x)] \xipq(\x,\cdot) \big\} \Big) d (\mPn - \mP) +  \opnm ,
\end{align*}
where the equality is a consequence of an analogous computation as on page 18 of \cite{gonzalez2022weak}. %
\end{proof}

\section{Applications}\label{sec:Applications}

As a consequence of the weak limits from previous section we highlight various implications of our general theory. In all those applications we tacitly assume the ground space $\Omega$ to be Polish and the cost function $c$ to be uniformly bounded. Further, our results are always to be understood for $n,m \rightarrow \infty$ with $\frac{m}{n+m}\to \lambda\in (0,1)$.

\subsection{Sinkhorn costs}
In his original work \cite{cuturi13} proposed the use of the Sinkhorn cost, defined by
\begin{equation*}
\mathrm{d_S}(\mP,\mQ)\coloneqq\E_{(\X,\Y)\sim \pi_{\mP,\mQ}}\left[c(\X,\Y)\right],
\end{equation*}
    which arises by  selecting the evaluation function $\eta\in L^\infty(\Omega\times \Omega)$ as the cost function $c$ itself. Theorem~\ref{Theo:TCLPlan} now asserts a general asymptotically normal behavior which was previously only known for discrete settings \citep{klatt2020empirical, hundrieser2021limit} and for smooth costs \citep{gonzalez2022weak, Goldfeld2022LimitTF}.  
    \begin{Corollary}\label{Coro:SinkhCost}
    The Sinkhorn cost satisfies asymptotically
        \begin{equation*}
    \sqrt{\frac{n\, m}{n+m}}\left( \mathrm{d_S}(\mP_n,\mQ_m)- \mathrm{d_S}(\mP,\mQ) \right)\convW \cN\left(0,\sigma^2_{\lambda}\left(c\right)\right),
\end{equation*}
where the variance $\sigma^2_{\lambda}(c)$ is defined in \eqref{eq:VarianceEvalPlan}. 
    \end{Corollary}

\subsection{Entropic optimal transport maps}
For probability measures $\mP, \mQ$ on a Euclidean space $\Omega =\RR^d$, \cite{seguy2018large} introduced an entropic surrogate for the (unregularized) optimal transport map, defined for $\x\in \RR^d$ by
        $$ \rM(\x)\coloneqq \E_{(\X,\Y)\sim \pipq}[\Y | \X = \x],$$
        where $\pipq$ denotes the entropic optimal transport plan between $\mP$ and $\mQ$ for cost function $c$. 
        Replacing the population measures $(\mP, \mQ)$ by empirical measures $(\mPn, \mQm)$ yields a computationally appealing map estimator $ \rM_{n,m}$ for the (unregularized) optimal transport map which has been analyzed only recently by \cite{pooladian2021entropic, pooladian2022debiaser,pooladian2023minimax} for squared Euclidean costs in various settings. Theorem \ref{Theo:TCLCondPlan} yields novel pointwise weak limits for the empirical entropic optimal map estimator which follow by considering the vector valued map $\eta = (\eta_i)_{i \in \{1, \dots, d\}}$ with $\eta_i\colon \RR^{2d} \rightarrow \RR^d, (\x,\y) \mapsto \y_i$ for any $i \in \{1, \dots, d\}$. 
        \begin{Corollary}
  \label{Coro:Maps}  
  For any $\x\in \R^d$ it follows that 
        $$   \sqrt{\frac{nm}{n+m}}\big( \rM_{n,m}(\x) - \rM(\x)\big) \convW \cN(0,\Sigma^2_{\lambda}(\eta, \x)),$$
        with asymptotic covariance given by $(\Sigma_{\lambda})_{i,j} = \lambda\operatorname{Cov}\left( \X_i, \X_j \right)+ (1-\lambda )\operatorname{Cov}\left( \Y_i, \Y_j \right)$ for \begin{align*}
            \X_i &:= (1-\AQ\AP)^{-1}\AQ\big\{ [\eta_i(\cdot)-(\rM(\x))_{i}] \xipq(\x,\cdot) \big\} (\Y), \text{ with } \X\sim \mP,&&\\
            \Y_i &:= (1-\AP\AQ)^{-1}\big\{ [\eta_i(\cdot)-(\rM(\x))_{i}] \xipq(\x,\cdot) \big\} (\Y), \text{ with } \Y\sim \mQ.&&
        \end{align*} 
        \end{Corollary}

\subsection{Sinkhorn divergence}
Despite having lower statistical and computational complexity, the regularized transport cost lacks direct practical applications. Unlike its non-regularized counterpart, the fact that $\mP=\mQ$  does not imply $\rS({\rm P},{\rm Q})= 0$. This delimits its practical use regarding discrimination of probability measures. This phenomenon is a result of what is commonly referred to as the entropy bias. To overcome this issue \cite{genevay2018} introduce the Sinkhorn divergence:
$$\rD(\mP,{\rm Q})\coloneqq \rS(\mP,{\rm Q})-{\textstyle \frac{1}{2}}\left(  \rS(\mP,\mP)+ \rS({\rm Q},{\rm Q}) \right).$$
\cite{Feydy2019Interpolating} proved that if $\Omega$ is a compact metric space
and $c$ is a Lipschitz cost function that induces a positive universal kernel $K(x,y)=e^{-c},$ then $\rD(\mP,{\rm Q})$ is symmetric in $\mP,{\rm Q}$ and  
$\rD(\mP,{\rm Q})\geq 0$, with $\rD(\mP,{\rm Q})= 0$ if and only if $\mP={\rm Q}$. Therefore, as suggested by \cite{gonzalez2022weak},   $\rD(\mP_n,{\rm Q}_m)$ can be used as a test statistic for the nonparametric goodness-of-fit testing $H_0: \mP=\mQ$ vs. $H_1: \mP\neq \mQ$. Theorem~\ref{Theo:TCLcost} provides the weak limits for $\sqrt{\frac{n\, m}{n+m}}(\rD(\mP,{\rm Q})-\rD(\mP,{\rm Q}))$. 
\begin{Corollary}
    \label{Divergence}
    The empirical Sinkhorn divergence satisfies asymptotically
    $$ \sqrt{\frac{n\, m}{n+m}}\Big(\rD(\mP,{\rm Q})-\rD(\mP,{\rm Q})\Big)\convW \mathcal{N}(0,\sigma_{\lambda,Div}^2),$$
y    with asymptotic variance 
    \begin{align*}
    \sigma_{\lambda,Div}^2\coloneqq &\lambda \operatorname{Var}_{\X\sim\mP}\left[\left[f_{\mP,\mQ}-\frac{f_{\mP,\mP}+f_{\mQ,\mQ}}{2}\right](\X)\right]\\ &+(1-\lambda) \operatorname{Var}_{\Y\sim\mQ}\left[\left[g_{\mP,\mQ}-\frac{g_{\mP,\mP}+g_{\mQ,\mQ}}{2}\right](\Y)\right].
    \end{align*}
\end{Corollary}
It should be noted that under the null hypothesis, the limit is degenerate, i.e., $ \sigma_{\lambda,Div}^2=0$. Therefore, under $\mP=\mQ$, the convergence rate of the empirical Sinkhorn divergence is strictly faster than $\sqrt{
{n\, m}/({n+m})}$. For smooth costs, \cite{gonzalez2022weak} and \cite{Goldfeld2022LimitTF} showed that the correct rate to obtain non-degenerate limits is ${n\, m}/({n+m})$. This requires a second-order analysis of the regularized transport cost, which is beyond the scope of this work and will be considered as future work.
\subsection{Sinkhorn kernels (for Gaussian processes)}
Data such as images, shapes, or media can be viewed as elements of the probability measure space  \citep{glaunes2004diffeomorphic,muandet2012learning,ginsbourger2016design}. To analyze these data, several proposals define kernels on this space. Maximum mean discrepancy \citep{gretton2012kernel}, Sliced Wasserstein \citep{kolouri2018sliced}, or multivariate transport-based quantile  maps \citep{bachoc2020gaussian}, to name just a few, are examples of kernels on the space of probability measures. In the context of regularized transport, \cite{bachoc2023gaussian} defines a kernel as follows:
$$ \rK(\mP,\mQ)\coloneqq\rF\left(\operatorname{Var}_{\U\sim\mathcal{U}}[g_{\rm P,\mathcal{U}}(\U)-g_{\rm Q, \mathcal{U}}(\U)]\right),$$
where $\mathcal{U}$ is a fixed probability measure,
 $\rF(\sqrt{\cdot})$ is completely monotone on $[0, \infty)$ and there exists a finite non-negative Borel measure ${\rm \nu}$ on $[0, \infty)$ such that 
    $\rF(t) = \int_{0}^{\infty}
    e^{ - u t^2}
   d {\rm \nu} (u )$ for any $t \geq 0$. 
From Theorem~\ref{Theo:PotIntApprox} we obtain the limit of the following estimator of the kernel 
$$ \rK_{n}(\mP_n,\mQ_m)\coloneqq\rF\left(\operatorname{Var}_{\U\sim\mathcal{U}_{n}}[g_{\mPn,\mathcal{U}_{n}}(\U)-g_{\mQm, \mathcal{U}_{n}}(\U)]\right),$$
where $\mathcal{U}_{n}=\frac{1}{n}\sum_{i=1}^{n} \delta_{\U_i}$ denotes the empirical measure for a sequence $\U_1, \dots \U_{n}$ of i.i.d. random variables with law  $\mathcal{U}$. 
\begin{Corollary}\label{Coro:Kernel}
    If $\rF$ is continuously differentiable and $g_{\rm P,\mathcal{U}}\neq g_{\rm Q,\mathcal{U}}$, then 
    \begin{align*}
        \sqrt{\frac{n\, m}{n+m}}  \Big(\rK_{n}(\mP_n,\mQ_m)-\rK(\mP,\mQ)\Big) \convW \cN\left(0,  \rF'( \operatorname{Var}_{\U\sim\mathcal{U}}[g_{\mP,\mathcal{U}}(\U)-g_{\mQ, \mathcal{U}}(\U)])^2\, \tilde \sigma_{\lambda}^2\right),
    \end{align*}
    where $\rF'$ denotes the derivative of $\rF$ and $\tilde \sigma^2_{\lambda}$ is given by
\begin{multline*}
4\lambda\operatorname{Var}_{\U\sim\mathcal{U}}\bigg[ \Big( g_{\mP,\mathcal{U}}(\U)- g_{\mQ, \mathcal{U}}(\U)-\E_{\U\sim\mathcal{U}}[g_{\mP,\mathcal{U}}(\U)- g_{\mQ, \mathcal{U}}(\U)] \Big)^2 \\
    + \bigg( \Big((I_{L^2_0(\cU)}-\AQ{\mathcal{A}}_{\mathcal{U}})^{-1}\AQ{\mathcal{A}}_{\mathcal{U}} 
         +(I_{L^2_0(\cU)}-\AP{\mathcal{A}}_{\mathcal{U}})^{-1}\AP{\mathcal{A}}_{\mathcal{U}} \Big)(\bar{g}_{\mP,\mathcal{U}}-\bar{g}_{\mQ,\mathcal{U}})\bigg)(\U)\bigg]\\
         +4\, (1-\lambda)\operatorname{Var}_{\Y\sim\mQ} \bigg[ \Big((I_{L^2_0(\mQ)}-{\mathcal{A}}_{\mathcal{U}}\AQ)^{-1}\big({\mathcal{A}}_{\mathcal{U}} (\bar{g}_{\mP,\mathcal{U}}-\bar{g}_{\mQ,\mathcal{U}})\big)\Big)(\Y)\bigg]\\
         + 4 \, \lambda \operatorname{Var}_{\X\sim\mP} \bigg[ \Big((I_{L^2_0(\mP)}-{\mathcal{A}}_{\mathcal{U}}\AP)^{-1}\big({\mathcal{A}}_{\mathcal{U}} (\bar{g}_{\mP,\mathcal{U}}-\bar{g}_{\mQ,\mathcal{U}})\big)\Big)(\X)\bigg].
\end{multline*}
\end{Corollary}
The proof of Corollary \ref{Coro:Kernel} is stated in Appendix \ref{app:proofsApplications}. 
As in the case of the Sinkhorn divergence (Corollary~\ref{Divergence}) the limit degenerates to $0$ under the null hypothesis $H_0: \mP=\mQ$. The analysis of the limit distribution for this setting requires, as in Corollary~\ref{Divergence}, a second order chaos development of the potentials, which is the scope of future research.

\subsection{Optimal transport colocalization curves}
In the context of computational biology, \cite{klatt2020empirical} propose a colocalization measure for the interaction of protein networks, which is built on the concept of entropic optimal transport. This is a regularized variant of a transport-based concept advocated by \cite{tameling2021Colocalization}. Colocalization analysis for images generated by super-resolution fluorescence microscopy has emerged as a critical tool to analyze the interactions of protein structures as it quantifies their spatial closeness \citep{adler2010quantifying, zinchuk2011quantitative, wang2019spatially}.

In the presented methodology, \cite{klatt2020empirical} consider Euclidean costs $c(x,y) = \|x-y\|$ and define the regularized colocalization measure as the amount of mass transported at scales less than or equal to $t\geq 0$:
\begin{equation*}
\operatorname{RCol}_c(\mP, \mQ,t):=\pi_{\mP,\mQ}({c(\x,\y)\leq t}).
\end{equation*}
Essentially, this measure quantifies the spatial closeness of the underlying structures at various scales. In Theorem 7.1 of \cite{klatt2020empirical}, they provide confidence intervals for the discretized protein distribution (motivated by having only finite number of pixels), using distributional limits. Our Theorem \ref{Theo:TCLPlan} does not necessitate discretized signals and allows the development of confidence bands (for a finite number of $t$ choices) for any type of underlying signal. Moreover, our general theory also remains valid under more general cost functions, which might be of interest for the modelling of non-Euclidean colocalization. Overall, we arrive at the following result.
\begin{Corollary}
    For any finite collection $\{t_1, \dots, t_l\}\subseteq \RR$ the regularized colocalization curve satisfies asymptotically
    \begin{align*}
        \sqrt{\frac{nm}{n+m}}\Big( \operatorname{RCol}_c(\mPn, \mQm,t_i)-  \operatorname{RCol}_c(\mP, \mQ,t_i)\Big)_{i = 1, \dots, l} \convW \cN\Big(0, \Sigma_\lambda(t_1, \dots, t_l)\Big),
    \end{align*}
    with covariance $ \Sigma_\lambda(t_1, \dots, t_l)\in \RR^{l\times l}$ given by $\Sigma_\lambda(t_1, \dots, t_l)_{ij} = \operatorname{Cov}(\Zb_i, \Zb_j)$ for 
    $$\Zb_i \coloneqq  \sqrt{\lambda}(I_{L^2_0(\mP)}-\AQ\AP)^{-1}\big(\eta_{i,\x}-\AQ\eta_{i,\y}\big)(\X)+ \sqrt{1-\lambda} (I_{L^2_0(\mQ)}-\AP\AQ)^{-1}\big(\eta_{i,\y}-\AP\eta_{i,\x}\big)(\Y)$$
    with independent $\X\sim \mP$ and $\Y\sim \mQ$, and evaluation functions $\eta_i(\x,\y) = 1(c(\x,\y)\leq t_i)$.
\end{Corollary}

\section{Proof of Linearization of Empirical Potentials}\label{sec:proofs}

In this section we formulate the proof of  Theorem \ref{Theo:PotIntApprox}, our central result for the linearization of the empirical optimal transport potentials with respect to the (non-deterministic) empirical measure norm. At its core, our proof is motivated by the implicit function theorem, however several obstacles arrive which are due to the changing norm. Overall, the proof consists of six steps, which we now summarize. Proofs to auxiliary results are detailed in Appendix~\ref{app:proofs}. 

\begin{enumerate}
    \item \textit{Linearization of optimality criterion}: We start by linearizing the optimality criterion from \eqref{eq:optimalityCriterionNicePotentials} for which we perform a careful analysis of involved approximation errors (Proposition \ref{prop:linearPot}). Invoking an implicit function approach, we characterize the fluctuation of the empirical entropic optimal transport potentials in the space $L^2(\mPn)\times L^2(\mQn)$ in terms of the inverse of an operator involving $\bAPn$ and $\AQm$. %
    \item \textit{Approximation of empirical operator with auxiliary population operators}: In this step, we approximate the inverse of certain operators involving empirical operators $\bAPn$ and $\AQm$ by a truncated series up to order $N\in \NN$ involving certain auxiliary operators $\DPn$ and $\DQm$ (Proposition \ref{prop:linearPotOpchange}). This substitution allows us to bridge the gap between the empirical setting and the population setting without extending empirical potentials, providing a sensible framework for our analysis. 
    \item \textit{Explicit representation of linearization with auxiliary operators}: As a next step we derive an explicit representation of the terms introduced in the previous step and arrive at an infinite order $V$-statistic (Proposition \ref{prop:Vstatistic}) that depends on $N$.
    \item \textit{From $V$-Statistic to $U$-Statistic}: We then show that the $V$-statistic can be approximated in terms of a $U$-statistic, which is more amenable to statistical analysis (Proposition \ref{prop:Ustatistic}). Our approach relies on combinatorial arguments, which show that the relative amount of dependent terms is asymptotically negligible.
    \item \textit{Derivation of the H\'ajek projection}: This step consists in computing the H\'ajek projection for the $U$-statistic from previous step (Proposition \ref{prop:writeFproperly}). The resulting projection yields a suitable representation as a truncated series of population operators.
    \item \textit{Weak convergence of the H\'ajek projection}: Finally, we further refine our representation of the  H\'ajek projection to return to an infinite series (Proposition \ref{prop:FandFinf}) and show its weak convergence by assessing its asymptotic variance (Proposition \ref{prop:HajekProjection}).
\end{enumerate}

Before we start with the different step we recall the set of assumptions under which all the subsequent results will be shown. We will refer to them as the \emph{standard assumptions}: 
The space $\Omega$ is Polish and the cost function $c\in L^\infty(\Omega\times \Omega)$ is uniformly bounded and measurable. Further, let $\X_1, \dots, \X_n$  and $\Y_1, \dots, \Y_m$ be two mutually independent samples of i.i.d. random variables with respective distributions  $\mP, \mQ\in \mathcal{P}(\Omega)$, and denote by $\mPn$ and $\mQn$ the corresponding empirical measures. Moreover, let $\eta \in L^\infty(\Omega\times \Omega)$ be a uniformly bounded and measurable function. 

\subsection*{Step 1 - Linearization of optimality criterion}
We observe that the optimality conditions \eqref{optimalvalues} allow us to write the potentials as a solution of an algebraic equation in $L^2(\mPn)\times L^2_0(\mQm)$. 
Let 
$$\overline F_{n,m}:L^2(\mP_n)\times L^2_0(\mQ_m) \to L^2(\mP_n)\times L^2_0(\mQ_m)  $$
be defined as 
\begin{equation*}
\overline F_{n,m}\left(\begin{array}{c}
     f  \\
     g
\end{array} \right)=
    \left(\begin{array}{c}
          f+\log\left(\int e^{g(\y)}C(\cdot,\y)d\mQ_m(\y)\right)\\
    g+\log\left(\int e^{f(\x)}C(\x,\cdot)d\mP_n(\x)\right) - \int \log\left(\int e^{f(\x)}C(\x,\y)d\mP_n(\x)\right)d \mQm(\y)
    \end{array}\right)\, .
\end{equation*}
 We want to differentiate $\overline F_{n,m}$ w.r.t. $(f,g)^t$, and evaluate that derivative in the optimal values, i.e.,
\begin{equation}
    \label{optimalvalues}
 \left(\begin{array}{c}
     f_{\mP_n,\mQ_m}  \\
     g_{\mP_n,\mQ_m}
\end{array} \right)
\end{equation}
We need to obtain that  derivative in the strong sense, but we first need to find a candidate. Then we set a direction $(f,g)^t$ and compute the directional derivative of $F$ in \eqref{optimalvalues}. The result is the operator 
\begin{align*}
    \begin{split}
        {\overline \Gamma_{n,m}}:L^2(\mP_n)\times L^2_0(\mQ_m)&\to L^2(\mP_n)\times L^2_0(\mQ_m), \\
        \left(\begin{array}{c}
f  \\
     g
\end{array} \right)&\mapsto \left(\begin{array}{c}
f  \\
     g
\end{array} \right)
      +\left(\begin{array}{c}
    \AQm\\
\bAPn
    \end{array}\right)\left(\begin{array}{c}
g  \\
     f
\end{array} \right)
    \end{split}
\end{align*}
 with
 \begin{equation*}
     \APn f = \frac{\int e^{f_{\mP_n,\mQ_m}(\x)}C(\x,\cdot)f(\x)d\mP_n(\x)}{\int e^{f_{\mP_n,\mQ_m}(\x)}C(\x,\cdot)d\mP_n(\x)},
         \quad  \AQm g = \frac{\int e^{g_{\mP_n,\mQ_m}(\y)}C(\cdot,\y)g(\y)d\mQ_m(\y)}{\int e^{g_{\mP_n,\mQ_m}(\x)}C(\cdot,\y)d\mQ_m(\y)},
 \end{equation*}
 and $$\bAPn f = \APn f - \int (\APn f) d \mQm.$$ 

\begin{Proposition}\label{prop:linearPot}
Under the standard assumptions, it holds for $n,\, m\to \infty$ with $m/(n+m)\to \lambda\in (0,1)$ that 
    \begin{equation}
    \label{linearBothrows}
  \E\left\|  \overline \Gamma_{n,m}\left(\begin{array}{c}
     \bfpq - \bfmn   \\
     \bgpq-\bgmn
\end{array} \right)-\left(\begin{array}{c}
     \mathcal{B}_\y(\mQ_m-\mQ)  \\
      \mathcal{B}_\x(\mP_n-\mP)-\int \mathcal{B}_\x(\mP_n-\mP)d\mQ_m
\end{array}\right)\right\|_{n\times m}{\leq \frac{\cK(c)}{n}},
\end{equation}
where we define
\begin{equation}\label{eq:DefinitionB}
     \left(\begin{array}{c}
     \mathcal{B}_\y(\mQ_m-\mQ)  \\
      \mathcal{B}_\x(\mP_n-\mP)
\end{array}\right)\coloneqq \left(\begin{array}{c}
     \int \xi_{\mP,\mQ}(\cdot,\y)d(\mQ_m-\mQ)(\y) \\
     \int \xi_{\mP,\mQ} (\x, \cdot)d(\mP_n-\mP)(\x)
\end{array} \right).
\end{equation}
Moreover, it holds 
$$ \E\left( \int \mathcal{B}_\x(\mP_n-\mP) d\mQ_m \right)^2\leq \frac{\cK(c)}{n\, m}. $$ 
\end{Proposition}

\subsection*{Step 2 - Approximation of empirical operator with auxiliary population operators}
From Lemma~\ref{lemma:empiricalInverse}, the inverse operators $(I_{L^2(\mP_n)}-\AQm\bAPn)^{-1}:L^2(\mP_n)\to L^2(\mP_n)$ and $(I_{L^2_0(\mQ_m)}-\bAPn\AQm)^{-1}:L^2_0(\mQ_m)\to L^2_0(\mQ_m)$ are well-defined. Basic algebraic calculations show 
\begin{equation}\label{eq:InverseGammaDefinition}
    {\textstyle\small \overline \Gamma_{n,m}^{-1}\left(\begin{array}{c} f\\ g \end{array}\right) = \left(\begin{array}{c}
           (I_{L^2(\mP_n)}-\AQm\bAPn)^{-1}f - (I_{L^2(\mP_n)}-\AQm\bAPn)^{-1}\AQm g \\
      (I_{L^2_0(\mQ_m)}-\bAPn\AQm)^{-1} g -  (I_{L^2_0(\mQ_m)}-\bAPn\AQm)^{-1}\bAPn f  \end{array}\right)},
\end{equation}
so that $\overline\Gamma_{n,m}$ is invertible in $L^2(\mPn)\times L^2_0(\mQ_m)$. 
At this point, let us recall that the inverse of an operator $(I-A)$ with $\|A\|<1$ can be expressed in terms of a Neumann series $(I-A)^{-1}=\sum_{k=0}^\infty A^k$, where the series converges in norm operator. In particular, it enables the subsequent approximation result.

\begin{Lemma}\label{Coro:limitSum}
Under the standard assumptions, it holds for any $n,m,N\in \NN$ and $\kappa \in \RR$  and $\delta(c)\in (0,1)$ from Lemma \ref{lemma:empiricalInverse} that
\begin{align*}
            &\sup_{\substack{f\in L^2_0(\mPn)\\ \|f\|_{n}\leq 1} } \|\sum_{k=0}^{N} (\AQm\APn)^k (f+\kappa \one) -(I_{L^2(\mP_n)}-\AQm\bAPn)^{-1} (f+\kappa \one)\|_{n}  \leq \frac{\delta(c)^{N+1}}{1-\delta(c)} + N|\kappa|,\\
           & \sup_{\substack{g\in L^2_0(\mQm)\\ \|g\|_{m}\leq 1}}\| \sum_{k=0}^{N} (\APn\AQm)^k (g +\kappa \one) -(I_{L^2_0(\mQ_m)}-\bAPn\AQm)^{-1}(g)\|_{m}  \leq \frac{\delta(c)^{N+1}}{1-\delta(c)} + (N+1)|\kappa|.
\end{align*}
\end{Lemma}

 Having $\mathcal{A}_{\mP}$ and $\mathcal{A}_{\mQ}$   instead of their empirical versions $\APn$ and $\AQm$ as in Proposition~\ref{prop:linearPot}, would provide an easier way to apply this result to find the weak limits of the regularized optimal transport plans.  However, the operator $\mathcal{A}_{\mP}$ (resp. $\mathcal{A}_{\mQ}$) is not defined $ L^2(\mP_n)$  (resp. $ L^2(\mQ_m)$), but only on $ L^2(\mP)$ (resp. $ L^2(\mQ)$). The auxiliary operators
  \begin{equation*}
     \mathcal{D}_{\mP_n}f = \frac{\int e^{f_{\mP,\mQ}(\x)}C(\x,\cdot)f(\x)d\mP_n(\x)}{\int e^{f_{\mP,\mQ}(\x)}C(\x,\cdot)d\mP(\x)},
         \quad  \mathcal{D}_{\mQ_m}g = \frac{\int e^{g_{\mP,\mQ}(\y)}C(\cdot,\y)g(\y)d\mQ_m(\y)}{\int e^{g_{\mP,\mQ}(\x)}C(\cdot,\y)d\mQ(\y)}
 \end{equation*}
 solve this issue. The subsequent result provides quantitative bounds for the difference between $\APn$ and $\DPn$ (resp.\ $\AQm$ and $\DQm$), which we use to approximate the different blocks in the representation \eqref{eq:InverseGammaDefinition} of $\overline\Gamma_{n,m}^{-1}$ in terms of $\DPn$ and $\DQm$. 
 
 \begin{Lemma}\label{LemmaBoundOfEigD}
Under the standard assumptions, the following bounds are satisfied deterministically for constants $\cK(c), \cK'(c), \cK''(c), \cK'''(c)\geq 1$ which only depend on $c$.
 \begin{enumerate}
    \item For the operator norms induced by $L^2(\mPn)$ and $L^2(\mQm)$ it holds that 
    \begin{align*}
        \max\left(\norm{\DPn - \APn}_{n,m}, \norm{\DQm - \AQm}_{m,n}\right)\leq \norm{\xipq - \ximn}_{n\times m} \leq \cK(c).
    \end{align*}
     \item For the operator norms induced by $L^2(\mPn)$ and $L^2(\mQm)$ it holds that 
    \begin{align*}
        \max\left(\norm{\DPn }_{n,m}, \norm{\DQm}_{m,n}\right)\leq 1 + \norm{\xipq - \ximn}_{n\times m}\leq \cK'(c)
    \end{align*}
    as well as 
    \begin{multline*}
   \max(\|(\DQm\DPn-\AQm\APn)\|_{n}, \|(\DPn\DQm-\APn\AQm)\|_{m}) \\
  \leq  \cK'(c) \|\xi_{\mP,\mQ}- \xi_{\mPn,\mQm}\|_{n\times m}\leq \cK''(c).
    \end{multline*}
	\item For any $n,m,N\in \NN$ and $\kappa\in \RR$ it holds for  $\delta(c)<1$ is defined in \eqref{eq:boundeig} that
\begin{subequations}
\begin{align}
           &\label{eq:UpperBoundDPnDQm1}\sup_{\substack{f\in L^2_0(\mPn)\\ \|f\|_{n}\leq 1} } \left\|\left[\sum_{k=0}^{N} (\DQm\DPn)^k -(I_{L^2(\mP_n)}-\AQm\bAPn)^{-1} \right] (f+\kappa \one)\right\|_{n} \\
          \label{eq:UpperBoundDPnDQm2} +&\sup_{\substack{f\in L^2_0(\mPn)\\ \|f\|_{n}\leq 1} } \left\|\left[\sum_{k=0}^{N} (\DPn\DQm)^k \DPn-(I_{L^2_0(\mQm)}-\bAPn\AQm)^{-1}\bAPn \right] (f + \kappa \one)\right\|_{m}\\
           \label{eq:UpperBoundDPnDQm3} +& \sup_{\substack{g\in L^2_0(\mQm)\\ \|g\|_{m}\leq 1}}\left\| \sum_{k=0}^{N} (\DPn\DQm)^k (g +\kappa \one) - (I_{L^2_0(\mQ_m)}-\bAPn\AQm)^{-1}(g)\right\|_{m}  \\
           \label{eq:UpperBoundDPnDQm4} +& \sup_{\substack{g\in L^2_0(\mQm)\\ \|g\|_{m}\leq 1}}\left\| \sum_{k=0}^{N} (\DQm\DPn)^k \DQm (g +\kappa \one) - (I_{L^2(\mPn)}-\AQm\bAPn)^{-1} \AQm(g)\right\|_{n}\\
           \notag\leq & \;\;\cK'''(c)^N(\norm{\xipq - \ximn}_{n\times m} + |\kappa|)+\frac{4\delta(c)^{N+1}}{1-\delta(c)}.
\end{align}
\end{subequations}

\end{enumerate}
\end{Lemma}

To conclude, via Proposition~\ref{prop:linearPot}, we obtain the following result.

\begin{Proposition}\label{prop:linearPotOpchange}
Under the standard assumptions, it holds for $\cB_\x, \cB_\y$ defined in \eqref{eq:DefinitionB} as $n,\, m\to \infty$, with $\frac{m}{n+m}\to \lambda\in (0,1)$, and $N=N(n,m) = \Nnm$ that 
    \begin{equation*}
   \left\|  \left(\begin{array}{c}
     f_{\mP_n,\mQ_m}- f_{\mP,\mQ}   \\
     g_{\mP_n,\mQ_m} -
     g_{\mP,\mQ}
\end{array} \right)-\Theta_{n,m}^N\left(
\begin{array}{c}
\mathcal{B}_\y (\mQm-\mQ)\\
\mathcal{B}_\x (\mPn-\mP)
\end{array} \right)\right\|_{n\times m}{=\opnm},
\end{equation*}
Here, the operator $\Theta_{n,m}^N:L^2_0(\mP_n)\times L^2_0(\mQ_m) \to L^2(\mP_n)\times L^2(\mQ_m) $ is defined~as
$$  {\textstyle\small \left(\begin{array}{c}
     f \\
     g 
\end{array}\right)\mapsto\left(\begin{array}{c}
           -\sum_{k=0}^N(\mathcal{D}_{\mQ_m}\mathcal{D}_{\mP_n})^k f + \sum_{k=0}^N(\mathcal{D}_{\mQ_m}\mathcal{D}_{\mP_n})^{k}\mathcal{D}_{\mQ_m}g \\
      -\sum_{k=0}^N(\mathcal{D}_{\mP_n}\mathcal{D}_{\mQ_m})^{k} g +  \sum_{k=0}^N(\mathcal{D}_{\mP_n}\mathcal{D}_{\mQ_m})^{k} \mathcal{D}_{\mP_n} f  \end{array}\right)}. $$
\end{Proposition}

\subsection*{Step 3 - Explicit representation of linearization with auxiliary operators}

For the proof of Theorem \ref{Theo:PotIntApprox} we define the quantities
\begin{equation}\label{eq:Anml-quantity}
\begin{aligned}
     A_{n,m}^{(1)}&\coloneqq {  \int \eta(\x,\y)(\bfmn(\x)  - \bfpq(\x))\xipq(\x,\y) d(\mPn \otimes \mQm)(\x,\y),}\\
     A_{n,m}^{(2)}&\coloneqq \int \eta(\x,\y)(\bgmn(\x)  - \bgpq(\x))\xipq(\x,\y) d(\mPn \otimes \mQm)(\x,\y).
\end{aligned}
\end{equation}
Next, we infer, using Proposition \ref{prop:linearPotOpchange} for $N=N(n,m) = \Nnm$, that $ A_{n,m}^{(l)} =  \tilde A_{n,m}^{(l)} + o_p(\sqrt{(n+m)/nm})$ for $l\in \{1,2\}$ with 
\begin{align*}
     \tilde A_{n,m}^{(1)} &\coloneqq \int  \eta  \left(-
  \sum_{k=0}^N(\mathcal{D}_{\mQ_m}\mathcal{D}_{\mP_n})^k \cB_{\y}(\mQm - \mQ)+ \sum_{k=0}^N(\mathcal{D}_{\mQ_m}\mathcal{D}_{\mP_n})^{k}\mathcal{D}_{\mQ_n} \cB_{\x}(\mPn - \mP)\right)\xi_{\mP,\mQ}d {\mP_n \mQ_m},\\ 
   \tilde A_{n,m}^{(2)} &\coloneqq \int\eta \left( - \sum_{k=0}^N(\mathcal{D}_{\mP_n}\mathcal{D}_{\mQ_m})^{k}\cB_{\x}(\mPn - \mP)+ \sum_{k=0}^N(\mathcal{D}_{\mP_n}\mathcal{D}_{\mQ_m})^k \mathcal{D}_{\mP_n}\cB_{\y}(\mQm - \mQ)\right)\xi_{\mP,\mQ}d {\mP_n \mQ_m},
\end{align*}
where  $\cB_\x$, $\cB_\y$ are defined in \eqref{eq:DefinitionB}. %
The following result shows that each one of the terms of $\tilde A_{n,m}^{(1)}$ and $\tilde A_{n,m}^{(2)}$ can be written as an infinite order $V$-statistic.  Before stating the mentioned result, for $n\in \N$, we denote $[[n]]=\{1, \dots, n\}$ and $\mathbf{i}=(i_1, \dots, i_k)\in [[n]]^k$.  Furthermore, for $\mathbf{i}\in [[n]]^k$, the notation $\X_{\mathbf{i}}$ stands for $(\X_{i_1}, \dots, \X_{i_k})$.

\begin{Proposition}\label{prop:Vstatistic}
Under the standard assumptions, it holds for $l\in \{1,2\}$ as $n,\, m\to \infty$, with $\frac{m}{n+m}\to \lambda\in (0,1)$, and $N=N(n,m) = \Nnm$ that,\begin{align}\label{eq:AnmPrime_representation}
     \tilde A_{n,m}^{(l)}= \sum_{k = 0}^N \left[-U_{n,m,k}^{(2l-1)}+ U_{n,m,k}^{(2l)}\right]+o_p\left(\sqrt{\frac{n+m}{nm}}\right),%
\end{align} where we set
 \begin{align*}
      & U_{n,m,k}^{(1)} = \frac{1}{n^{k+1}\, m^{k+2}} \sum_{\substack{\mathbf{i}\in [[n]]^{k+1}\\ {\mathbf{j}\in [[m]]^{k+2}}}}
  f_k^{(1)}( \X_{\mathbf{i}},\Y_{\mathbf{j}}), \ & U_{n,m,k}^{(2)} = \frac{1}{n^{k+2}\, m^{k+2}} \sum_{\substack{\mathbf{i}\in [[n]]^{k+2}\\ {\mathbf{j}\in [[m]]^{k+2}}}}
  f_k^{(2)}( \X_{\mathbf{i}},\Y_{\mathbf{j}}),\\
  & U_{n,m,k}^{(3)} = \frac{1}{n^{k+2}\, m^{k+1}} \sum_{\substack{\mathbf{i}\in [[n]]^{k+2}\\ {\mathbf{j}\in [[m]]^{k+1}}}}
  f_k^{(3)}( \X_{\mathbf{i}},\Y_{\mathbf{j}}), \ & U_{n,m,k}^{(4)} = \frac{1}{n^{k+2}\, m^{k+2}} \sum_{\substack{\mathbf{i}\in [[n]]^{k+2}\\ {\mathbf{j}\in [[m]]^{k+2}}}}
  f_k^{(4)}( \X_{\mathbf{i}},\Y_{\mathbf{j}}),
    \end{align*}
and, upon abbreviating $\xi \coloneqq \xipq$, write
\begin{align*}
    f_k^{(1)}( \X_{\mathbf{i}},\Y_{\mathbf{j}})&= \xi(\X_{i_{k+1}}\Y_{j_{k+1}})\eta(\X_{i_{k+1}},\Y_{j_{k+1}})  
     (\xi(\X_{i_1},\Y_{j_{k+2}})-1)
  \prod_{l=1}^{k}\xi(\X_{i_l},\Y_{j_l})\xi(\X_{i_{l+1}},\Y_{j_l}),\\
   f_k^{(2)}( \X_{\mathbf{i}},\Y_{\mathbf{j}})&= \xi(\X_{i_{k+1}}\Y_{j_{k+1}})\eta(\X_{i_{k+1}},\Y_{j_{k+1}})   \xi(\X_{i_1},\Y_{j_{k+2}})  (\xi(\X_{i_{k+2}},\Y_{j_{k+2}})-1)
   \\
   & \qquad \qquad \qquad\qquad\qquad  \qquad \qquad \qquad\qquad\times\prod_{l=1}^{k}\xi(\X_{i_l},\Y_{j_l})\xi(\X_{i_{l+1}},\Y_{j_l}),\\
   f_k^{(3)}( \X_{\mathbf{i}},\Y_{\mathbf{j}})&= \xi(\X_{i_{k+1}}\Y_{j_{k+1}})\eta(\X_{i_{k+1}},\Y_{j_{k+1}})  
     (\xi(\X_{i_{k+2}},\Y_{j_1})-1)
   \prod_{l=1}^{k}\xi(\X_{i_l},\Y_{j_l})\xi(\X_{i_{l}},\Y_{j_{l+1}}),\\
f_k^{(4)}( \X_{\mathbf{i}},\Y_{\mathbf{j}})&= \xi(\X_{i_{k+1}}\Y_{i_{k+1}})\eta(\X_{i_{k+1}},\Y_{j_{k+1}})   \xi(\X_{i_{k+2}},\Y_{j_{1}})  (\xi(\X_{i_{k+2}},\Y_{j_{k+2}})-1)
   \\
   & \qquad \qquad \qquad\qquad\qquad  \qquad \qquad \qquad\qquad\times\prod_{l=1}^{k}\xi(\X_{i_l},\Y_{j_l})\xi(\X_{i_{l}},\Y_{j_{l+1}}).
\end{align*}
   
\end{Proposition}

\subsection*{Step 4 - From $V$-Statistic to $U$-Statistic}

Proposition~\ref{prop:Vstatistic} shows that $\tilde A_{n,m}^{(l)}$ for $l\in \{1, 2\}$ is, up to negligible error, a $V$-statistic of order $k(n,m)$, with $k(n,m)\to \infty$ as $n,m\to \infty$. That is, an infinite order $V$-statistic.  The following result shows that $\tilde A_{n,m}^{(l)}$ for $l\in \{1, 2\}$ is, up to negligible error, an infinite order $U$-statistic. 
\begin{Proposition}\label{prop:Ustatistic}
Under the standard assumptions, let $U_{n,m,k}^{(1)}$, $U_{n,m,k}^{(2)}$, $U_{n,m,k}^{(3)}$ and  $U_{n,m,k}^{(4)}$, be as in Proposition~\ref{prop:Vstatistic} and let $ \mathcal{S}_{k,n}\subset [[n]]^k$ be set of $\mathbf{i}\in [[n]]^k$ such that $i_{p}\neq i_{q}$, for $p\neq q$. Then the following equalities hold up to additive $o_p\left(\sqrt{\frac{n+m}{nm}}\right)$ terms for $n,m\to \infty$, with $\frac{m}{n+m}\to \lambda\in (0,1)$ and $N=\Nnm$,
 \begin{align*}
    \!\!\!\!\!\!  &\sum_{k = 0}^{N} U_{n,m,k}^{(1)} = \sum_{k = 0}^{N} \frac{1}{n^{k+1}\, m^{k+2}} \sum_{\substack{\mathbf{i}\in \mathcal{S}_{k+1,n}\\ {\mathbf{j}\in \mathcal{S}_{k+2,m}}}}
  \!\!\!\!\!\!f_k^{(1)}( \X_{\mathbf{i}},\Y_{\mathbf{j}}),  \!\!\!& \sum_{k = 0}^{N} U_{n,m,k}^{(2)} = \sum_{k = 0}^{N} \frac{1}{n^{k+2}\, m^{k+2}} \sum_{\substack{\mathbf{i}\in \mathcal{S}_{k+2,n}\\ {\mathbf{j}\in \mathcal{S}_{k+2,m}}}}
  \!\!\!\!\!\!f_k^{(2)}( \X_{\mathbf{i}},\Y_{\mathbf{j}}),\\
 \!\!\!\!\!\! & \sum_{k = 0}^{N} U_{n,m,k}^{(3)} =\sum_{k = 0}^{N} \frac{1}{n^{k+2}\, m^{k+1}} \sum_{\substack{\mathbf{i}\in \mathcal{S}_{k+2,n}\\ {\mathbf{j}\in \mathcal{S}_{k+1},m}}}
  \!\!\!\!\!\!f_k^{(3)}( \X_{\mathbf{i}},\Y_{\mathbf{j}}), \!\!\!& \sum_{k = 0}^{N} U_{n,m,k}^{(4)} =\sum_{k = 0}^{N} \frac{1}{n^{k+2}\, m^{k+2}} \sum_{\substack{\mathbf{i}\in \mathcal{S}_{k+2,n}\\ {\mathbf{j}\in \mathcal{S}_{k+2,m}}}}
  \!\!\!\!\!\!f_k^{(4)}( \X_{\mathbf{i}},\Y_{\mathbf{j}}).
    \end{align*} 
\end{Proposition}

\subsection*{Step 5 - Derivation of H\'ajek projection}
Next, we define the random variables
\begin{align*}
      &\U_{n,m,N}^{(1)} = \sum_{k = 0}^{N} \frac{1}{n^{k+1}\, m^{k+2}} \sum_{\substack{\mathbf{i}\in \mathcal{S}_{k+1,n}\\ {\mathbf{j}\in \mathcal{S}_{k+2,m}}}}
  f_k^{(1)}( \X_{\mathbf{i}},\Y_{\mathbf{j}}),  & \U_{n,m,N}^{(2)} = \sum_{k = 0}^{N} \frac{1}{n^{k+2}\, m^{k+2}} \sum_{\substack{\mathbf{i}\in \mathcal{S}_{k+2,n}\\ {\mathbf{j}\in \mathcal{S}_{k+2,m}}}}
  f_k^{(2)}( \X_{\mathbf{i}},\Y_{\mathbf{j}}),\\
  & \U_{n,m,N}^{(3)} =\sum_{k = 0}^{N} \frac{1}{n^{k+2}\, m^{k+1}} \sum_{\substack{\mathbf{i}\in \mathcal{S}_{k+2,n}\\ {\mathbf{j}\in \mathcal{S}_{k+1,m}}}}
  f_k^{(3)}( \X_{\mathbf{i}},\Y_{\mathbf{j}}),  &\U_{n,m,N}^{(4)} =\sum_{k = 0}^{N} \frac{1}{n^{k+2}\, m^{k+2}} \sum_{\substack{\mathbf{i}\in \mathcal{S}_{k+2,n}\\ {\mathbf{j}\in \mathcal{S}_{k+2,m}}}}
  f_k^{(4)}( \X_{\mathbf{i}},\Y_{\mathbf{j}}),
    \end{align*}
and consider the H\'ajek projection of each $\U^{(l)}_{n,m,N}$ for  $l \in \{1,\dots, 4\}$ onto the set of random variables $\sum_{i= 1}^{n} \mathbf{g}_{i,x}(\X_i) + \sum_{j = 1}^{m} \mathbf{g}_{j,y}(\Y_j)$. We denote it by 
\begin{align*}
   \frac{1}{n}\sum_{i= 1}^{n}  F^{(l)}_{N, x}(\X_i) + \frac{1}{m}\sum_{j= 1}^{m}  F^{(l)}_{N,y}(\Y_j).
\end{align*}
The respective functions $F^{(l)}_{N,x}$ and $F^{(l)}_{N,y}$ are given by suitable conditional expectations and  following result shows that  $F^{(l)}_{k}(\y)$ can be written in a more appropriate way.

  \begin{Proposition}\label{prop:writeFproperly}
Under the standard assumptions, it holds that
          \begin{align*}%
   F^{(1)}_{N,y}(\y) &=  \sum_{k=0}^N\left\{( (\AP\AQ)^{k} \AP \eta_{\x})(\y)-\E[(\AP\AQ)^{k}\AP \eta_{\x})(\Y)]\right\}, & F^{(1)}_{N,x}(\x) &= 0,\\
   F^{(2)}_{N,x}(\x) &=  \sum_{k=0}^N\left\{( (\AQ\AP)^{k} \AQ\AP\eta_{\x})(\x)-\E[( (\AQ\AP)^{k} \AQ\AP\eta_{\x})(\X)]\right\},& F^{(2)}_{N,y}(\y) &= 0,\\
   F^{(3)}_{N,x}(\x) &= \sum_{k=0}^N \left\{( (\AQ\AP)^{k} \AQ\eta_{\y})(\x)-\E[( (\AQ\AP)^{k} \AQ\eta_{\y})(\X)]\right\},& F^{(3)}_{N,y}(\y) &= 0,\\
   F^{(4)}_{N,y}(\y) &= \sum_{k=0}^N\left\{ ( (\AP\AQ)^{k} \AP\AQ\eta_{\y})(\y)-\E[( (\AP\AQ)^{k} \AP\AQ\eta_{\y})(\Y)]\right\},& F^{(4)}_{N,x}(\x) &= 0,
   \end{align*}
where $\X\sim {\rm P}$ and $\Y\sim {\rm Q}$. Further, it holds $\E(F^{(l)}_{N,x}(\X)) = \E(F^{(l)}_{N,y}(\Y)) = \E(\U^{(l)}_{n,m,N})=0$ for any $l\in \{1, \dots, 4\}$, $z\in \{x,y\}$,  and $n,m, N\in \NN$.
  \end{Proposition}

\subsection*{Step 6 - Weak convergence of H\'ajek Projection}
  Based on Proposition~\ref{prop:writeFproperly} and since $\AP(\one) = \AQ(\one)= \one$ in conjunction with \eqref{eq:ExpectationExchange} we infer for $\bar\eta_{\x}\coloneqq \eta_{\x} - \E[\eta_{\x}(\X)]$ that
\begin{align*}
    \frac{1}{m}  \sum_{j=1}^m F_{N,y}^{(1)}(\Y_j) &= \frac{1}{m} \sum_{j=1}^m \bigg(\sum_{k=0}^N  ((\AP\AQ)^{k} \AP \bar \eta_{\x})(\Y_j)- \E[((\AP\AQ)^{k}\AP \bar \eta_{\x})(\Y)] \bigg)\\
    &=  \frac{1}{m} \sum_{j=1}^m \bigg(\sum_{k=0}^N  ((\AP\AQ)^{k} \AP \bar \eta_{\x})(\Y_j) \bigg).
\end{align*}
Since $(\AP\AQ)\colon L^2_0(\mQ) \to L^2_0(\mQ)$ has, by Lemma \ref{lemma:empiricalInverse},  an operator norm strictly smaller than one, it follows that the  operator $(1-\AP\AQ)^{-1}$ is  bounded and can be represented in terms of a Neumann series, i.e.,  
$$ (1-\AP\AQ)^{-1} = \sum_{k = 0}^{\infty} (\AP\AQ)^k \colon L^2_0(\mQ) \to L^2_0(\mQ).$$
As a consequence, the term $$\sum_{k= 0}^{\infty} (\AP\AQ)^k\AP \bar \eta_{\x} = (1-\AP\AQ)^{-1}\AP\bar\eta_{\x}$$ is well-defined as an element of $L^2_0(\mQ)$. Therefore, it is finite $\mQ$-almost everywhere and $\sum_{k = 0}^\infty\E[|((\AP\AQ)^{k}\AP \bar\eta_{\x})(\Y)|]<\infty$.  

The next step is to prove that the difference between $ \frac{1}{m}  \sum_{j=1}^m F_{N,y}^{(1)}(\Y_j)$ and its corresponding limit for $N\rightarrow \infty$ tends to $0$ faster than the parametric rate, i.e., is of order $o_p(\sqrt{(n+m)/nm})$. This is content of the following result. 
\begin{Proposition}\label{prop:FandFinf}
Define the functions 
 \begin{align*}%
   F^{(1)}_{\infty,y}(\y) &=  \sum_{k=0}^\infty ( (\AP\AQ)^{k} \AP \bar\eta_{\x})(\y), &
   F^{(2)}_{\infty,x}(\x) &=  \sum_{k=0}^\infty( (\AQ\AP)^{k} \AQ\AP \bar\eta_{\x})(\x),\\
   F^{(3)}_{\infty,x}(\x) &= \sum_{k=0}^\infty ( (\AQ\AP)^{k} \AQ \bar\eta_{\y})(\x), &
   F^{(4)}_{\infty,y}(\y) &= \sum_{k=0}^\infty ( (\AP\AQ)^{k} \AP\AQ\bar\eta_{\y})(\y).
   \end{align*}
    Then, under the standard assumptions, it follows for each $l \in \{1, \dots, 4\}$ and $n,m, N\rightarrow\infty$ with $\frac{m}{n+m} \to \lambda \in (0,1)$ that
    $$  \E\left[\left( F^{(l)}_{N,x}(\X)-F^{(l)}_{\infty,x}(\X)\right)^2 \right]+ \E\left[\left( F^{(l)}_{N,y}(\Y)-F^{(l)}_{\infty,y}(\Y)\right)^2\right] = o\left( \frac{n+m}{nm} \right).$$
In the expectations, we consider $\X\sim {\rm P}$ and $\Y\sim {\rm Q}$. 
\end{Proposition}

As a consequence of Proposition \ref{prop:FandFinf}, it follows for $n,m, N\rightarrow\infty$ and each $l\in\{1, \dots, 4\}$,
$$ \frac{1}{n} \sum_{i = 1}^{n} F_{N,x}^{(l)}(\X_i) +  \frac{1}{m} \sum_{j = 1}^{m} F_{N,y}^{(l)}(\Y_ij) =  \frac{1}{n} \sum_{i = 1}^{n} F_{\infty,x}^{(l)}(\X_i) +  \frac{1}{m} \sum_{j = 1}^{m} F_{\infty,y}^{(l)}(\Y_ij) + \opnm.$$

As a last step before finishing the proof of Theorem \ref{Theo:PotIntApprox} we show that the left-hand side coincides, up to negligible error, with the random quantity $\U_{n,m,N}^{(l)}$ from Step 5. This is content of the subsequent result and relies on careful analysis of the variances. 

\begin{Proposition}\label{prop:HajekProjection}
    Under the standard assumptions, it follows for each $l \in \{1, \dots, 4\}$ and $ N = \Nnm$ for $n,m \to \infty$ with $\frac{m}{n+m}\to \lambda\in (0,1)$ that
    \begin{align*}
    \U_{n,m,N}^{(l)}&= \frac{1}{n}  \sum_{i=1}^n  F^{(l)}_{N,x}(\X_i)   +
 \frac{1}{m}  \sum_{j=1}^m  F^{(l)}_{N,y}(\Y_j) + \opnm,
\end{align*}
where in the expectations we consider $\X\sim {\rm P}$ and $\Y\sim {\rm Q}$.
\end{Proposition}

With all these tools we can finally formulate the proof of Theorem \ref{Theo:PotIntApprox}.

\begin{proof}[Proof of Theorem \ref{Theo:PotIntApprox}]
    Let $l\in \{1, 2\}$ and recall the quantities $A_{n,m}^{(l)}$ from \eqref{eq:Anml-quantity}. Combining Propositions \ref{prop:Vstatistic}, \ref{prop:Ustatistic}, and \ref{prop:writeFproperly} yields for $n,m\to \infty$ with $\frac{m}{n+m} \to \lambda\in (0,1)$ and $N = N(n,m) = \Nnm$ that $A_{n,m}^{(l)} = \U^{(2l)}_{n,m,N} - \U^{(2l-1)}_{n,m,N}+ \opnm.$
    Invoking Propositions \ref{prop:FandFinf} and \ref{prop:HajekProjection} and recalling that $\E[F^{(l)}_{N,x}(\X)] = \E[F^{(l)}_{N,y}(\Y)] = \E[\U^{(l)}_{n,m,N}]=0$, the assertion follows at once. 
\end{proof}

\begin{acks}[Acknowledgments] 
A. González-Sanz expresses gratitude to Professor A. Munk for the kind invitation to the University of Göttingen, where the final stages of this work were conducted.
\end{acks}

\begin{funding}
The research of Alberto González-Sanz is partially supported by grant PID2021-128314NB-I00 funded by MCIN/AEI/
10.13039/501100011033/FEDER, UE and by the AI Interdisciplinary Institute ANITI, which is
funded by the French \emph{``investing for the Future – PIA3''} program under the
Grant agreement ANR-19-PI3A-0004. The research of  S. Hundrieser is partially supported by the 
Research Training Group 2088 ``\emph{Discovering structure in complex data: Statistics
meets Optimization and Inverse Problems}'', which is funded by the  Deutsche Forschungsgemeinschaft (DFG, German Research Foundation).

\end{funding}


\appendix
\section{Omitted proofs of Section~\ref{sec:preliminaries}}\label{app:preliminaries}

\begin{proof}[Proof of Lemma \ref{lemma:empiricalInverse}] We prove Assertions \ref{lem:ApEqualAp0}, \ref{1} and \ref{3} separately -- note that Assertion \ref{2} is consequence of \ref{1}~and~\ref{3}. %

\begin{proof}[Proof of \ref{lem:ApEqualAp0}]
Notice for $f\in L^2(\mP)$ and $g\in L^2(\mQ)$ that
$$(\AP f)(\y) = \int \xipq(\x,\y)f(\x)d\mP(\x) \quad \text{ and }\quad (\AP g)(\x) = \int \xipq(\x,\y)g(\y)d\mQ(\x).$$
Since $\xipq(\x,\y)$ is the density for $\pipq$ with respect to $\mP\otimes \mQ$, and $\pipq$ is a plan between $\mP$ and $\mQ$ we conclude by Fubini's theorem that 
\begin{align*}
    \int(\AP f)(\y) d\mQ(\y) %
    &= \int f(\x) d \pi_{\mP,\mQ}(\x,\y) = \int f(\x) d \mP(\x), \\
    \int(\AQ g)(\y) d\mP(\y) 
    &= \int g(\y) d \pi_{\mP,\mQ}(\x,\y) = \int g(\y) d \mQ(\y). 
\end{align*}
By linearity of $\AP$ and $\AQ$, the asserted properties follow.
\end{proof}

\begin{proof}[Proof of \ref{1}] 
Let $\lambda$ be an eigenvalue of $\mathcal{A}_{\mQ}\mathcal{A}_{\mP}$ (resp. $\AP\AQ$). On the one hand, Jensen's inequality and the optimality conditions yields
$\lambda^2 \|f\| \leq \|f\|  $, so $\lambda\leq 1$. On the other hand, $\lambda$ is positive by \cite[Appendix A.4]{bickel1993efficient}. Hence, since $\MP$ (resp. $\MQ$) admit eigenvalue equal to one on $L^2_0(\mP)$ (resp. $L^2_0(\mQ)$) and equal to zero on the orthogonal complement, the assertion follows from \ref{lem:ApEqualAp0}. 
\end{proof}

\begin{proof}[Proof of \ref{3}]
First note by Assertion \ref{lem:ApEqualAp0} that any constant function is an eigenvector for $\AP\AP$ with eigenvalue equal to one. By definition of $\AQ\bAP$, constant functions admit eigenvalue equal to zero. Hence, it remains to analyze the remaining eigenvalues, which correspond to orthogonal eigenvectors within $L^2_0(\mP)$. Since by Assertion \ref{lem:ApEqualAp0} the operators $\AP$ and $\bAP$ coincide on $L^2_0(\mP)$, let $f\in L^2_0(\mP)\backslash\{0\}$ be an eigenvector of $\AQ\AP$ such that for $\lambda\in [0,1]$ we have
    $$\lambda^2 f(\x') = {\int  \xi_{\mP,\mQ}(\x',\y)\int \xi_{\mP,\mQ}(\x,\y) f(\x) d\mP(\x)  d\mQ(\y)} = (\AQ\AP f)(x).  $$
    Then the pair $(g,f)$, where $g$ is defined  as 
    $$g(\y)=-\frac{1}{\lambda}\int \xi_{\mP,\mQ}(\x,\y) f(\x) d\mP(\x),$$
    satisfies
\begin{align*}
     -\lambda \begin{pmatrix}
         f\\
          g 
    \end{pmatrix} = 
    \begin{pmatrix}
          0 & \AQ \\
          \AP & 0
    \end{pmatrix}
    \begin{pmatrix}
         f \\
         g 
    \end{pmatrix}.
\end{align*}
  Hence, $(1-\lambda)$ is an eigenvalue of $\Gamma$ with associated eigenvector $(f,g)^t$. We normalize $(f,g)^t$ as $(f,g)^t/\left(\| f\|^2  +  \|g\|^2 \right)^{1/2}$, which is well-defined since $f \neq 0$. We keep the same notation for the normalized function. As a consequence,
  \begin{align*}
      (1-\lambda)&= \left\langle \Gamma\left(\begin{array}{c}
           f \\
           g 
      \end{array}\right), \left(\begin{array}{c}
           f \\
           g \end{array}\right)\right\rangle_{L^2(\mP)\times L^2( \mQ) }\\
           &= \int (f^2(\x)+(\mathcal{A}_{\mQ} g)(\x) f(\x))+   (g^2(\y)+(\mathcal{A}_{\mP} f)(\y) g(\y)) d\mP(\x) d\mQ(\y)\\
           &= \int \xi_{\mP,\mQ} (\x,\y)\left(f^2(\x)+ g(\y) f(\x)+   (g^2(\y)+f(\x) g(\y))\right) d\mP(\x) d\mQ(\y),
  \end{align*}
where the last equality is consequence of the optimality conditions. Then
    \begin{align*}
      (1-\lambda)&= \int \left(f(\x)+    g(\y)\right)^2  \xi_{\mP,\mQ} (\x,\y) d\mP(\x) d\mQ(\y) ,
  \end{align*}
By the lower bound for the density $\xi_{\mP, \mQ}$ from Lemma \ref{lem:regularity} it follows that
\begin{align*}
    \vert 1-\lambda\vert  &\geq  \exp(-3\norm{c}_\infty) \int \left(f(\x)+    g(\y)\right)^2  d\mP(\x) d\mQ(\y) \\
    &= \exp(-3\norm{c}_\infty) \left(\| f\|^2   +\|g\|^2 +2 \int f(\x) d\mP(\x) \int  g(\y)  d\mQ(\y)\right).
\end{align*}
Since $f$ is centered and $(f,g)^t$ is normalized, 
$   \vert 1-\lambda\vert    \geq  \exp(-3\norm{c}_\infty).$ 
On the other hand, \ref{1} ensures $0\leq \lambda\leq 1$. 
Therefore $1-\lambda= \vert 1-\lambda\vert    \geq  \exp(-3\norm{c}_\infty)$, which implies $\delta(c) = 1 - \exp(-3\norm{c}_\infty) \geq  \lambda  $. We thus conclude that any eigenvalue of $\AQ\bAP$ associated to an eigenvector in $L^2(\mP)$ satisfies the asserted bound. Moreover, an analogous argument also asserts an identical bound for the eigenvalues of $\bAP\AQ$.
\end{proof}
Overall, we now conclude the proof of all four assertions. 
\end{proof}

\section{Omitted proofs of Section \ref{sec:main_results}}\label{app:main}
\begin{proof}[Proof of Corollary~\ref{cor:PotIntApprox}]
Applying Theorem \ref{Theo:PotIntApprox}, with $\eta=\tilde \eta/\xipq$, gives 
\begin{align*}
    & \int \tilde \eta(\bfmn  - \bfpq)d\mPn \\
    =& \int \Big((I_{L^2_0(\mP)}-\AQ\AP)^{-1}\big(\AQ\AP \eta_{\x}-\E_{\X\sim \mP}[\eta_{\x}(\X)]\big)\Big) d (\mPn - \mP) \\
&-\int \Big((I_{L^2_0(\mQ)}-\AP\AQ)^{-1}\big(\AP \eta_{\x}-\E_{\X\sim \mP}[\eta_{\x}(\X)]\big)\Big) d (\mQm - \mQ)+ \op\left(\sqrt{\frac{n+m}{nm}}\right),
\end{align*}
Since
\begin{equation*}
    \eta_{\x}(\x)= \int \eta(\x,\y)\xi_{\mP,\mQ}(\x,\y)d\mQ(\y)=\int\tilde{\eta}(\x) d\mQ(\y)=\tilde{\eta}(\x),
\end{equation*}
then \eqref{stat1CoroApproxPoint} holds. By using a symmetric argument \eqref{stat2CoroApproxPoint} also holds. 
We prove now the last statement (when $\tilde \eta$ is constant). First we note that for a constant function $\tilde \eta  \in L^\infty(\Omega)$ it holds that $\AQ \tilde \eta=\tilde \eta$ and $\AP\tilde \eta=\tilde \eta$. Then 
$\AQ\AP \tilde \eta-\E_{\X\sim \mP}[\tilde \eta(\X)]= 0. $
By the same reasoning, it follows that$\AP \tilde \eta-\E_{\X\sim \mP}[\tilde \eta(\X)]=0$, $\AP\AQ \tilde\eta-\E_{\Y\sim \mQ}[\tilde\eta(\Y)]=0$ and $\AQ \tilde\eta-\E_{\Y\sim \mQ}[\tilde\eta(\Y)]=0$, so that only $\op\left(\sqrt{\frac{n+m}{nm}}\right)$ remains in \eqref{stat1CoroApproxPoint} and \eqref{stat2CoroApproxPoint}.
\end{proof}

\begin{proof}[Proof of Proposition \ref{prop:TCLcost}]
    We prove that 
    \begin{multline}
        \label{claimStabilityVariance}
        \operatorname{Var}_{\X\sim \mP_n}[f_{\mP_n,\mQ_m}(\X)]= {\frac 1 n}\sum_{i=1}^n f_{\mP_n,\mQ_m}^2(\X_i)-\Big({ \frac 1 n}\sum_{i=1}^n f_{\mP_n,\mQ_m}(\X_i)\Big)^2 \\\convP\operatorname{Var}_{\X\sim \mP}[f_{\mP,\mQ}(\X)].
    \end{multline}
    To do so, let us recall that Equation (4.3)  and Lemma 14 in \cite{rigollet2022sample} give the bound 
    $$ \E \left[\int  \left( f_{\mP_n,\mQ_m}-\int f_{\mP_n,\mQ_m} d\mP_n -  f_{\mP,\mQ} + \int f_{\mP,\mQ} d\mP_n\right)^2 d\mP_n\right]\leq  \frac{\cK(c)}{n}.$$
   This inequality can be rewritten as 
   \begin{equation}
       \label{boundOfRigolletVariance}
       \E\vert \operatorname{Var}_{\X\sim \mP_n}[f_{\mP_n,\mQ_m}(\X)]-\operatorname{Var}_{\X\sim\mP_n}[f_{\mP,\mQ}(\X)]\vert \leq  \frac{\cK(c)}{n}.
   \end{equation}
     The limit 
     $ \operatorname{Var}_{\X\sim \mP_n}[f_{\mP,\mQ}(\X)]\convAS \operatorname{Var}_{\X\sim \mP}[f_{\mP,\mQ}(\X)],$
     is a just application of the univariate strong law of large numbers. Therefore, via \eqref{boundOfRigolletVariance}, 
     $\operatorname{Var}_{\X\sim \mP_n}[f_{\mP_n,\mQ_m}(\X)] \convP\operatorname{Var}_{\X\sim \mP}[f_{\mP,\mQ}(\X)]$, so that \eqref{claimStabilityVariance} holds. A symmetric argument proves the same statement for $\mQ$, i.e., 
     $  \operatorname{Var}_{\Y\sim \mQ_m}[g_{\mP_n,\mQ_m}(\Y)] \convP  \operatorname{Var}_{\Y\sim \mQ}[g_{\mP,\mQ}(\Y)]$.  This concludes the proof of the first claim of this theorem. The second one holds via Slutsky’s theorem.     
\end{proof}

\begin{proof}[Proof of Proposition \ref{Proposition:VariancePlans}]
We only prove that 
\begin{multline}\label{eq:varianceConvergence}
 \hat{V}_{n,m}:=\operatorname{Var}_{\X\sim \mPn} \left[\sum_{k = 0}^N(\AQn\APn)^{k}\big(\hat\eta_{\x,n,m}-\AQn\hat\eta_{\y,n,m}\big)(\X)\right] \\ \convP V:=\operatorname{Var}_{\X\sim \mP}\left( (I_{L^2_0(\mP)}-\AQ\AP)^{-1}\big(\eta_{\x}-\AQ\eta_{\y}\big)(\X)\right).
\end{multline}
The second assertion then follows by Slutsky's Theorem. 
For the proof of \eqref{eq:varianceConvergence} we follow the following steps:
\begin{enumerate}
    \item \label{step0} We show that the truncation parameter $N = N(n,m)\in \NN\cup \{\infty\}$ can be selected arbitrarily, provided that $N(n,m)\to \infty$ for $n,m\to \infty$ with $\frac{m}{n+m}\to \lambda\in(0,1)$.
    \item \label{step1}  We compare $\hat{V}_{n,m}$ with
$$  \hat{U}_{n,m}:=\operatorname{Var}_{\X\sim \mPn} \left[\sum_{k = 0}^N(\AQn\APn)^{k}\big(\eta_{\x}-\AQn\eta_{\y}\big)(\X)\right].$$
 \item \label{step2}  We compare $\hat{U}_{n,m}$ with 
 $$ \hat{V}_{n,m}^{\mathcal{D}}:=\operatorname{Var}_{\X\sim \mPn} \left[\sum_{k = 0}^N(\DQm\DPn)^{k}\big(\eta_{\x}-\AQm\eta_{\y}\big)(\X)\right]. $$
  \item \label{step2.5}  We compare $\hat{V}_{n,m}^{\mathcal{D}}$ with 
 $$ \hat{V}_{n,m}^{\mathcal{D}'}:=\operatorname{Var}_{\X\sim \mPn} \left[\sum_{k = 0}^N(\DQm\DPn)^{k}\big(\eta_{\x}-\AQ\eta_{\y}\big)(\X)\right]. $$
 \item  \label{step3} We compare $\hat{V}_{n,m}^{\mathcal{D}'}$ with 
 $$ \hat{W}_{n,m}:=\operatorname{Var}_{\X\sim \mPn} \left[\sum_{k = 0}^N(\AQ\AP)^{k}\big(\eta_{\x}-\AQ\eta_{\y}\big)(\X)\right].$$
  \item \label{step4}  We compare $\hat{W}_{n,m}$ with 
  $$ V^{(N)}:=\operatorname{Var}_{\X\sim \mP} \left[\sum_{k = 0}^N(\AQ\AP)^{k}\big(\eta_{\x}-\AQ\eta_{\y}\big)(\X)\right].$$
  \item \label{step5}   We compare $V^{(N)}$ with $V$.
\end{enumerate}

\emph{Step \ref{step0} - Independence of truncation parameter.} Consider deterministic truncation sequences $N=N(n,m), N'=N'(n,m)\in \NN\cup\{ \infty\}$ with $N, N'\rightarrow \infty$ under $n,m\to \infty$ with $\frac{m}{n+m}\rightarrow \lambda\in (0,1)$, and define $\hat V_{n,m}'$ analogous to $\hat V_{n,m}$ but with $N(n,m)$ replaced by $N'(n,m)$. In what follows we show that $\hat V_{n,m} - \hat V_{n,m}' \convP 0$. Without loss of generality assume that $N(n,m) < N'(n,m)$. Further, note that 
\begin{multline*}
   \int \hat \eta_{\x,n,m} d\mPn =\int \eta \ximn d(\mPn \otimes \mQm) \\
   =\int  \eta(\x,\y) \ximn(\x,\y) \ximn(\x',\y) d\mPn(\x) d\mPn(\x') d \mQm(\y)
   = \int \AQm \hat \eta_{\y,n,m} d\mPn 
\end{multline*}
and thus $\hat \eta_{\x,n,m}  - \AQm\hat \eta_{\y,n,m}\in L^2_0(\mPn)$ with $L^2_0(\mPn)$-norm bounded by  $2\norm{\eta}_{\infty}$. Hence, it follows using Lemma~\ref{lemma:empiricalInverse} that
\begin{align*}
    &\norm{\sum_{k = N+1}^{N'}(\AQn\APn)^{k}\big(\hat\eta_{\x,n,m}-\AQn\hat\eta_{\y,n,m}\big)}_{L^2_0(\mPn)}\\
     \leq & \sum_{k = N+1}^\infty \norm{\AQm\APn}^k_{L^2_0(\mPn)}(\norm{\hat \eta_{\x,n,m}-\AQn\hat \eta_{\y,n,m}}_{L^2_0(\mPn)})\\
     \leq & \,\frac{\delta(c)^{N+1}}{1 - \delta(c)}( 2\norm{\eta}_\infty)
\end{align*}
for some $\delta(c)\in (0,1)$ which only depends on $c$. As $\delta(c)<1$, it follows that asymptotically the right-hand side (deterministically) tends to zero, which shows that $\hat V_{n,m} - \hat V_{n,m}' \convP 0$. With this insight at our disposal, we will assume throughout the remainder of the proof that $N= N(n,m)=\Nnm$.

\emph{Step \ref{step1} - $\hat{V}_{n,m}- \hat{U}_{n,m}\convP 0$.} Set 
$$
  K_{N,n,m}=  \norm{\sum_{k = 0}^N(\AQn\APn)^{k}}_{L^2_0(\mPn)}\left( \norm{(\eta_{\x}-\hat\eta_{\x,n,m})}_{L^2_0(\mPn)} +\norm{\AQm(\eta_{\y}-\hat\eta_{\y,n,m})}_{L^2_0(\mPn)} \right)
$$
and note that $K_{N,n,m}\leq \cK (c, \eta)$, for all $N,n,m$. This fact holds via Lemma~\ref{lemma:empiricalInverse}, $\|\eta\|_{\infty}<\infty$ and \eqref{eq:unifControlAQ}.
The inequality $|a^2-b^2|\leq |a+b||a-b|$ and Lemma~\ref{lemma:empiricalInverse} yield the bound 
\begin{align*}
    \vert \hat{V}_{n,m}&- \hat{U}_{n,m}\vert   \\
   & \leq K_{N,n,m}\norm{\sum_{k = 0}^N(\AQn\APn)^{k}}_{L^2_0(\mPn)}\norm{(\eta_{\x}-\hat\eta_{\x,n,m})-\AQm(\eta_{\y}-\hat\eta_{\y,n,m})}_{L^2_0(\mPn)}\\
     & \leq K_{N,n,m}\sum_{k = 0}^N\norm{(\AQn\APn)}_{L^2_0(\mPn)}^{k}\left(\norm{\eta_{\x}-\hat\eta_{\x,n,m}}_n+\norm{\AQm(\eta_{\y}-\hat\eta_{\y,n,m})}_{n}\right)\\
     & \leq \frac{K_{N,n,m}}{1- \delta(c)}\left(\norm{\eta_{\x}-\hat\eta_{\x,n,m}}_n+\norm{\AQm(\eta_{\y}-\hat\eta_{\y,n,m})}_{n}\right),
\end{align*}
with $\delta(c)\in (0,1)$. %
Lemma~\ref{lemma:empiricalInverse} also implies the upper bound 
\begin{equation}\label{eq:unifControlAQ}
    \sup_{\|f\|_m\leq 1}\norm{\AQn f}_{n}\leq \norm{f}_{m}.
\end{equation}
As a consequence, 
$$   \vert \hat{V}_{n,m}- \hat{U}_{n,m}\vert \leq \frac{K_{N,n,m}}{1- \delta(c)}\left(\norm{\eta_{\x}-\hat\eta_{\x,n,m}}_n+\norm{\eta_{\y}-\hat\eta_{\y,n,m}}_{m}\right).$$
We show that $\norm{\eta_{\x}-\hat\eta_{\x,n,m}}_n$ tends to $0$ in probability. The same holds for $\norm{\eta_{\y}-\hat\eta_{\y,n,m}}_{m}$. We take expectations in 
\begin{align*}
    \norm{\eta_{\x}-\hat\eta_{\x,n,m}}_n^2&=\frac{1}{n}\sum_{i=1}^n  \left(\hat\eta_{\x,n,m}(\X_i) -\eta_{\x}(\X_i)\right)^2\\
    &= \frac{1}{n}\sum_{i=1}^n\left(\int \eta(\X_i,\y)d(\pinm-\pipq)(\y)\right)^2 
\end{align*}
to obtain 
$$ \E\left[\norm{\eta_{\x}-\hat\eta_{\x,n,m}}_n^2\right]=\E\left[ \left(\int \eta(\X_i,\y)d(\pinm-\pipq)(\y)\right)^2\right].$$
Then we prove that 
$$\E\left[ \left(\int \eta(\X_i,\y)(\ximn(\X_i,\y)-\xipq(\X_i,\y) d\mQm(\y) \right)^2\right]$$ and $$\E\left[ \left(\int \eta(\X_i,\y)\xipq(\X_i,\y)  d(\mQm-\mQ)(\y) \right)^2\right]$$ tend to $0$. The first limit holds via Jensen's inequality and \cite[Theorem 5]{rigollet2022sample}:
\begin{multline*}
    \E\left[ \left(\int \eta(\X_i,\y)(\ximn(\X_i,\y)-\xipq(\X_i,\y)) d\mQm(\y) \right)^2\right]\\\leq  \norm{ \eta}_{\infty}^2 \E \left[ \int (\ximn(\X_i,\y)-\xipq(\X_i,\y))^2 d\mQm(\y)\right] \leq \cK(c,\eta)\left( \frac{n+m}{n\, m}\right).
\end{multline*}
The second limit follows by using independence, i.e., 
$$
    \E\left[ \left(\int \eta(\X_i,\y)\xipq(\y)  d(\mQm-\mQ)(\y) \right)^2\right]= \frac{1}{m}\E\left[ \left( \eta(\X_i,\Y_j)\xipq(\X_i,\Y_j) \right)^2\right],
$$
and the fact that $\eta$ and $\xipq$ are bounded.

\emph{Step \ref{step2} - $ \hat{U}_{n,m}-\hat{V}_{n,m}^{\mathcal{D}}\convP 0$.} 
The proof is based on the  operator exchange trick.  More precisely,
\eqref{eq:boundDifferencePowersDPnApn} states
\begin{multline*}
    \norm{\sum_{k=0}^{N}(\DQm\DPn)^k - (\AQm\APn)^k}_{n}\\ \leq  \sum_{k=0}^{N} \norm{(\DQm\DPn)^k - (\AQm\APn)^k}_{n} \leq \cK(c)^{N}\norm{\xipq - \ximn}_{n\times m},
\end{multline*}
so that, with $N$ chosen as in Step~\ref{step0}, we obtain by \citet[Theorem~5]{rigollet2022sample} that asymptotically, 
\begin{multline*}
    \Bigg\vert  \operatorname{Var}_{\X\sim \mPn} \left[\sum_{k = 0}^N(\AQn\APn)^{k}\big(\eta_{\x}-\AQn \eta_{\y}\big)(\X)\right]\\ -\operatorname{Var}_{\X\sim \mPn} \left[\sum_{k = 0}^N(\DQm\DPn)^{k}\big(\eta_{\x}-\AQm \eta_{\y}\big)(\X)\right]\Bigg\vert \convP 0.
\end{multline*}

\emph{Step \ref{step2.5} - $ \hat{V}_{n,m}^{\mathcal{D}}- \hat{V}_{n,m}^{\mathcal{D}'}\convP 0$}. Again employing the bound $|a^2-b^2|\leq |a+b||a-b|$ and the following inequalities (derived from Lemma~\ref{LemmaBoundOfEigD}):
\begin{multline*}
    \hat{V}_{n,m}^{\mathcal{D}}\leq \sum_{k=0}^N\norm{(\DQm\DPn)^{k}\big(\eta_{\x}-\AQm \eta_{\y}\big)}_n \\ \leq \norm{\eta_{\x}-\AQm \eta_{\y}}_n  \sum_{k=0}^N \cK(c)^k 
    \leq \cK(c) \norm{\eta}_{n\times m}\frac{1-\cK(c)^N}{1-\cK(c)} 
\end{multline*}
and 
$  \hat{V}_{n,m}^{\mathcal{D}'} \leq\norm{\eta}_{n\times m}\frac{1-\cK(c)^N}{1-\cK(c)},$
we obtain using reverse triangle inequality 
\begin{align*}
    &\vert \hat{V}_{n,m}^{\mathcal{D}}-\hat{V}_{n,m}^{\mathcal{D}'}\vert \\
    &\leq  \cK(c,\eta) \frac{1-\cK(c)^N}{1-\cK(c)} \Bigg\vert \norm{\sum_{k = 0}^N(\DQm\DPn)^{k}\big(\eta_{\x}-\AQm \eta_{\y}\big)}_{L_0^2(\mPn)}\\
    & \qquad \qquad \qquad \qquad \qquad \qquad -\norm{\sum_{k = 0}^N(\DQm\DPn)^{k}\big(\eta_{\x}-\AQ\eta_{\y}\big)}_{L_0^2(\mPn)}\Bigg\vert\\
    &\leq \cK(c,\eta) \frac{1-\cK(c)^N}{1-\cK(c)} \norm{\sum_{k = 0}^N(\DQm\DPn)^{k}(\AQm-\AQ) \eta_{\y}}_{L_0^2(\mPn)}.
\end{align*}
Therefore, Lemma~\ref{LemmaBoundOfEigD} implies 
\begin{align*}
    \vert \hat{V}_{n,m}^{\mathcal{D}}-\hat{V}_{n,m}^{\mathcal{D}'}\vert 
    &\leq \cK(c,\eta) \left(\frac{1-\cK(c)^N}{1-\cK(c)}\right)^2 \norm{(\AQm-\AQ) \eta_{\y}}_{n}\\
    &\leq \cK(c,\eta) \left(\frac{1-\cK(c)^N}{1-\cK(c)}\right)^2 (\norm{(\DQm-\AQm) \eta_{\y}}_{n}+\norm{(\DQm-\AQ) \eta_{\y}}_{n})\\
    &\leq \cK(c,\eta) \left(\frac{1-\cK(c)^N}{1-\cK(c)}\right)^2 (\norm{(\xipq-\ximn) \eta_{\y}}_{n\times m}+\norm{(\DQm-\AQ) \eta_{\y}}_{n})\\
    &\leq \cK(c,\eta) \left(\frac{1-\cK(c)^N}{1-\cK(c)}\right)^2 \norm{\eta}_{\infty}(\norm{\xipq-\ximn }_{n\times m}+\norm{(\DQm-\AQ) \eta_{\y}}_{n}).
\end{align*}
The first term of the sum, $\norm{\xipq-\ximn }_{n\times m}$, is upper bounded by $ \cK(c)\frac{n+m}{n\, m}$ (see \citet[Theorem~5]{rigollet2022sample}). To study the second one, $\norm{(\DQm-\AQ) \eta_{\y}}_{n}$,  we note that, since the subsequent terms are uncorrelated for $s\neq j$, i.e.,
\begin{multline*}
     \E\bigg[\left(\xipq(\X_1,\Y_j)\eta_{\y}(\Y_j)-\E [\xipq(\X_1,\Y_j)\eta_{\y}(\Y_j)\vert \X_1] \right)\\
\left(\xipq(\X_1,\Y_s)\eta_{\y}(\Y_s)-\E [\xipq(\X_1,\Y_s)\eta_{\y}(\Y_s)\vert \X_1] \right)\bigg]=0,
\end{multline*}
it follows that
\begin{align*}
    &\E[\norm{(\AQm-\AQ) \eta_{\y}}_{n}^2]\\
    &=\frac{1}{m^2}\E\left[ \left(\sum_{j=1}^m\xipq(\X_1,\Y_j)\eta_{\y}(\Y_j)-\E [\xipq(\X_1,\Y_j)\eta_{\y}(\Y_j)\vert \X_1] \right)^2 \right]\\
    &=\frac{1}{m}\E\left[ \left(\xipq(\X_1,\Y_j)\eta_{\y}(\Y_j)-\E [\xipq(\X_1,\Y_j)\eta_{\y}(\Y_j)\vert \X_1] \right)^2\right]\leq \frac{\cK(c,\eta)}{m}.
\end{align*}
As a consequence, 
$$ \left\vert \hat{V}_{n,m}^{\mathcal{D}}-\E[\hat{V}_{n,m}^{\mathcal{D}'} ]\right\vert
    \leq \frac{\cK(c,\eta)}{m}\left(\frac{1-\cK(c)^N}{1-\cK(c)}\right)^2, $$
    which tends to $0$ by choice of $N$ from Step \ref{step0}. 
    
\emph{Step \ref{step3} - $ \hat{V}_{n,m}^{\mathcal{D}'}- \hat{W}_{n,m}\convP 0$.} Define the quantity 
\begin{align*}
   C_N= \norm{ \sum_{k = 0}^N((\AQ\AP)^{k}\big(\eta_{\x}-\AQ\eta_{\y}\big)}_n+\norm{ \sum_{k = 0}^N(\DQm\DPn)^k\big(\eta_{\x}-\AQ\eta_{\y}\big)}_n
\end{align*}
and observe that  
\begin{equation*}
    C_N\leq  \cK(\eta) \frac{1-\cK(c)^{2N}}{1-\cK(c)}.
\end{equation*}
Then 
\begin{align*}
    \vert \hat{V}_{n,m}^{\mathcal{D}'}- \hat{W}_{n,m} \vert   &\leq  C_N\norm{ \sum_{k = 0}^N((\AQ\AP)^{k}-(\DQm\DPn)^k)\big(\eta_{\x}-\AQ\eta_{\y}\big)}_n\\
    &\leq  C_N \sum_{k = 0}^N \norm{((\AQ\AP)^{k}-(\DQm\DPn)^k)\big(\eta_{\x}-\AQ\eta_{\y}\big) }_n\\
    &\leq   C_N\sum_{k = 0}^N\norm{((\AQ\AP)^{k}-(\DQm\DPn)^k)\big(\eta_{\x}-\AQ\eta_{\y}\big)}_n\\
    &\leq   C_N\sum_{k = 0}^N k \cK(c)^k \norm{(\AQ\AP-\DQm\DPn)\big(\eta_{\x}-\AQ\eta_{\y}\big) }_n\\
    &\leq   N(C_N)^2\norm{(\AQ\AP-\DQm\DPn)\big(\eta_{\x}-\AQ\eta_{\y}\big)}_n.
\end{align*}
We show  now that 
\begin{equation}
    \label{step3Final}
    \norm{(\AQ\AP-\DQm\DPn)\big(\eta_{\x}-\AQ\eta_{\y}\big)}_n=\Op\left( \sqrt{\frac{n+m}{n\,m}}\right).
\end{equation}
To do so, we write
\begin{align*}
  \E  &\norm{(\AQ\AP-\DQm\DPn)\big(\eta_{\x}-\AQ\eta_{\y}\big)}_n^2\\
  &=  \E[((\AQ\AP-\DQm\DPn)\big(\eta_{\x}-\AQ\eta_{\y}\big)(\X_i))^2]\\
  &=  \E\left[ \left(\int \xipq(\X_i,\y) \xipq(\x,\y)  \big(\eta_{\x}-\AQ\eta_{\y}\big)(\x) d(\mPn\otimes \mQm-\mP\otimes \mQ)(\x,\y)\right)^2\right].
\end{align*}
We denote $ p(\X_i,\x,\y): =\xipq(\X_i,\y) \xipq(\x,\y)  \big(\eta_{\x}-\AQ\eta_{\y}\big)(\x)$ and observe that, by using independence, we have 
\begin{align*}
    \E&\left[ \left(\int \xipq(\X_i,\y) \xipq(\x,\y)  \big(\eta_{\x}-\AQ\eta_{\y}\big)(\x) d(\mPn\otimes \mQm-\mP\otimes \mQ)(\x,\y)\right)^2\right]\\
    &= \frac{1}{(n\, m)^2}\sum_{s=1}^n \sum_{j=1}^m\E\left[ \left(p(\X_1,\X_s,\Y_j)-\E[p(\X_1,\X_s,\Y_j)\vert \X_1]\right)^2\right]
    \leq \frac{\cK(c, \eta)}{n\, m}.
\end{align*}
As a consequence \eqref{step3Final} holds and 
$$  \E\vert \hat{V}_{n,m}^{\mathcal{D}'}- \hat{W}_{n,m} \vert \leq  \frac{\cK(c, \eta) }{\sqrt{n\, m}}  N \frac{1-\cK(c)^{2N}}{1-\cK(c)},$$
which tends to zero by our choice of $N$ from Step \ref{step0}.

\emph{Step \ref{step4} - $ \hat{W}_{n,m}
 - V^{(N)}\convP 0$.} We call $f_N(\x)= \sum_{k = 0}^N(\AQ\AP)^{k}\big(\eta_{\x}-\AQ\eta_{\y}\big)(\x))$ and note that 
 \begin{align*}
     &\vert \hat{W}_{n,m}- V^{(N)}\vert \leq \Bigg\vert\int f_N^2 d(\mPn-\mP) \Bigg\vert
     +  \left\vert \left( \int f_N d \mPn\right)^2  -\left( \int f_N  d \mP\right)^2\right\vert .
 \end{align*}
We analyze the first term by noticing that $\|f_N\|_{\infty}\leq \cK(c,\eta) \frac{1-\cK(c)^N}{1-\cK(c)}$ (see Lemma~\ref{LemmaBoundOfEigD}) implies 
\begin{align*}
    \E\left[\left(\int f_N^2 d(\mPn-\mP) \right)^2\right]&=\frac{1}{n} \E\left[f_N^4(\X_1) \right] \leq  \frac{\cK(c,\eta)^4}{n} \left(\frac{1-\cK(c)^N}{1-\cK(c)}\right)^4.
\end{align*}
Our choice of $N$ form Step \ref{step0} then implies the convergence towards $0$ of the first term, i.e., $$\E\left[\left(\int f_N^2 d(\mPn-\mP) \right)^2\right]\to 0.$$
To prove the same convergence of the second term we use  the formula $|a^2-b^2|\leq |a+b||a-b|$ to obtain 
\begin{align*}
    \left\vert \left( \int f_N d \mPn\right)^2  -\left( \int f_N  d \mP\right)^2\right\vert &\leq \left\vert   \int f_N d \mPn  +\int f_N  d \mP\right\vert \left\vert \left( \int f_N d (\mPn-  \mP)\right)\right\vert \\
    & \leq 2\cK(c,\eta) \frac{1-\cK(c)^N}{1-\cK(c)} \left\vert  \int f_N d (\mPn-  \mP)\right\vert .
\end{align*}
As a consequence, 
$$  \E\left[\left( \left( \int f_N d \mPn\right)^2  -\left( \int f_N  d \mP\right)^2\right)^2\right] \to 0
 $$ 
 by repeating the previous arguments. 
 
\emph{ Step \ref{step5} - $ V^{(N)}-V
 \convP 0$.}
 We use the formula $|a^2-b^2|\leq |a+b||a-b|$  and the fact that $V^{(N)}\leq V\leq \cK(c,\eta)$ to obtain
 \begin{align*}
     \vert V^{(N)}-V\vert &\leq \cK(c,\eta) \left(\operatorname{Var}_{\X\sim \mP} \left[\sum_{k =N+1 }^{\infty}(\AQ\AP)^{k}\big(\eta_{\x}-\AQ\eta_{\y}\big)(\X)\right]\right)^{\frac{1}{2}}\\
     &\leq \cK(c,\eta) \left( \sum_{k =N +1}^{\infty}\norm{ \AQ\AP}_{L^2_0(\mP)}^{k}\norm{\eta_{\x}-\AQ\eta_{\y}}_{L^2_0(\mP)}\right)^{\frac{1}{2}}.
 \end{align*}
Lemma~\ref{lemma:empiricalInverse}  yields 
$$ \vert V^{(N)}-V\vert^2\leq   \cK(c,\eta)^2 \sum_{k=N+1}^{\infty} \delta(c)^k ,$$
for $\delta(c)\in (0,1)$. As a consequence $\vert V^{(N)}-V\vert\to 0$ as $N\to \infty$, irrespective of the rate between $N$ and $n$. 
\end{proof}
\begin{proof}[Proof of Proposition \ref{Proposition:VarianceCondPlans}]
    The proof for the convergence of the variance estimator is analogous to the proof of Proposition \ref{Proposition:VariancePlans} with one modification: instead of showing that $$\norm{\hat \eta_{\x,n,m} - \eta_{\x}}_{L^2(\mPn)} + \norm{\hat \eta_{\y,n,m} - \eta_{\y}}_{L^2(\mPn)} = \mathcal{O}_p\left(\sqrt{\frac{n+m}{nm}}\right)$$ we need to show that \begin{align*}
        &\norm{\eta(\x,\cdot) \ximn(\x,\cdot) - \eta(\x,\cdot)\xipq(\x,\cdot)}_{L^2(\mQm)} \\
        &+ \norm{\hat\eta_{\x,n,m}(\x) \ximn(\x,\cdot) - \eta_{\x}(\x)\xipq(\x,\cdot)}_{L^2(\mQm)}
 = \mathcal{O}_p\left(\sqrt{\frac{n+m}{nm}}\right).
    \end{align*}
    We start with the first term, for which we consider the upper bound
    \begin{align*}
         \norm{\eta(\x,\cdot) \ximn(\x,\cdot) - \eta(\x,\cdot)\xipq(\x,\cdot)}_{L^2(\mQm)} &\leq \norm{\eta}_{\infty} \norm{\ximn(\x,\cdot) - \xipq(\x,\cdot)}_{L^2(\mQm)}.
    \end{align*}
    The difference $\ximn(\x,\cdot) - \xipq(\x,\cdot)$ can be rewritten as 
    \begin{align*}
        \ximn(\x,\cdot) - \xipq(\x,\cdot) &= \left[ \exp(\bfmn(\x) + \bgmn(\cdot))  - \exp(\bfpq(\x) + \bgpq(\cdot))\right]\exp(-c(\x,\cdot))%
    \end{align*}
    The mean value theorem states
    $$ \exp(x)-\exp(y)\leq \exp(\vert x\vert+\vert y\vert) \vert x-y\vert.$$
    Then 
     \begin{align*}
      \vert   \ximn(\x,\cdot) - \xipq(\x,\cdot)\vert  &\leq  \cK(c) \vert \bfmn(\x) - \bfpq(\x)\vert +\vert  \bgmn(\cdot)- \bgpq(\cdot)\vert
    \end{align*}
so that
\begin{align*}
      \norm{  \ximn(\x,\cdot) - \xipq(\x,\cdot)}_{m}  &\leq  \cK(c) \vert \bfmn(\x) - \bfpq(\x)\vert +\norm{  \bgmn(\cdot)- \bgpq(\cdot)}_m.
    \end{align*}
Therefore, Lemma \ref{lem:pointwiseConvergence} and \citet[Inequality (4.3) and Lemma~14]{rigollet2022sample} implies 
\begin{align}\label{eq:convergencePointwiseDensity}
    \E[\norm{  \ximn(\x,\cdot) - \xipq(\x,\cdot)}_{m}^2]\leq  \cK(c)\left(\frac{n+m}{n\, m}\right).
\end{align} 

For the second term we note that 
\begin{align*}
   & \norm{\hat\eta_{\x,n,m}(\x) \ximn(\x,\cdot) - \eta_{\x}(\x)\xipq(\x,\cdot)}_{L^2(\mQm)}\\ \leq & \norm{[\hat\eta_{\x,n,m}(\x) - \eta_{\x}(\x)]
(\ximn(\x,\cdot)}_{L^2(\mQm)} +  \norm{\eta_{\x}(\x)[\ximn(\x,\cdot) - \xipq(\x,\cdot)]}_{L^2(\mQm)}\\
& \leq \cK(c) |\hat\eta_{\x,n,m}(\x) - \eta_{\x}(\x)| + \norm{\eta}_\infty \norm{  \ximn(\x,\cdot) - \xipq(\x,\cdot)}_{m}.
\end{align*}
Both terms in the upper bound are of order $\mathcal{O}_p\left(\sqrt{\frac{n+m}{nm}}\right)$; this follows by adapting the proof of Theorem 4 of \cite{rigollet2022sample} for general $\eta(\cdot,\y)$, while convergence of the second term is treated in \eqref{eq:convergencePointwiseDensity}.
\end{proof}
\begin{proof}[Proof of Lemma \ref{lem:pointwiseConvergence}]
By symmetry of the argument it suffices to only provide the proof for the first assertion. Using the optimality condition \eqref{eq:optimallityCodt0} it holds, 
\begin{align*}
     &\left|\exp(-\bfmn(\x)) - \exp(-\bfpq(\x)) \right|\\
     =&\left|\int e^{\bar g_{\mPn,\mQm}(\y)-c(\x,\y)}d\mQn(\y)-\int e^{\bar g_{\mP,\mQ}(\y)-c(\x,\y)}d\mQ(\y)\right| \\
    \leq & \left|\int \left(e^{\bar g_{\mPn,\mQm}(\y)} - e^{\bar g_{\mPn,\mQm}(\y)}\right)C(\x,\y) d\mQn(\y)\right| + \left|\int e^{\bar g_{\mP,\mQ}(\y)-c(\x,\y)}d(\mQn-\mQ)(\y) \right|.
\end{align*}
Based on the quantitative bounds for optimal potentials (Lemma \ref{lem:regularity}), the first is upper bounded by $\exp(2 \cK(c) + \norm{c}_\infty) \norm{\bgmn - \bgpq}_{L^2(\mQm)}$, which tends to zero in probability by Equation (4.3) and Lemma~14 in \cite{rigollet2022sample}, while the second term tends to zero in probability by the weak law of large numbers.
\end{proof}

\section{Omitted proofs of Section \ref{sec:Applications}}\label{app:proofsApplications}
\begin{proof}[Proof Corollary~\ref{Coro:Kernel}]
We decompose $K_n(\mP_n,\mQ_m)-K(\mP,\mQ)=(K_n(\mP_n,\mQ_m)-K_n(\mP,\mQ))+(K_n(\mP,\mQ)-K(\mP,\mQ))$ and study each term separately. We start with the second one. We call $h_{\mathcal{U}}= (g_{\mP,\mathcal{U}}- g_{\mQ, \mathcal{U}})-\E_{\U\sim\mathcal{U}}[g_{\mP,\mathcal{U}}(\U)- g_{\mQ, \mathcal{U}}(\U)]$ and observe that
\begin{align*}
   &  \operatorname{Var}_{\U\sim\mathcal{U}_n}[g_{\mP,\mathcal{U}}(\U)-g_{\mQ, \mathcal{U}}(\U)]-\operatorname{Var}_{\U\sim\mathcal{U}}[g_{\rm P,\mathcal{U}}(\U)-g_{\rm Q, \mathcal{U}}(\U)]) \\&=
        \frac{1}{n} \sum_{i=1}^n (h^2(\U_i)-\E[ h^2(\U_i)])- \left(\frac{1}{n}\sum_{i=1}^n h(\U_i)\right)^2+\left( \E[h(\U_i)]\right)^2
        \\&=
       \frac{1}{n} \sum_{i=1}^n (h^2(\U_i)-\E[ h^2(\U_i)])+\left(\E[h(\U_i)]- \frac{1}{n}\sum_{i=1}^n h(\U_i)\right)\left(\frac{1}{n}\sum_{i=1}^n h(\U_i)+ \E[h(\U_i)]\right).
\end{align*}
Hence, by Slutsky's theorem (note that $h$ is centered) and the delta-method, we obtain 
\begin{multline}
    \label{firsterm}
    K_n(\mP,\mQ)-K(\mP,\mQ)\\=F'( \operatorname{Var}_{\U\sim\mathcal{U}}[g_{\mP,\mathcal{U}}(\U)-g_{\mQ, \mathcal{U}}(\U)])\left(\frac{1}{n} \sum_{i=1}^n (h^2(\U_i)-\E[ h^2(\U_i)])\right)+\op(1/n).
\end{multline}
To analyze the second term, let $\bar{g}_{\nu,\mathcal{U}_n}$  be such that $\int\bar{g}_{\nu,\mathcal{U}_n} d \mathcal{U}_n=0$, for $\nu\in \{\mP,\mQ,\mPn,\mQm\}$. Moreover, in the subsequent computation we will abbreviate the norm $\norm{\cdot}_{L^2(\mathcal{U}_n)}$ by $\norm{\cdot}$ and the  inner product $\langle \cdot, \cdot\rangle_{L^2(\mathcal{U}_n)}$ by $\langle \cdot, \cdot \rangle$. With this notation, a straight-forward computation asserts
    \begin{align*} 
        &\|\bar{g}_{\mPn,\mathcal{U}_n}-\bar{g}_{\mP,\mathcal{U}} \|^2\\
        &=\|\bar{g}_{\mPn,\mathcal{U}_n}-\bar{g}_{\mP,\mathcal{U}} \|^2+\langle \bar{g}_{\mPn,\mathcal{U}_n}-\bar{g}_{\mP,\mathcal{U}}, \bar{g}_{\mP,\mathcal{U}}-\bar{g}_{\mQ,\mathcal{U}}\rangle+ \langle \bar{g}_{\mPn,\mathcal{U}_n}-\bar{g}_{\mP,\mathcal{U}}, \bar{g}_{\mQ,\mathcal{U}}-\bar{g}_{\mQm,\mathcal{U}_n}\rangle\\
        &\  +\langle \bar{g}_{\mP,\mathcal{U}}-\bar{g}_{\mQ,\mathcal{U}}, \bar{g}_{\mPn,\mathcal{U}_n}-\bar{g}_{\mP,\mathcal{U}} \rangle+ \| \bar{g}_{\mP,\mathcal{U}}-\bar{g}_{\mQ,\mathcal{U}}\|^2+ \langle \bar{g}_{\mP,\mathcal{U}}-\bar{g}_{\mQ,\mathcal{U}}, \bar{g}_{\mQ,\mathcal{U}}-\bar{g}_{\mQm,\mathcal{U}_n}\rangle\\
        &\ 
        +\langle \bar{g}_{\mQ,\mathcal{U}}-\bar{g}_{\mQm,\mathcal{U}_n}, \bar{g}_{\mPn,\mathcal{U}_n}-\bar{g}_{\mQ,\mathcal{U}}\rangle+ \langle \bar{g}_{\mQ,\mathcal{U}}-\bar{g}_{\mQm,\mathcal{U}_n}, \bar{g}_{\mP,\mathcal{U}}-\bar{g}_{\mQ,\mathcal{U}}\rangle+ \| \bar{g}_{\mQ,\mathcal{U}}-\bar{g}_{\mQm,\mathcal{U}_n}\|^2\\
        &= \|\bar{g}_{\mP,\mathcal{U}}-\bar{g}_{\mQ,\mathcal{U}}\|^2+2\langle \bar{g}_{\mPn,\mathcal{U}_n}-\bar{g}_{\mP,\mathcal{U}}, \bar{g}_{\mP,\mathcal{U}}-\bar{g}_{\mQ,\mathcal{U}}\rangle +2\langle \bar{g}_{\mQ,\mathcal{U}}-\bar{g}_{\mQm,\mathcal{U}_n}, \bar{g}_{\mP,\mathcal{U}}-\bar{g}_{\mQ,\mathcal{U}}\rangle\\
       & \ +\op\left( \sqrt{\frac{n+m}{n\, m}}\right),
    \end{align*}
    where for the second equality we used the Cauchy-Schwarz inequality combined with the convergence result of empirical potentials detailed in Equation (4.3) and Lemma 14 of \cite{rigollet2022sample}.
  Corollary~\ref{cor:PotIntApprox} applied to $\tilde\eta=(\bar{g}_{\mP,\mathcal{U}}-\bar{g}_{\mQ,\mathcal{U}})$ gives (note that $\tilde\eta$ is centered)
  \begin{align*}
       &\|\bar{g}_{\mPn,\mathcal{U}_n}-\bar{g}_{\mP,\mathcal{U}} \|^2-\|\bar{g}_{\mP,\mathcal{U}}-\bar{g}_{\mQ,\mathcal{U}}\|^2\\
        &= 
        2\int\bigg( \Big((I_{L^2_0(\cU)}-\AQ{\mathcal{A}}_{\mathcal{U}})^{-1}\AQ{\mathcal{A}}_{\mathcal{U}} +(1-\AP\bar{\mathcal{A}}_{\mathcal{U}})^{-1}\AP{\mathcal{A}}_{\mathcal{U}} \Big)\tilde{\eta}\bigg) d (\mathcal{U}_n - \mathcal{U})   \\  &-2\int \Big((I_{L^2_0(\mQ)}-{\mathcal{A}}_{\mathcal{U}}\AQ)^{-1}\big({\mathcal{A}}_{\mathcal{U}} \tilde \eta\big)\Big) d (\mQm - \mQ)- 2\int \Big((I_{L^2_0(\mP)}-{\mathcal{A}}_{\mathcal{U}}\AP)^{-1}\big({\mathcal{A}}_{\mathcal{U}} \tilde \eta\big)\Big) d (\mPn - \mP),
  \end{align*}
which holds up to additive $\op\left(\sqrt{\frac{n+m}{nm}}\right)$ terms. 
    A first order Taylor's development yields
    \begin{multline*}
         K_n(\mPn,\mQm)-K_n(\mP,\mQ)\\
         =F'_n\Bigg( 2\int\bigg( \Big((I_{L^2_0(\cU)}-\AQ{\mathcal{A}}_{\mathcal{U}})^{-1}\AQ{\mathcal{A}}_{\mathcal{U}} +(1-\AP\bar{\mathcal{A}}_{\mathcal{U}})^{-1}\AP{\mathcal{A}}_{\mathcal{U}} \Big)\tilde{\eta}\bigg) d (\mathcal{U}_n - \mathcal{U})   \\  -2\int \Big((I_{L^2_0(\mQ)}-{\mathcal{A}}_{\mathcal{U}}\AQ)^{-1}\big({\mathcal{A}}_{\mathcal{U}} \tilde \eta\big)\Big) d (\mQm - \mQ)\\ \qquad \qquad \qquad- 2\int \Big((I_{L^2_0(\mP)}-{\mathcal{A}}_{\mathcal{U}}\AP)^{-1}\big({\mathcal{A}}_{\mathcal{U}} \tilde \eta\big)\Big) d (\mPn - \mP)\Bigg)
         + \op\left(\sqrt{\frac{n+m}{nm}}\right)
    \end{multline*}
    with $F_n'=F'( \operatorname{Var}_{\U\sim\mathcal{U}_n}[g_{\mP,\mathcal{U}}(\U)-g_{\mQ, \mathcal{U}}(\U)])$. The continuity of $F'(\cdot)$ allows replacing $F_n'$ by $F'( \operatorname{Var}_{\U\sim\mathcal{U}}[g_{\mP,\mathcal{U}}(\U)-g_{\mQ, \mathcal{U}}(\U)])$ and obtain 
      \begin{multline*}
         K_n(\mPn,\mQm)-K_n(\mP,\mQ)
         =F'( \operatorname{Var}_{\U\sim\mathcal{U}}[g_{\mP,\mathcal{U}}(\U)-g_{\mQ, \mathcal{U}}(\U)]) \\
        \times \Bigg( 2\int\bigg( \Big((I_{L^2_0(\cU)}-\AQ{\mathcal{A}}_{\mathcal{U}})^{-1}\AQ{\mathcal{A}}_{\mathcal{U}} +(1-\AP\bar{\mathcal{A}}_{\mathcal{U}})^{-1}\AP{\mathcal{A}}_{\mathcal{U}} \Big)\tilde{\eta}\bigg) d (\mathcal{U}_n - \mathcal{U})   \\  -2\int \Big((I_{L^2_0(\mQ)}-{\mathcal{A}}_{\mathcal{U}}\AQ)^{-1}\big({\mathcal{A}}_{\mathcal{U}} \tilde \eta\big)\Big) d (\mQm - \mQ)\\ \qquad \qquad \qquad- 2\int \Big((I_{L^2_0(\mP)}-{\mathcal{A}}_{\mathcal{U}}\AP)^{-1}\big({\mathcal{A}}_{\mathcal{U}} \tilde \eta\big)\Big) d (\mPn - \mP)\Bigg)+\op\left(\sqrt{\frac{n+m}{nm}}\right),
    \end{multline*}
    which, together with \eqref{firsterm}, concludes the proof. 
\end{proof}

\section{Omitted proofs of Section \ref{sec:proofs}}\label{app:proofs}

\begin{proof}[Proof of Proposition~\ref{prop:linearPot}]

\emph{Step 1 - Perturbation analysis of $\overline F_{n,m}(\cdot, \cdot)$.}
To show that the directional derivative  is actually strong, we verify that
\begin{multline}\label{FrechetBound}
   \left\| \overline F_{n,m}\left(\begin{array}{c}
    \overline  f_{\mP_n,\mQ_m}+f  \\
    \overline  g_{\mP_n,\mQ_m}+g
\end{array} \right)- \overline F_{n,m}\left(\begin{array}{c}
    \overline  f_{\mP_n,\mQ_m} \\
    \overline  g_{\mP_n,\mQ_m}
\end{array} \right)-\overline{\Gamma}_{n,m}\left(\begin{array}{c}
     f  \\
     g
\end{array} \right)\right\|_{n\times m}
\leq \cK(c, \kappa)\left\| \left(\begin{array}{c}
     f \\
     g
\end{array} \right)\right\|_{n\times m}^2,\!\!\!\!\!
\end{multline}
for all $(f,g)^t\in L^2(\mPn)\times L^2_0(\mQm)$ such that
\begin{equation}\label{boundOnDirectins}
    \|(f,g)^t \|_{L^{\infty}(\mP_n)\times L^{\infty}(\mQ_m) }\leq \kappa.
\end{equation}
 The constant $\kappa$ is supposed to be irrespective of the sample sizes, $n$ and $m$. 
We prove this claim first for the first row. That is, 
\begin{multline*}
\bigg\|\log\left(\int e^{g_{\mP_n,\mQ_m} (\y)+g(\y)}C(\cdot,\y)d\mQ_m(\y)\right)-\log\left(\int e^{g_{\mP_n,\mQ_m} (\y)}C(\cdot,\y)d\mQ_m(\y)\right)
\\
-\int \xi_{\mP_n,\mQ_m}(\cdot, \y) g(\y) d\mQ_m 
\bigg\|_n \leq \cK(c, \kappa) \left\| \left(\begin{array}{c}
     f \\
     g
\end{array} \right)\right\|_{n\times m}^2.
\end{multline*}
To simplify the notation, we define for $i=1, \dots, n$, 
$$ M_i=\log\left(\int e^{g_{\mP_n,\mQ_m} (\y)+g(\y)}C(\X_i,\y)d\mQ_m(\y)\right)-\log\left(\int e^{g_{\mP_n,\mQ_m} (\y)}C(\X_i,\y)d\mQ_m(\y)\right).$$
The inequality 
\begin{equation}
    \label{inequalityLog}
    \left\vert \log(a)-\log(b)-\frac{a-b}{b}\right\vert \leq \frac{(a-b)^2}{\min(a^2, b^2)},
\end{equation}
 derived from the mean value theorem and valid for all $a,b>0$, yields 
 \begin{align*}
     &N_i=\left\vert M_i - \frac{\int (e^{g_{\mP_n,\mQ_m} (\y)+g(\y)}-\int e^{g_{\mP_n,\mQ_m} (\y)})C(\X_i,\y)d\mQ_m(\y)}{\int e^{g_{\mP_n,\mQ_m} (\y)}C(\X_i,\y)d\mQ_m(\y)} \right\vert \\
     &\leq 
   \left( \frac{\int (e^{g_{\mP_n,\mQ_m} (\y)+g(\y)}- e^{g_{\mP_n,\mQ_m} (\y)})C(\X_i,\y)d\mQ_m(\y)}{\min \left(\int e^{g_{\mP_n,\mQ_m} (\y)}C(\X_i,\y)d\mQ_m(\y),  \int e^{g_{\mP_n,\mQ_m} (\y)+g(\y)}C(\X_i,\y)d\mQ_m(\y)\right)}\right)^2.
 \end{align*}
 Due to Lemma~\ref{lem:regularity} and \eqref{boundOnDirectins},  we can find a constant $\cK(c,\kappa)$ such that 
  \begin{align*}
     N_i \leq \cK(c,\kappa)
   \left( {\int (e^{g_{\mP_n,\mQ_m} (\y)+g(\y)}- e^{g_{\mP_n,\mQ_m} (\y)})C(\X_i,\y)d\mQ_m(\y)}\right)^2.
 \end{align*}
 The mean value theorem yields the inequality 
 $ \left\vert e^a-e^b\right\vert \leq e^{\max(a,b)} \vert a-b\vert $ (as before, it is valid for $a,b\in \R$), which implies 
   \begin{align*}
     N_i &\leq \cK(c,\kappa) \|e^c\|_{L^{\infty}(\mP_n)\times L^{\infty}(\mQ_m)}
   \left( {\int \left\vert e^{g_{\mP_n,\mQ_m} (\y)+g(\y)}- e^{g_{\mP_n,\mQ_m} (\y)} \right\vert d\mQ_m(\y)}\right)^2\\
   &\leq \cK(c,\kappa) \|e^c\|_{L^{\infty}(\mP_n)\times L^{\infty}(\mQ_m)}
    {\int \left( e^{g_{\mP_n,\mQ_m} (\y)+g(\y)}- e^{g_{\mP_n,\mQ_m} (\y)} \right)^2 d\mQ_m(\y)}\\
   &\leq\cK(c,\kappa) e^{2\|g_{\mP_n,\mQ_m}\|_{L^{\infty}(\mQ_m)}+2\|g\|_{L^{\infty}(\mQ_m)}}
 \int g^2(\y) d\mQ_m(\y)\\
 &\leq \cK(c,\kappa) \|g\|_m^2.
 \end{align*}
The last inequality is consequence of Lemma~\ref{lem:regularity}, the assumption~\eqref{boundOnDirectins} and the boundedness assumption on the cost. 

To prove  \eqref{FrechetBound} for the second row, we set $$A_{n,m}=\log\left(\int e^{f_{\mP_n,\mQ_m} (\x)+f(\x)}C(\x,\cdot)d\mP_n(\x)\right)-\log\left(\int e^{f_{\mP_n,\mQ_m} (\x)}C(\x,\cdot)d\mP_n(\x)\right).$$  As before, we have
$$
\bigg\|A_{n,m}-\int \xi_{\mP_n,\mQ_m}(\x, \cdot) f(\x) d\mP_n(\x) 
\bigg\|_m =\cK(c,\kappa)\left\| \left(\begin{array}{c}
     f \\
     g
\end{array} \right)\right\|_{n\times m}^2.
$$
Jensen's inequality yields the bound
\begin{align*}
    \bigg\vert  \int A_{n,m} d\mQ_m-\int \xi_{\mP_n,\mQ_m}&(\x, \y) f(\x) d\mP_n(\x) d\mQ_m(\y)\bigg\vert\\
    &=\bigg\vert  \int \bigg(A_{n,m}(\y)-\int \xi_{\mP_n,\mQ_m}(\x, \y) f(\x) d\mP_n(\x) \bigg)d\mQ_m(\y)\bigg\vert\\
    &\leq \bigg\|A_{n,m}-\int \xi_{\mP_n,\mQ_m}(\x, \cdot) f(\x) d\mP_n(\x) 
\bigg\|_m,
\end{align*}
where we can conclude 
\begin{equation}
    \label{eq:boundForTheUnbiased}
     \bigg\vert  \int A_{n,m} d\mQ_m-\int \xi_{\mP_n,\mQ_m}(\x, \y) f(\x) d\mP_n(\x) d\mQ_m\bigg\vert =\cK(c)\left\| \left(\begin{array}{c}
     f \\
     g
\end{array} \right)\right\|_{n\times m}^2.
\end{equation}
Since the second row of \eqref{FrechetBound} is just $$A_{n,m}-\int \xi_{\mP_n,\mQ_m}(\x, \cdot) f(\x) d\mP_n(\x)-\left(\int A_{n,m} d\mQ_m-\int \xi_{\mP_n,\mQ_m}(\x, \y) f(\x) d\mP_n(\x) d\mQ_m(\y)\right),$$
we conclude that \eqref{FrechetBound} holds. 

\emph{Step 2 - Fluctuation analysis of $(\overline F_{n,m} - \overline F)(\bfpq, \bgpq)^t$.}
Via Lemma~\ref{lem:regularity}, the difference between the empirical and population potentials is bounded in $L^{\infty}(\mP_n)\times L^{\infty}(\mQ_m)$ by a constant depending on the cost, i.e.,  $ \|(\bar f_{\mP,\mQ}-\bar f_{\mP_n,\mQ_m}, \, \bar g_{\mP,\mQ}-\bar g_{\mP_n,\mQ_m})^t\|_{L^{\infty}(\mP_n)\times L^{\infty}(\mQ_m)}\leq \cK(c)$, for all $n,m\in \N$. Then the pair of functions $(\bar f_{\mP,\mQ}-\bar f_{\mP_n,\mQ_m}, \, \bar g_{\mP,\mQ}-\bar g_{\mP_n,\mQ_m})^t$ satisfies assumption \eqref{boundOnDirectins}. 
As a consequence, via Equation (4.3) and Lemma~14 in \cite{rigollet2022sample}, we obtain 
\begin{equation*}
  \E \left\| \overline{F}_{n,m}\left(\begin{array}{c}
     \bar f_{\mP,\mQ}  \\
     \bar g_{\mP,\mQ}
\end{array} \right)-\overline{F}_{n,m}\left(\begin{array}{c}
     \bar f_{\mP_n,\mQ_m} \\
     \bar g_{\mP_n,\mQ_m}
\end{array} \right)-\overline{\Gamma}_{n,m}\left(\begin{array}{c}
     \bar f_{\mP,\mQ}-f_{\mP_n,\mQ_m}   \\
     \bar g_{\mP,\mQ}-g_{\mP_n,\mQ_m} 
\end{array} \right)\right\|_{n\times m}\leq\frac{\cK(c)}{n}.
\end{equation*}
Since 
$$ 0= \overline{F}_{n,m}\left(\begin{array}{c}
     \bar f_{\mP_n,\mQ_m}  \\
     \bar g_{\mP_n,\mQ_m}
\end{array} \right)=\overline{F}\left(\begin{array}{c}
     \bar f_{\mP,\mQ}  \\
     \bar g_{\mP,\mQ}
\end{array} \right),$$
where 
\begin{equation*}
\overline F\left(\begin{array}{c}
     f  \\
     g
\end{array} \right)=
    \left(\begin{array}{c}
          f+\log\left(\int e^{g(\y)}C(\cdot,\y)d\mQ(\y)\right)\\
    g+\log\left(\int e^{f(\x)}C(\x,\cdot)d\mP(\x)\right) - \int \log\left(\int e^{f(\x)}C(\x,\y)d\mP(\x)\right)d \mQ(\y)
    \end{array}\right)\,,
\end{equation*}
we obtain  
\begin{equation}
    \label{FrechetBound2}
  \E \left\| \overline  F_{n,m}\left(\begin{array}{c}
 \bar f_{\mP,\mQ}  \\
 \bar g_{\mP,\mQ} 
\end{array} \right)- \overline  F\left(\begin{array}{c}
    \bar  f_{\mP,\mQ}  \\
    \bar  g_{\mP,\mQ}
\end{array} \right)-\overline 
 \Gamma_{n,m}\left(\begin{array}{c}
  \bar    f_{\mP,\mQ}-\bar f_{\mP_n,\mQ_m}   \\
    \bar  g_{\mP,\mQ}-\bar g_{\mP_n,\mQ_m} 
\end{array} \right)\right\|_{n\times m}
\leq  \frac{\cK(c)}{n} .
\end{equation}
We denote
$$ R_{n,m}^{(1)}(\cdot)=\log\left(\int e^{\bar g_{\mP,\mQ} (\y)}C(\cdot,\y)d\mQ_m(\y)\right)-\log\left(\int e^{\bar g_{\mP,\mQ} (\y)}C(\cdot,\y)d\mQ(\y)\right) $$
and 
$$ R_{n,m}^{(2)}(\cdot)=    \log\left(\int e^{\bar f_{\mP,\mQ} (\x)}C(\x,\cdot)d\mP_n(\x)\right)-\log\left(\int e^{\bar f_{\mP,\mQ} (\x)}C(\x,\cdot)d\mP(\x)  \right). $$
With this notation, 
$$
    \overline  F_{n,m}\left(\begin{array}{c}
     \bar f_{\mP,\mQ}  \\
    \bar  g_{\mP,\mQ}
\end{array} \right)- \overline  F\left(\begin{array}{c}
     \bar f_{\mP,\mQ}  \\
    \bar  g_{\mP,\mQ}
\end{array} \right)
= \left(\begin{array}{c}
          R_{n,m}^{(1)}(\cdot)\\
    R_{n,m}^{(2)}(\cdot)- \int R_{n,m}^{(2)}(\y) d\mQm(\y)
    \end{array}\right)
$$
is an element of $L^2(\mP_n)\times L^2_0(\mQ_m)$. We focus only on the first component,  the second follows as in \eqref{eq:boundForTheUnbiased}. As before (see \eqref{inequalityLog}) the logarithm can be avoided to obtain  
\begin{equation}
    \label{lineariTn}
    \E\left\| R_{n,m}^{(1)}(\cdot)-\frac{\int e^{\bar g_{\mP,\mQ} (\y)}C(\cdot,\y)d(\mQ_m-\mQ)(\y)}{\int e^{\bar g_{\mP,\mQ} (\y)}C(\cdot,\y)d\mQ(\y)}\right\|_n\!\!\!\!\leq \cK(c) \E\left\|\int e^{\bar g_{\mP,\mQ} (\y)}C(\cdot,\y)d(\mQ_m-\mQ)(\y)\right\|_n^2.
\end{equation}
Since 
$$ \frac{\int e^{\bar g_{\mP,\mQ} (\y)}C(\cdot,\y)d(\mQ_m-\mQ)(\y)}{e^{\bar g_{\mP,\mQ} (\y)}C(\cdot,\y)d\mQ(\y)}=\frac{1}{m}\sum_{j=1}^m \frac{e^{\bar g_{\mP,\mQ} (\Y_j)}C(\cdot,\Y_j)}{\int e^{\bar g_{\mP,\mQ} (\y)}C(\cdot,\y)d\mQ(\y)}-1.$$
Note that 
\begin{align*}
   \E\left[\left(\frac{e^{\bar g_{\mP,\mQ} (\Y_j)}C(\X_i,\Y_j)}{\int e^{\bar g_{\mP,\mQ} (\y)}C(\X_i,\y)d\mQ(\y)}-1\right)\left(\frac{e^{\bar g_{\mP,\mQ} (\Y_l)}C(\X_k,\Y_l)}{\int e^{\bar g_{\mP,\mQ} (\y)}C(\X_k,\y)d\mQ(\y)}-1\right)\right]=0,
\end{align*}
whenever $i\neq k$ or $j\neq l$. Then, since $\Y_1, \dots, \Y_m$ are identically distributed, 
\begin{multline*}
    \E\left[ \int{ \left(\frac{\int e^{\bar g_{\mP,\mQ} (\y)}C(\x,\y)d(\mQ_m-\mQ)(\y)}{e^{\bar g_{\mP,\mQ} (\y)}C(\x,\y)d\mQ(\y)}\right)^2}  d \mPn(\x) \right]\\=
    \frac{1}{n\, m^2}\sum_{i=1}^n \E\left[\left(\sum_{j=1}^m\frac{e^{\bar g_{\mP,\mQ} (\Y_j)}C(\X_i,\Y_j)}{\int e^{\bar g_{\mP,\mQ} (\y)}C(\X_i,\y)d\mQ(\y)}-1\right)^2\right]\\=
    \frac{1}{n\, m^2}\sum_{i,j=1}^n \E\left[\left(\frac{e^{\bar g_{\mP,\mQ} (\Y_j)}C(\X_i,\Y_j)}{\int e^{\bar g_{\mP,\mQ} (\y)}C(\X_i,\y)d\mQ(\y)}-1\right)^2\right]\\=
    \frac{1}{m} \E\left[\left(\frac{e^{\bar g_{\mP,\mQ} (\Y_1)}C(\X_i,\Y_1)}{\int e^{\bar g_{\mP,\mQ} (\y)}C(\X_i,\y)d\mQ(\y)}-1\right)^2\right].
\end{multline*}
As a consequence, 
$$ \E\left[\left\|\int e^{\bar g_{\mP,\mQ} (\y)}C(\cdot,\y)d(\mQ_m-\mQ)(\y)\right\|_n^2\right]\leq \frac{1}{m} \E\left[\left(\frac{e^{\bar g_{\mP,\mQ} (\Y_1)}C(\X_i,\Y_1)}{\int e^{\bar g_{\mP,\mQ} (\y)}C(\X_i,\y)d\mQ(\y)}-1\right)^2\right],$$
so that 
$$ \E\left[\left\|\int e^{g_{\mP,\mQ} (\y)}C(\cdot,\y)d(\mQ_m-\mQ)(\y)\right\|_n^2\right]\leq \frac{\cK(c)}{m}.$$
Then, introducing this bound in \eqref{lineariTn}  we obtain (assuming $m/(n+m)\to \lambda\in (0,1)$)
$$ 
    \E\left[\left\| R_{n,m}^{(1)}(\cdot)-\frac{\int e^{\bar g_{\mP,\mQ} (\y)}C(\cdot,\y)d(\mQ_m-\mQ)(\y)}{\int e^{\bar g_{\mP,\mQ} (\y)}C(\cdot,\y)d\mQ(\y)}\right\|_n\right]\leq \frac{\cK(c)}{n},
$$
which, together with \eqref{FrechetBound2}, concludes the proof of \eqref{linearBothrows}. 

\emph{Step 3 - Expectation bound for $ \int \mathcal{B}(\mP_n-\mP) d\mQ_m $.} For the second claim we note that
$$ \int \mathcal{B}(\mP_n-\mP) d\mQ_m =\frac{1}{n\, m}\sum_{i=1}^n\sum_{j=1}^m \Big(\xi_{\mP,\mQ} (\X_i,\Y_j)- \int \xi_{\mP,\mQ} (\x,\Y_j)d\mP(\x)\Big) $$
and observe,
\begin{multline*}
     \E \left[ \xi_{\mP,\mQ} (\X_{1},\Y_{1})- \int \xi_{\mP,\mQ} (\x,\Y_{1})d\mP(\x)\bigg\vert \{\X_i\}_{i = 2}^{n}, \{\Y_j\}_{j = 1}^m\right]\\
    =\E \left[ \xi_{\mP,\mQ} (\X_{1},\Y_{1})- \int \xi_{\mP,\mQ} (\x,\Y_{1})d\mP(\x)\bigg\vert \Y_1 \right]=0.
\end{multline*}
As a consequence, by using independence, 
\begin{align*}
    &\E \left[\left(\frac{1}{n\, m}\sum_{i=1}^n\sum_{j=1}^m \Big(\xi_{\mP,\mQ} (\X_i,\Y_j)- \int \xi_{\mP,\mQ} (\x,\Y_j)d\mP(\x)\Big)\right)^2\right]\\
    &=\frac{1}{n^2 m^2}\sum_{i=1}^n\sum_{j=1}^m\E \left[\left( \xi_{\mP,\mQ} (\X_i,\Y_j)- \int \xi_{\mP,\mQ} (\x,\Y_j)d\mP(\x)\right)^2\right]\\
    &=\frac{1}{n\,  m}\E \left[\left( \xi_{\mP,\mQ} (\X_i,\Y_j)- \int \xi_{\mP,\mQ} (\x,\Y_j)d\mP(\x)\right)^2\right],
\end{align*}
which asserts the second claim with \ref{lem:regularity}. 
\end{proof}

\begin{proof}[Proof of Lemma \ref{Coro:limitSum}]
Note for $k\in \N$, any $f\in L^2_0(\mPn)$, $g\in L^2_0(\mQm)$, and $\kappa\in \R$ using Lemma \ref{lemma:empiricalInverse}, \begin{align*}
	(\AQm\APn)^k(f + \kappa \one) &= (\AQm\APn)^k(f) + \kappa = (\AQm\bAPn)^k(f) + \kappa,\\
	(\APn\AQm)^k(g + \kappa \one) &= (\APn\AQm)^k(g) + \kappa = (\bAPn\AQm)^k(g)+\kappa. 
\end{align*}
We thus obtain that 
\begin{align*}
            &\sup_{ \substack{f\in L^2_0(\mPn)\\ \|f\|_{n}\leq 1} }\| \sum_{k=0}^{N} (\AQm\APn)^k (f + \kappa) -(I_{L^2(\mP_n)}-\AQm\bAPn)^{-1} (f+\kappa)\|_{n}\\
            \leq \; &\sup_{  \substack{f\in L^2_0(\mPn)\\ \|f\|_{n}\leq 1} }\| \sum_{k=N+1}^{\infty} (\AQm\bAPn)^k (f)\|_{n} + N|\kappa|\\
            \leq \; &\sum_{k=N+1}^{\infty}  \norm{\AQm\bAPn}^k_n + N|\kappa|
            \leq \frac{\delta(c)^{N+1}}{1-\delta(c)} + N|\kappa|.
        \end{align*}
Likewise, it holds that 
 \begin{align*}
            &\sup_{ \substack{g\in L^2_0(\mQm)\\ \|g\|_{n}\leq 1} }\| \sum_{k=0}^{N} (\APn\AQm)^k (g + \kappa) -(I_{L^2(\mQm)}-\bAPn\AQm)^{-1} (g)\|_{n}\\
            \leq \; &\sup_{  \substack{g\in L^2_0(\mQm)\\ \|g\|_{m}\leq 1} }\| \sum_{k=N+1}^{\infty} (\bAPn\AQm)^k (g) \|_{n} + (N+1) |\kappa|\leq \frac{\delta(c)^{N+1}}{1-\delta(c)} + (N+1)|\kappa|.\qedhere
        \end{align*}
\end{proof}

\begin{proof}[Proof of Lemma \ref{LemmaBoundOfEigD}]

For the proof of Assertion 1 let $f\in L^2(\mPn)$ be such that $\norm{f}_n\leq 1$. Then, \begin{align*}
    (\DPn - \APn)(f) = \int \{\xipq(\x,\cdot) - \ximn(\x,\cdot)\}f(\x) d\mPn(\x).
\end{align*}
Hence, by Jensen's inequality combined with the Cauchy-Schwarz inequality we obtain, \begin{align*}
    \norm{(\DPn - \APn)(f)}_m^2 =& \int \left[ \int \{\xipq(\x,\y) - \ximn(\x,\y)\}f(\x) d\mPn(\x)\right]^2 d\mQm(\y)\\
    =& \iint \left[\xipq(\x,\y) - \ximn(\x,\y)\right]^2 f^2(\x) d\mPn(\x) d\mQm(\y)\\
    \leq & \iint \left[\xipq(\x,\y) - \ximn(\x,\y)\right]^2 d\mPn(\x) d\mQm(\y) \norm{f}_n^2\\
    =& \norm{\xipq - \ximn}_{n\times m},
\end{align*}
which asserts that $\norm{\DPn - \APn}_{n,m}\leq \norm{\xipq - \ximn}_{n\times m}$. By symmetry, it also follows that $\norm{\DQm - \AQm}_{m,n}\leq \norm{\xipq - \ximn}_{n\times m}$. The second upper bound in terms of $\cK(c)$ then follows by Lemma \ref{lem:regularity}.

To prove Assertion 2 note using Lemma \ref{lemma:empiricalInverse} and Assertion 1 that $$\max(\norm{\DPn}_{n,m},\norm{\DQm}_{m,n}) \leq 1+\norm{\xipq - \ximn}_{n\times m} \leq 1 + \cK(c) =: \cK'(c).$$
Moreover, since $\DQm\DPn-\AQm\APn= \DQm(\DPn - \APn) + (\DQm-\AQm)\APn$ it follows by triangle inequality that 
\begin{align*}
    \norm{\DQm\DPn-\AQm\APn}_{n} & \leq  \norm{\DQm(\DPn-\APn)}_{n} +  \norm{(\DQm-\AQm)\APn}_{n}\\
    &\leq \norm{\DQm}_{m,n}\norm{\DPn-\APn}_{n,m} + \norm{\DQm-\AQm}_{m,n}\norm{\APn}_{n,m}\\
   & \leq \cK'(c)\norm{\xipq - \ximn}_{n\times m} \leq \cK'(c)\cK(c) =: \cK''(c).
\end{align*}

For Assertion 3 first note for $f\in L^2_0(\mPn)$ with $\norm{f}_n \leq 1$ and $g\in L^2_0(\mQn)$ with $\norm{g}_m\leq 1$ that $\max(\norm{f+\kappa \one}_n,\norm{g + \kappa \one}_m) \leq 1+|\kappa|$. Using Lemma \ref{Coro:limitSum}, the bound for \eqref{eq:UpperBoundDPnDQm1} and \eqref{eq:UpperBoundDPnDQm3} follows by triangle inequality once we show for the operator norms on $L^2(\mPn)$ and $L^2(\mQm)$ for some constant $\tilde \cK'''(c)\geq 1$ that
\begin{align*}
& \sum_{k=0}^{N} \norm{(\DQm\DPn)^k - (\AQm\APn)^k}_{n} + \norm{ (\DPn\DQm)^k - (\APn\AQm)^k}_{m}\\
\leq \;& \tilde \cK'''(c)^N\norm{\xipq - \ximn}_{n\times m}.
\end{align*}
To this end, first note for $k=0$ that $(\DQm\DPn)^0 = (\AQm\APn)^0$ and $(\DPn\DQm)^0 = (\APn\AQm)^0$.
Further, recall for bounded operators $A,B$ on a normed vector space that 
$$\norm{(A+B)^k - A^k}\leq 2^{k} \max(\norm{A}, \norm{B})^{k-1} \norm{B}.$$ 

Since $\norm{\AQm\APn}_n\leq 1$ and $\norm{\APn\AQm}_n\leq 1$ by  Lemma \ref{lemma:empiricalInverse} we conclude via  Assertion~2 and Lemma \ref{lem:regularity} for any $k \in \NN$ that
\begin{align*}
    \norm{(\DQm\DPn)^k -(\AQm\APn)^k}_n  %
    \leq  \;& 2^k(1 +\cK''(c))^{k-1} \norm{(\DQm\DPn - \AQm\APn)}_{n}\\
    \leq  \;& 2(2 + 2\cK''(c))^{k-1} \norm{\xipq - \ximn}_{n\times m}.
\end{align*}
An identical bound also holds for $\norm{(\DPn\DQm)^k -(\APn\AQm)^k}_n$. Hence, summing up over $k=1, \dots, N$ we find 
\begin{align}
\notag& \sum_{k=0}^{N} \norm{(\DQm\DPn)^k - (\AQm\APn)^k}_{n} + \norm{ (\DPn\DQm)^k - (\APn\AQm)^k}_{m} \\
\notag \leq \;& \sum_{k=0}^{N} 2(2 + 2\cK''(c))^{k-1} \norm{\xipq - \ximn}_{n\times m} \\
\notag = \;& 4 \frac{2(2 + 2\cK''(c))^{N}-1}{1 + 2\cK''(c)}  \norm{\xipq - \ximn}_{n\times m} \\
\label{eq:boundDifferencePowersDPnApn} \leq  \;& \tilde \cK'''(c)^{N}\norm{\xipq - \ximn}_{n\times m}.%
\end{align}

To show the assertion for \eqref{eq:UpperBoundDPnDQm2} and \eqref{eq:UpperBoundDPnDQm4} we note, for $f\in L^2_0(\mPn), g\in L^2_0(\mQm)$ and $\kappa \in \RR$, that 
\begin{align*}
  \left| \int \bAPn (f+\kappa \one)d\mQm\right|  = 0 \quad \text{and} \quad  \left| \int \AQm (g+\kappa \one)d\mPn\right| = |\kappa|
\end{align*}
as well as
\begin{align*}
    \left|\int \DPn (f+\kappa \one)d\mQm \right|=& \left|\iint \xipq(\x,\y) (f(\x)+ \kappa) d \mPn(\x) d\mQm(\y)\right|\\
   \notag \leq & \left|\iint \ximn(\x,\y) (f(\x) + \kappa) d \mPn(\x) d\mQm(\y)\right|\\
    \notag& + \left|\iint \left[\ximn - \xipq\right](\x,\y) (f(\x) + \kappa) d \mPn(\x) d\mQm(\y)\right|\\
    \notag \leq & |\kappa| + \norm{\xipq-\ximn}_{n\times m} (\norm{f}_{n}  + |\kappa|)
\end{align*}
and likewise 
$$  \left|\int \DQm (g+\kappa \one)d\mPn \right| \leq  |\kappa| + \norm{\xipq-\ximn}_{n\times m} (\norm{f}_{n}  + |\kappa|).$$
Now for the term \eqref{eq:UpperBoundDPnDQm2} we apply triangle inequality and obtain
\begin{subequations}
\begin{align}
\notag&\sup_{\substack{f\in L^2_0(\mPn)\\ \|f\|_{n}\leq 1} } \left\|\left[\sum_{k=0}^{N} (\DPn\DQm)^k \DPn -(I_{L^2_0(\mQm)}-\bAPn\AQm)^{-1}\bAPn \right] (f + \kappa \one)\right\|_{m}\\
\label{eq:boundDifferenceOp1}\leq &\sup_{\substack{f\in L^2_0(\mPn)\\ \|f\|_{n}\leq 1} } \left\|\sum_{k=0}^{N} (\DPn\DQm)^k (\DPn - \bAPn) (f+\kappa \one)\right\|_m \\
\label{eq:boundDifferenceOp2}&+\sup_{\substack{f\in L^2_0(\mPn)\\ \|f\|_{n}\leq 1} }\left\| \left(\sum_{k=0}^{N} (\DPn\DQm)^k -(I_{L^2_0(\mQm)}-\bAPn\AQm)^{-1}\right)\bAPn  (f)\right\|_{m},
\end{align}
\end{subequations}
where for the last term we used that $\bAPn(f + \kappa \one) = \bAPn(f)$. 

To analyze the term in \eqref{eq:boundDifferenceOp1} we rewrite the term $$(\DPn - \bAPn) (f+\kappa \one) =  (\DPn - \APn) (f+\kappa \one) + \kappa \one$$ as $f_n + \kappa_n \one\in L^2(\mQm)$ for $f_n\in L^2_0(\mQm)$ and $ \kappa_n\in \RR$. Further, by Assertion 1, it follows that $\max(\norm{f_n}_m, |\kappa_n|)\leq \norm{\xipq - \ximn}_{n\times m}(1+|\kappa|) + |\kappa|$. Moreover, upon combining Lemma \ref{lemma:empiricalInverse} with \eqref{eq:boundDifferencePowersDPnApn}, we obtain 
$$  \|\sum_{k=0}^{N} (\DPn\DQm)^k\| \leq N + \tilde \cK'''(c)^N \norm{\xipq - \ximn}_{n\times m},$$
which yields by Lemma \ref{lem:regularity} that
\begin{align*}
&\sup_{\substack{f\in L^2_0(\mPn)\\ \|f\|_{n}\leq 1} } \left\|\sum_{k=0}^{N} (\DPn\DQm)^k (\DPn - \bAPn) (f+\kappa \one)\right\|_m \\
\leq & \left( N + \tilde \cK'''(c)^N \norm{\xipq - \ximn}_{n\times m}\right) 2\left( \norm{\xipq - \ximn}_{n\times m}( 1+ |\kappa|)  + |\kappa|\right)\\
\leq & \left(N + \tilde \cK'''(c)^N  \cK(c)\right)2\left( \norm{\xipq - \ximn}_{n\times m}  +(1+2\cK(c)) |\kappa|\right)\\
\leq &  \;\;\cK'''(c)^N\left( \norm{\xipq - \ximn}_{n\times m}  +|\kappa|\right).
\end{align*}
Further, since $\bAPn  (f)\in L^2_0(\mQm)$ and $\|\bAPn\|_{n,m}\leq 1$, the first part of this assertion that  gives the following upper bound for \eqref{eq:boundDifferenceOp2}:
\begin{align*}
& \sup_{\substack{f\in L^2_0(\mPn)\\ \|f\|_{n}\leq 1} }\left\| \left(\sum_{k=0}^{N} (\DPn\DQm)^k -(I_{L^2_0(\mQm)}-\bAPn\AQm)^{-1}\right)\bAPn  (f)\right\|_{m}\\
\leq \;&\tilde \cK'''(c)^N\norm{\xipq - \ximn}_{n\times m}+\frac{\delta(c)^{N+1}}{1-\delta(c)}.
    \end{align*}

For the upper bound of the term \eqref{eq:UpperBoundDPnDQm4} note that $$(I_{L^2(\mPn)}-\AQm\bAPn)^{-1} \AQm(g) = (I_{L^2(\mPn)}-\AQm\bAPn)^{-1} \AQm(g+\kappa\one) -\kappa \one.$$ Hence, following an analogous strategy as for \eqref{eq:UpperBoundDPnDQm2}, we obtain    
\begin{align*}
   & \sup_{\substack{g\in L^2_0(\mQm)\\ \|g\|_{m}\leq 1}}\left\| \sum_{k=0}^{N} (\DQm\DPn)^k \DQm (g +\kappa \one) - (I_{L^2(\mPn)}-\AQm\bAPn)^{-1} \AQm(g)\right\|_{n}\\
    \leq \;& \sup_{\substack{g\in L^2_0(\mQm)\\ \|g\|_{m}\leq 1}}\left\| \sum_{k=0}^{N} (\DQm\DPn)^k (\DQm - \AQm) (g +\kappa \one)\right\|_{n} +|\kappa|\\
    & + \sup_{\substack{g\in L^2_0(\mQm)\\ \|g\|_{m}\leq 1}}\left\| \left(\sum_{k=0}^{N} (\DQm\DPn)^k  - (I_{L^2(\mPn)}-\AQm\bAPn)^{-1} \right)\AQm(g+\kappa \one)\right\|_{n}\\
    \leq\; & \cK'''(c)^N (\norm{\xipq - \ximn}_{n\times m} + |\kappa|) + \frac{\delta(c)^{N+1}}{1-\delta(c)},
\end{align*}
where we used in the second inequality the bound from the first part of this assertion and that $\AQm(g + \kappa\one) = \AQm(g) +\kappa\one$ for $g\in L^2_0(\mQm)$ and $\AQm(g)\in L^2_0(\mPn)$.
\end{proof}

\begin{proof}[Proof of Proposition \ref{prop:linearPotOpchange}]
   Based on Proposition~\ref{prop:linearPot} we need to check that 
$$\left\|\Theta_{n,m}^N \left(\begin{array}{c}
     \mathcal{B}_\y(\mQ_m-\mQ)  \\
      \mathcal{B}_\x(\mP_n-\mP)
\end{array}\right)  -\Gamma_{n,m}^{-1}\left(\begin{array}{c}
     \mathcal{B}_\y(\mQ_m-\mQ)  \\
      \mathcal{B}_\x(\mP_n-\mP)-\int \mathcal{B}_\x(\mP_n-\mP)d\mQ_m
\end{array}\right)\right\|_{n\times m}= o_p\left(\sqrt{\frac{n+m}{nm}}\right).$$ 

Due to  Proposition~\ref{prop:linearPot} we know that 
$$ \left|\int \mathcal{B}_\x (\mP_n-\mP)d\mQ_m\right| + \left| \int \mathcal{B}_\y (\mQm-\mQ)d\mPn\right| = \Op\left(\frac{1}{\sqrt{nm}}\right)$$
Further, Theorem 5 of \cite{rigollet2022sample} asserts that $$ \norm{\ximn - \xipq}_{n\times m} = \Op\left(\sqrt{\frac{n+m}{nm}}\right).$$
Note that by our choice for $N= N(n,m)$ for any $a>1$ it holds $a^N \lesssim \mathcal{O}\big((\frac{nm}{n+m})^{1/4}\big)$.
Combining these upper bounds with Lemma \ref{LemmaBoundOfEigD} we obtain
\begin{align*}
 & \|\sum_{k=0}^{N} (\DQm\DPn)^k (\mathcal{B}_\y (\mQm-\mQ))-(I_{L^2(\mP_n)}-\AQm\bAPn)^{-1}  (\mathcal{B}_\y (\mQm-\mQ))\|_{n} \\
 \leq \; & \cK(c)^N\norm{\xipq - \ximn}_{n\times m}\\
& \quad \times \left(
	\left\| \mathcal{B}_\y(\mQ_m-\mQ) - \int  \mathcal{B}_\y(\mQ_m-\mQ) d\mPn\right\|_{n}  + \left|\int  \mathcal{B}_\y(\mQ_m-\mQ) d\mPn\right|\right)\\
 &\quad + \frac{\delta(c)^N}{1-\delta(c)} \left\| \mathcal{B}_\y(\mQ_m-\mQ) - \int  \mathcal{B}_\y(\mQ_m-\mQ) d\mPn\right\|_{n}\\
  \lesssim \;& \Op \left( \left(\frac{nm}{n+m}\right)^{\frac{1}{4}-\frac{1}{2}} \right)\left(
	\left\| \mathcal{B}_\y(\mQ_m-\mQ)\right\|_{n} +  \Op \left( \frac{1}{\sqrt{nm}} \right) \right)\\
 &\quad +o(1)\left(
	\left\| \mathcal{B}_\y(\mQ_m-\mQ)\right\|_{n} +  \Op \left( \frac{1}{\sqrt{nm}} \right) \right)\\
\lesssim \; & o_p(1)\left\| \mathcal{B}_\y(\mQ_m-\mQ)\right\|_{n}  + o_p\left( \sqrt{\frac{n+m}{nm}} \right).
 \end{align*}
 Analogous upper bounds also hold for the remaining three terms, where $\mathcal{B}_\y(\mQ_m-\mQ)$ is possibly replaced by $\mathcal{B}_\x (\mPn-\mP)$.
Hence, the assertion follows once we show that 
\begin{align*}
	& \left\|\int \xi_{\mP,\mQ}(\cdot,\y)d(\mQ_m-\mQ)(\y)\right\|_{n} = \mathcal{O}_p\left(\frac{1}{\sqrt{m}}\right),\\
	& \left\| \int \xi_{\mP,\mQ} (\x, \cdot)d(\mP_n-\mP)(\x)\right\|_{m} = \mathcal{O}_p\left(\frac{1}{\sqrt{n}}\right).
\end{align*}
We prove only the first one. The second holds by symmetry.
First we compute  
\begin{multline*}
   \E\left[\left\|\int \xi_{\mP,\mQ}(\cdot,\y)d(\mQ_m-\mQ)(\y)\right\|_{n}^2\right]\\
   =\frac{1}{n\, m^2}\sum_{i=1}^n\sum_{j=1}^m \sum_{k=1}^m \E \left[(\xi_{\mP,\mQ}(\X_i,\Y_j)-1) (\xi_{\mP,\mQ}(\X_i,\Y_k)-1)\right].
\end{multline*}
Then we note that $\E \left((\xi_{\mP,\mQ}(\X_i,\Y_j)-1) (\xi_{\mP,\mQ}(\X_i,\Y_k)-1)\right)=0$, if  $k\neq j$. Therefore, 
\begin{align*}
    \E\left[ \left\|\int \xi_{\mP,\mQ}(\cdot,\y)d(\mQ_m-\mQ)(\y)\right\|_{n}^2\right]
   &=\frac{1}{n\, m^2}\sum_{i=1}^n\sum_{j=1}^m \E \left[\left(\xi_{\mP,\mQ}(\X_i,\Y_j)-1\right)^2\right]\\
      &=\frac{1}{m}\E \left[\left(\xi_{\mP,\mQ}(\X_i,\Y_j)-1\right)^2\right]
\end{align*}
and the claim follows at once. 
\end{proof}

\begin{proof}[Proof of Proposition \ref{prop:Vstatistic}]
We show the assertion by proving that the terms in the definition of $\tilde A_{n,m}^{(l)}$ for $l\in \{1, 2\}$ can be rewritten as in \eqref{eq:AnmPrime_representation}. 
To avoid repetition, we will only state the proof for the first of the terms, the remaining terms can be treated analogously. That is, we show 
    \begin{align*}
   \sum_{k = 0}^N U_{n,m,k}^{(1)}
   &= \frac{1}{n\, m^2}
   \sum_{k=0}^N\sum_{i=1}^n\sum_{s,j=1}^m\xi(\X_i,\Y_j)\eta(\X_i,\Y_s) 
   (\mathcal{D}_{\mQ_m}\mathcal{D}_{\mP_n})^k \left( \xi (\cdot,\Y_k)-1\right)(\X_j)\\
    &=\int  \sum_{k=0}^N(\mathcal{D}_{\mQ_m}\mathcal{D}_{\mP_n})^k(\mathbb{G}_{\mQ,s}^m)  
 \eta \,  \xi d \mP_n d\mQ_m.
\end{align*}
Set  $f\in L^2(\mP_n)$. A routine calculation gives 
$$\mathcal{D}_{\mQ_m}\mathcal{D}_{\mP_n} f
    =\frac{1}{n\, m}\sum_{i_1=1}^n\sum_{j_1=1}^m\xi(\cdot ,\Y_{j_1}) \xi(\X_{i_1},\Y_{j_1}) f(\X_{i}).$$
By repeating this formula again, we obtain 
$$
    (\mathcal{D}_{\mQ_m}\mathcal{D}_{\mP_n})^2 f=
  {\tiny  \frac{1}{n^2 m^2} \sum_{i_1,i_2=1}^n\sum_{j_2,j_2=1}^m }
    \xi(\cdot ,\Y_{j_2})\xi(\X_{i_2},\Y_{j_2}) 
    \xi(\X_{j_2} ,\Y_{j_1})
    \xi(\X_{i_1},\Y_{j_1})  f(\X_{i_1}).
$$
An inductive argument gives for $k\in\NN$ that 
\begin{multline*} 
    (\mathcal{D}_{\mQ_m}\mathcal{D}_{\mP_n})^k f = 
   \tiny  \frac{1}{n^k\, m^k}\sum_{\bi \in \nnk}\sum_{\bj \in \mmk }\xi(\cdot,\Y_{j_k}) \xi(\X_{i_{k}},\Y_{j_k})  f(\X_{i_1})\\
  \times  \prod_{l=1}^{k-1}\xi(\X_{i_l},\Y_{j_l})\xi(\X_{i_{l+1}},\Y_{j_l}),
\end{multline*}
which applied to $f=\mathbb{G}_{\mQ,s}^m=\frac{1}{m}\sum_{j=1}^m\xi(\cdot,\Y_j)-1$ 
gives
\begin{multline*} 
(\mathcal{D}_{\mQ_m}\mathcal{D}_{\mP_n})^k \mathbb{G}_{\mQ,s}^m
   =\tiny  \frac{1}{n^{k}\, m^{k+1}} \sum_{\bi \in \nnk}\sum_{\bj \in \mm{k+1}}\xi(\cdot,\Y_{j_{k}}) \xi(\X_{i_{k}},\Y_{j_{k}}) (\xi(\X_{i_1},\Y_{j_{k+1}})-1) \\
  \times \prod_{l=1}^{k-1}\xi(\X_{i_l},\Y_{j_{l}})\xi(\X_{i_{l+1}},\Y_{j_{l}}).
\end{multline*}
By integrating this quantity  w.r.t. $ \eta \xi d\mP_n d\mQ_m$ we obtain 
\begin{align*}
     & \int \eta \xi(\mathcal{D}_{\mQ_m}\mathcal{D}_{\mP_n})^k \mathbb{G}_{\mQ,s}^m d\mP_n d\mQ_m\\
     =&  \frac{1}{n^{k+1}\, m^{k+2}} \sum_{\bi \in \nn{k+1}}\sum_{\bj \in \mm{k+2}} \eta(\X_{i_{k+1}},\Y_{j_{k+1}})  \xi(\X_{i_{k+1}},\Y_{j_{k+1}}) (\xi(\X_{i_1},\Y_{j_{k+1}})-1)  \\
     & \quad \quad \quad \quad \quad \quad \quad \quad \quad \quad \quad \quad \quad \times \prod_{l=1}^{k}\xi(\X_{i_l},\Y_{j_{l}})\xi(\X_{i_{l+1}},\Y_{j_{l}})\\
     =&\frac{1}{n^{k+1}\, m^{k+2}} \sum_{\substack{\mathbf{i}\in [[n]]^{k+1}}} \sum_{\bj \in \mm{k+2}}f_k^{(1)}( \X_{\mathbf{i}},\Y_{\mathbf{j}}),
\end{align*}
which concludes the proof. 
\end{proof}

\begin{proof}[Proof of Proposition \ref{prop:Ustatistic}]
We only state the proof for the first sum $\sum_{k = 0}^{N} U_{n,m,k}^{(1)}$, the remaining terms can be treated analogously. 

We note that if $k+2\leq m$ and  $k+1\leq n$, the set $\mathcal{S}_{k,n,m} \coloneqq \mathcal{S}_{k+1,n}\times \mathcal{S}_{k+2,m}$ of $(k+1)(k+2)$-tuples $({\bf i}, {\bf j})=(i_1, \dots, i_{k+1}, j_{1}, \dots, j_{k+2})$ such that $i_{l}\neq i_{q}$ and $j_{l}\neq j_{q}$, for all $q\neq l $, has cardinality $ s(n,m,k) = n (n-1) \cdots (n-k)m(m-1) \cdots (m-k-1)$. 
In particular, the cardinality of the complementary set of $\mathcal{S}_{k,n,m}$, denoted by $\mathcal{S}_{k,n,m}^c$, is bounded~by
\begin{align*}
    & n^{k+1}m^{k+2} -  n(n-1)\cdots (n-k)  m(m-1)\cdots (m-k-1) \\
\leq \; & n^{k+1}m^{k+2} - n (n-k)^k m (m-k-1)^{k+1}\\ = \;& nm\{n^{k}(m^{k+1} - (m-k-1)^{k+1}) + (n^{k} - (n-k)^k)(m-k-1)^{k+1}\}\\
\leq \; & n^{k+1}m^{k+1} k^2 + n^{k}m^{k+2} k^2,
\end{align*}
where we used in the last inequality that $a^k - b^k \leq k a^{k-1} (a-b)$ for $a\geq b \geq 0$ and $k \in \N$. Hence, if $k^2/n\to 0$, then $s(n,m,k) = (1-o(1)) n^{k+1} m^{k+2}$.
For $ (\bi,\bj)\in \mathcal{S}_k^c$, we bound 
\begin{equation}
   \label{boundonf} 
    \E  \left[\vert   f_k^{(1)}(\X_{\bi},\Y_{\bj})\vert \right]
    \leq (\cK(c))^{2k+1}\E[\vert \eta(\X,\Y)\vert], 
\end{equation}
with $\cK(c)>1$ from Lemma \ref{lem:regularity}.
As a consequence, 
\begin{multline*}
    {\tiny  \frac{1}{n^{k+1}\, m^{k+2}} \sum_{\substack{\bi\in \nn{k+1}}}\sum_{\substack{\bj\in \mm{k+2}}} }
  f_k^{(1)}(\X_{\bi},\Y_{\bj})
  \\={\tiny  \frac{1}{n^{k+1}\, m^{k+2}} \left[\sum_{(\bi, \bj)\in \mathcal{S}_{k,n,m}}
  f_k^{(1)}(\X_{\bi},\Y_{\bj}) +  \sum_{(\bi, \bj)\in \mathcal{S}_{k,n,m}^c}
  f_k^{(1)}(\X_{\bi},\Y_{\bj})\right] } .
\end{multline*}
We note that the second term, which we denote by $B_k^{(n)}$, can be upper bounded through \eqref{boundonf} as follows 
\begin{multline*}
    \E[\vert B_k^{(n)}\vert]\le \frac{n^{k+1}m^{k+2}-s(n,m,k))}{n^{k+1}m^{k+2}} (\cK(c))^{2k+1}\E[\vert \eta(\X,\Y)\vert]\\
\leq
\frac{n+m}{nm} k^2 (\cK(c))^{2k+1}\E[\vert \eta(\X,\Y)\vert].
\end{multline*}
Then 
$$\E \bigg[\bigg\vert  \int \eta \xi\sum_{k=0}^N (\mathcal{D}_{\mQ_m}\mathcal{D}_{\mP_n})^k \mathbb{G}_{\mQ,s}^m d\mP_n d\mQ_m
     - {\tiny  \sum_{k=0}^N\frac{1}{n^{k+1}\, m^{k+2}} \sum_{(\bi, \bj)\in \mathcal{S}_{k,n,m}}
  f_k^{(1)}(\X_{\bi},\Y_{\bj})}\bigg\vert\bigg] $$
can be  upper bounded by 
\begin{multline*}
     B_N\coloneqq  \sum_{k=0}^N \E[\vert B_k^{(n)}\vert] \leq \sum_{k=0}^N\frac{n+m}{nm}N^2 (\cK(c))^{2k+1}\E[\vert \eta(\X,\Y)\vert ]\\
=\frac{n+m}{nm}\frac{\cK(c)^{2N}-1}{\cK(c)^2-1} N^2\cK(c) \E[\vert \eta(\X,\Y)\vert].
\end{multline*}
Upon choosing $ N = N(n,m) = \Nnm$ it follows  that  
$$ B_N\lesssim \frac{n+m}{nm}\sqrt[4]{\frac{nm}{n+m}}\log^2\left(\frac{nm}{n+m}\right) = o\left(\sqrt{\frac{n+m}{nm}}\right),$$
where the hidden constant only depends on $\cK(c)$ and $\E(\vert \eta(\X,\Y)\vert)$.
As a consequence, 
\begin{multline}
  \sqrt{\frac{n\, m}{n+m}}  \E \bigg\vert  \int \eta \xi_{\mP,\mQ}\sum_{k=0}^N (\mathcal{D}_{\mQ_m}\mathcal{D}_{\mP_n})^k \mathbb{G}_{\mQ,s}^m d\mP_n d\mQ_m
     \\
     - {\tiny  \sum_{k=0}^N\frac{1}{n^{k+1}\, m^{k+2}} \sum_{(\bi, \bj)\in \mathcal{S}_{k,n,m}}
  f_k^{(1)}(\X_{\bi},\Y_{\bj})}\bigg\vert \to 0, 
\end{multline}
as $n,m\to \infty$, with $\frac{m}{n+m}\to \lambda\in (0,1)$, and $ N= N(n,m) = \Nnm$.
\end{proof}

\begin{proof}[Proof of Proposition \ref{prop:writeFproperly}]
We only prove the assertion for the first term, the remaining projections can be treated in an analogous manner. Moreover, to simplify notation we write $d \mP(\x)$ as $d \x$,  $d \mQ(\y)$ as $d \y$ and $ \xi_{\mP,\mQ}$ as  $\xi$. 

First notice that 
  $ \E[{\U}^{(1)}_{n,m,N}\vert \X_i]=0,$ for all $i\in \{ 1, \dots, n\}$. This is a consequence of 
  \begin{multline}\label{eq:condtionalexp}
      \E\bigg[f_k^{(1)}(\X_{\bi},\Y_{\bj})\bigg\vert \X_{\bi}, \{\Y_{j_{l}}\}_{l=1}^{k+1}\bigg]\\ =\xi(\X_{i_{k+1}}\Y_{i_{k+1}})\eta(\X_{i_{k+1}}\Y_{i_{k+1}})
       \prod_{l=1}^{k}\xi(\X_{i_l},\Y_{j_l})\xi(\X_{i_{l+1}},\Y_{j_l}) \\
      \times \E\left[(\xi(\X_{i_1},\Y_{j_{k+2}})-1)\big\vert \X_{i_1} \right]=0,
  \end{multline}
  where the conditional expectation in the last line coincides with zero since the optimality condition \eqref{eq:optimallityCodt0} asserts for any $\x \in \Omega$ that 
  $ \int \xi_{\mP,\mQ}(\x,\y) d{\rm Q}(\y)=1$. Equation \eqref{eq:condtionalexp} yields that the H\'ajek projection of $\U^{(1)}_{n,m,N}$ onto the set of  random variables $\sum_{i=1}^n \mathbf{g}_{i,x}(\X_i)$ is $\mathbf{g}_{i,x}={\bf 0}$, for all $i=1,\dots, n$. Again via \eqref{eq:condtionalexp}, the  $L^2(\P)$ projection of $\U^{(1)}_{n,m,N}$ onto the set of random variables $\sum_{j=1}^m \mathbf{g}_{i,y}(\Y_j)$ is 
  $$ \frac{1}{m} \sum_{j=1}^m \sum_{k=0}^N \tilde F_{k}(\Y_j),$$
  where, for $\y\in \Omega$, the function $ H_{k}(\y)$ is defined as
  \begin{align*}%
    \E\bigg[\xi_{\mP,\mQ}(\X_{k+1}\Y_{k+1})\eta(\X_{k+1}\Y_{k+1})  
     (\xi_{\mP,\mQ}(\X_{1},\y)-1) \prod_{l=1}^{k}\xi_{\mP,\mQ}(\X_{l},\Y_{l})\xi_{\mP,\mQ}(\X_{l+1},\Y_{l})\bigg].
  \end{align*}

    We show, using independence of $\X_1, \dots, \X_n, \Y_1, \dots, \Y_m$ for any $k\in\NN\cup\{0\}$, that
\begin{multline*}
\tilde{F}_k(\y)=
    \int\cdots \int  \bigg\lbrace \xi(\x_{k+1},\y_{k+1})\eta(\x_{k+1},\y_{k+1})  
     (\xi(\x_{1},\y)-1) \\
     \prod_{l=1}^{k}\xi(\x_{l},\y_{l})\xi(\x_{l+1},\y_{l}) \bigg\rbrace d\x_1 d\y_1 \dots d\x_{k+1} d\y_{k+1}.
  \end{multline*}
To prove the assertion for $k = 0$ note that
\begin{align*}
    \tilde{F}_0(\y)&=
    \int \int  \xi(\x_{1},\y_{1})\eta(\x_{1},\y_{1})  
     (\xi(\x_{1},\y)-1) 
      d\x_1 d\y_1\\
      &=\int \left( \int  \xi(\x_{1},\y_{1})\eta(\x_{1},\y_{1}) d\y_1 (\xi(\x_{1},\y)-1)\right)  d\x_1 \\
      &=\int \left( \int  \eta_{\x}(\x_1) (\xi(\x_{1},\y)-1)\right)  d\x_1  
      \\
      &=(\AP  \eta_{\x})(\y)-\int \eta_{\x}(\x_1)  d\x_1.
\end{align*}
Using the entropic optimal transport plan $\pi$ and its representation \eqref{eq:optimallityCodPlan} we obtain by Fubini's theorem that
$$ \int \eta_{\x}(\x_1)  d\x_1 = \int \eta_{\x}(\x_1) d \pi(\x_1,\y_1) = \iint \xi(\x_1, \y) \eta_{\x}(\x_1)  d\x_1 d\y_1  =\int \AP \eta_{\x}(\y_1)  d\y_1.$$

  Next, we consider the case $k\geq 1$. For the sake of readability we define the function  \begin{multline*}
      G_k(\y,\x_1,\y_1, \dots ,\x_{k} ,\y_{k})
    \\= \xi(\x_{k},\y_{k}) \xi(\x_{k},\y_{k-1}) \xi(\x_{k-1},\y_{k-1})\cdots \xi(\x_{2},\y_{2}) \xi(\x_{2},\y_{1})  \xi(\x_{1},\y_{1}) (\xi(\x_{1},\y)-1).
  \end{multline*}
  Then, by applying Fubini's theorem it follows that 
  \begin{align*}
       \tilde{F}_k(\y)&= \int\cdots \int  \bigg( \int \bigg( \int \bigg(\xi(\x_{k+1},\y_{k+1})\eta(\x_{k+1},\y_{k+1})  \bigg) d\y_{k+1} \xi(\x_{k+1},\y_{k}) \bigg) d \x_{k+1}\\
    & \qquad \qquad \qquad \qquad \qquad \qquad G_k(\x_1,\y_1, \dots ,\x_{k} ,\y_{k})
    \bigg) d\x_k d\y_k \dots d\x_{1} d\y_{1}\\
    &= \int\cdots \int \bigg( \int \bigg( \eta_{\x}(\x_{k+1})\xi(\x_{k+1},\y_{k}) \bigg) d \x_{k+1}\\
    & \qquad \qquad \qquad \qquad \qquad \qquad  G_k(\x_1,\y_1, \dots ,\x_{k} ,\y_{k})
    \bigg) d\x_k d\y_k \dots d\x_{1} d\y_{1}\\
    &= \int\cdots \int \bigg( (\AP \eta_{\x})(\y_{k})
    G_k(\x_1,\y_1, \dots ,\x_{k} ,\y_{k})
    \bigg) d\x_k d\y_k \dots d\x_{1} d\y_{1}.
  \end{align*}
 
 We prove by induction for $k\geq 1$ that
 \begin{multline}
       \label{inductionOnG}
     \int\cdots \int G_k(\y, \x_1,\y_1, \dots ,\x_{k} ,\y_{k}) f(\x_k)d\x_k d\y_k \dots d\x_{1} d\y_{1}\\
     = ((\AP\AQ)^{k} f)(\y)- \int ((\AQ\AP)^{k-1} f)  (\y)d\y,
 \end{multline}
for all $f\in L^2({\rm P})$. Note that for $k=1$, we have
\begin{align*}
    \iint G_1(\y, \x_1,\y_1) f(\y_1)d\x_1 d\y_1 &= \iint  \xi(\x_{1},\y_1) (\xi(\x_{1},\y)-1) f(\y_1) d\x_1 d\y_1 \\
    &=\iint  \xi(\x_{1},\y_1) f(\y_1) d\y_1 (\xi(\x_{1},\y)-1)  d\x_1     \\
    &=\int    (\AQ f)(\x_1) (\xi(\x_{1},\y)-1) d\x_1\\
    &=(\AP \AQ f)(\y)-\int f(\y) d\y. 
\end{align*}
We assume that \eqref{inductionOnG} holds for $k-1$, and we prove it for $k$; first we note that
$$
    G_k(\y, \x_1,\y_1, \dots ,\x_{k} ,\y_{k}) 
    = \xi(\x_{k},\y_{k}) \xi(\x_{k-1},\y_{k}) G_{k-1}(\y, \x_1,\y_1, \dots ,\x_{k-1} ,\y_{k-1}).
$$
As a consequence, by definitions of $\AP$ and $\AQ$ it holds that
\begin{align*}
     &\int G_k(\y, \x_1,\y_1, \dots ,\x_{k} ,\y_{k}) f(\y_k) d\y_k \dots d\x_{1} d\y_{1}\\
    &= \int \bigg( \cdots\int \bigg(\int \bigg(f(\y_k) \xi(\x_{k},\y_{k}) \bigg)  d\y_k \xi(\x_{k},\y_{k-1}) \bigg) d\y_k \\
   & \qquad \qquad \qquad \qquad G_{k-1}(\y, \x_1,\y_1, \dots ,\x_{k-1} ,\y_{k-1}) \bigg) d\x_{k-1}  d\y_{k-1} \dots d\x_{1} d\y_{1}\\
    &= \int \bigg( \cdots\int \bigg(\int (\AQ f)(\x_k)  \xi(\x_{k},\y_{k-1}) \bigg) d\x_k \\
   & \qquad \qquad \qquad \qquad G_{k-1}(\y, \x_1,\y_1, \dots ,\x_{k-1} ,\y_{k-1}) \bigg) d\x_{k-1}  d\y_{k-1} \dots d\x_{1} d\y_{1}\\
   &= \int (\AP\AQ f)(\y_{k-1})  G_{k-1}(\y, \x_1,\y_1, \dots ,\x_{k-1} ,\y_{k-1})  d\x_{k-1}  d\y_{k-1} \dots d\x_{1} d\y_{1}.
\end{align*}
By the induction hypothesis (applied to $(\AP\AQ f)$ instead of $f$), we obtain 
$$  \tilde{F}_{k}(\y) =  ( (\AP\AQ)^{k} \AP \eta_{\x})(\y)-\E[(\AP\AQ)^{k-1}\AP \eta_{\x})(\Y)].$$
Finally, to conclude the proof we infer \eqref{inductionOnG} for $k\geq  1$ from the observation 
$$ \E[(\AP\AQ)^{k-1}\AP \eta_{\x})(\Y)]=\E[(\AP\AQ)^{k}\AP \eta_{\x})(\Y)],$$
which is consequence of the optimality conditions \eqref{eq:optimallityCodPlan}; set  $g\in L^2(\mQ)$ and compute 
\begin{align*}
    \int (\AP\AQ g) (\y)  d\y&= \int \xi(\x,\y) \xi(\x,\y') g (\y')d\y'd\x   d\y\\
   &=\int \left( \int \xi(\x,\y) d\y \right)\xi(\x,\y') g (\y')d\y'd\x  \\ 
     &=\int \xi(\x,\y') g (\y')d\y'd\x \\
     &=\int g (\y')d\y'. \qedhere
\end{align*}
\end{proof}

\begin{proof}[Proof of Proposition \ref{prop:FandFinf}]

We only show the assertion for $l=1$, the remaining cases $l \in \{2,3,4\}$ can be treated analogously. 

     To do so, we define the difference  
\begin{equation*}
 \Delta_{m,N}^{(1)} :=\frac{1}{m} \sum_{j=1}^m F_{\infty,y}(\Y_j) - F_{N,y}(\Y_j)= \frac{1}{m} \sum_{j=1}^m\sum_{k=N+1}^\infty  \bigg(((\AP\AQ)^{k} \AP \bar \eta_{\x})(\Y_j)
    \bigg)
\end{equation*}
and introduce the function $$ H_N^{(1)}(\y)=\sum_{k=N+1}^\infty  \left(((\AP\AQ)^{k} \AP \bar\eta_{\x})(\y)
    \right).$$ Based on \eqref{eq:ExpectationExchange} we observe for any $j \in \{1, \dots, m\}$ and $N\in \NN$ that $\E[H_N^{(1)}(\Y_j)]=0$.
Further, since $\Y_1, \dots, \Y_m$ are independent and identically distributed it holds that 
\begin{equation}
    \label{DeltaLeq}
    \E((\Delta_{m,N}^{(1)})^2)= \frac{1}{m} \E[H_N^{(1)}(\Y_1)^2],
\end{equation}
where by triangle inequality we obtain
\begin{align*}
    \E(H_N^{(1)}(\Y_1)^2)^{1/2} & = \norm{\sum_{k=N+1}^\infty  ((\AP\AQ)^{k} \AP \bar \eta_{\x})}_{L^2(\mQ)} \\
    &\leq  \sum_{k=N+1}^\infty \norm{  (\AP\AQ)^{k} \AP \bar\eta_{\x}}_{L^2(\mQ)} \leq \sum_{k=N+1}^\infty \norm{\AP\AQ}^{k} \norm{ \AP \bar\eta_{\x}}_{L^2(\mQ)}.
\end{align*}
A straight-forward computation yields \begin{align*}
     \norm{ (\AP \bar\eta_{\x})}_{L^2(\mQ)}^2  %
 &= \int \left(\int \xipq(\x,\y) \bar\eta_\x(\x) d\mP(\x) \right)^2d\mQ(\y) \\
 &= \int \left(\int \bar\eta_\x(\x) d\mP(\x) \right)^2d\mQ(\y) = \left(\int \bar\eta_\x(\x) d\mP(\x) \right)^2 \\
 &= \left(\iint \bar\eta(\x,\y)\xipq(\x,\y) d \mQ(\y) d\mP(\x) \right)^2 \\
 &\leq \left(\iint \bar\eta^2(\x,\y)d\pi_{\mP,\mQ}(\x,\y) \right) = \norm{\bar\eta}_{L^2(\pi_{\mP,\mQ})}^2,
\end{align*}
where we use Jensen's inequality for the last line.
Combined with Lemma \ref{lemma:empiricalInverse} we have 
\begin{align*}
\E(H_N^{(1)}(\Y_j)^2) &\leq \left(\|\bar\eta \|_{L^2(\pi_{\mP,\mQ})}\sum_{k=N+1}^\infty \delta(c)^k\right)^2= \|\bar\eta \|_{L^2(\pi_{\mP,\mQ})}^2 \left(\frac{\delta(c)^{N+1}}{1-\delta(c)}\right)^2.
\end{align*}
Therefore, $\E(H_N^{(1)}(\Y_j)^2)\to 0$ as $N\to \infty$. Via \eqref{DeltaLeq}, we obtain 
$$  \frac{n\, m}{n+m}\E((\Delta_{m,N}^{(1)})^2)\leq \frac{n}{n+m} \E(H_N^{(1)}(\Y_j)^2) \to 0,  $$
as $\frac{n}{n+m} \to 1-\lambda\in (0,1) $ and $N\to \infty$. 
\end{proof}

\begin{proof}[Proof of Proposition \ref{prop:HajekProjection}]
    Again we focus on the first term $\U_{n,m,N}^{(1)}$.
We need to prove that
\begin{align}\label{eq:ApproxUnm}
    \sqrt{\frac{nm}{n+m}} \left(\U_{n,m,N}^{(1)} - \E[\U_{n,m,N}^{(1)}] - \frac{1}{m}  \sum_{j=1}^m  F^{(1)}_{N,y}(\Y_j)   \right)  = o_p(1)
\end{align}
for $n,m=m(n) \rightarrow \infty$ with $\frac{m}{n+m}\rightarrow \lambda \in (0,1)$ and $N=N(n,m) = \Nnm$. To this end, we first assume that $\Var[F^{(1)}_{\infty,y}(\Y)]>0$, then by Proposition \ref{prop:FandFinf} for $N$ sufficiently large it also follows that $$\Var\left[F^{(1)}_{N,y}(\Y)\right] = m \Var\left[\frac{1}{ m} \sum_{i=1}^m  F_{N}(\Y_i)\right] >0.$$
For this setting, it suffices by \citet[Theorem 11.2]{vaart_1998} and Slutsky's Lemma to verify that  
\begin{equation}
    \label{variancesame}
    \frac{\Var[\U_{n,m,N}]}{  \Var  \left[  \frac{1}{ m} \sum_{i=1}^m  F_{N}(\Y_i)\right]}\to  1.
\end{equation}
Notably, by  Proposition \ref{prop:writeFproperly} it holds $\E\left[   F^{(1)}_{N,y}(\Y)\right] = 0$, hence we note that 
\begin{align}
    \begin{split}
         \label{varianceofProjection}
     \Var  \left[  \frac{1}{ m} \sum_{j=1}^m F^{(1)}_{N,y}(\Y_j)\right] &= \E \left[\left(  \frac{1}{ m} \sum_{j=1}^m F^{(1)}_{N,y}(\Y_j)\right)^2\right]\\
     &=\frac{1}{m} \sum_{\substack{k = 0\\ r = 0}}^N\E \left[(\AP\AQ)^{k} \AP \bar\eta_{\x})(\Y_1)(\AP\AQ)^{r} \AP \bar\eta_{\x})(\Y_1)\right].
    \end{split}
\end{align}
Moreover, we recall that  
\begin{align*}
 \U_{n,m,N}^{(1)}=  \sum_{k = 0}^{N} \frac{1}{n^{k+1}\, m^{k+2}} \sum_{\substack{\mathbf{i}\in \mathcal{S}_{k+1,n}\\ {\mathbf{j}\in \mathcal{S}_{k+2,m}}}}
  f_k^{(1)}( \X_{\mathbf{i}},\Y_{\mathbf{j}})
\end{align*}
and note by \eqref{eq:condtionalexp} that
$$ \E [ \U_{n,m,N}^{(1)}]=\sum_{k = 0}^{N} \frac{1}{n^{k+1}\, m^{k+2}} \sum_{\substack{\mathbf{i}\in \mathcal{S}_{k+1,n}\\ {\mathbf{j}\in \mathcal{S}_{k+2,m}}}}
 \E[ f_k^{(1)}( \X_{\mathbf{i}},\Y_{\mathbf{j}})]=0.$$
 In consequence, we have  
 \begin{multline*}
    \Var[\U_{n,m,N}^{(1)}] =  \E [(\U_{n,m,N}^{(1)})^2]= \\\sum_{\substack{k=0\\ r=0}}^N \frac{1}{(  n^{k+1}\, m^{k+2}) (  n^{r+1}\, m^{r+2})}\sum_{\substack{\mathbf{i}, \mathbf{i}'\in \mathcal{S}_{k+1,n}\\ {\mathbf{j}, \mathbf{j}'\in \mathcal{S}_{k+2,m}}}}
 \E[ f_k^{(1)}( \X_{\mathbf{i}},\Y_{\mathbf{j}})f_r^{(1)}( \X_{\mathbf{i}'},\Y_{\mathbf{j}'})].
 \end{multline*}
 We observe that 
 $$ \E\left[ f_k^{(1)}( \X_{\mathbf{i}},\Y_{\mathbf{j}})f_r^{(1)}( \X_{\mathbf{i}'},\Y_{\mathbf{j}'})\right]=0, $$
 if ${j_{k+2}'}\neq {j_{k+2}} $. We denote the set of $(k+1)(k+2)(r+1)(r+2)$-tuples $\left(\begin{array}{c}
      (\mathbf{i}, \mathbf{i}')  \\
     (\mathbf{j}, \mathbf{j}')
 \end{array}\right)$, with $j_{k+2}'=j_{k+2}$ as $\mathcal{R}$.  Then we can restrict the precedent sum to such a set, i.e., 
 \begin{multline*}
    \Var[\U_{n,m,N}^{(1)}] =  \E [(\U_{n,m,N}^{(1)})^2]= \\\sum_{\substack{k=0\\ r=0}}^N \frac{1}{(  n^{k+1}\, m^{k+2}) (  n^{r+1}\, m^{r+2})}\sum_{\left(\begin{array}{c}
      (\mathbf{i}, \mathbf{i}')  \\
     (\mathbf{j}, \mathbf{j}')
 \end{array}\right)\in \mathcal{R}}
 \E[f_k^{(1)}( \X_{\mathbf{i}},\Y_{\mathbf{j}})f_r^{(1)}( \X_{\mathbf{i}'},\Y_{\mathbf{j}'})].
 \end{multline*}
The set $\mathcal{T}$ of $(k+1)(k+2)(r+1)(r+2)$-tuples $\left(\begin{array}{c}
      (\mathbf{i}, \mathbf{i}')  \\
     (\mathbf{j}, \mathbf{j}')
 \end{array}\right)$ such that $\mathbf{i},\mathbf{i}'\in \mathcal{S}_{k+1,n}$, $\mathbf{j},\mathbf{j}'\in \mathcal{S}_{k+2,m}$ and 
 $$ \{i_1, \dots, i_{k+1}, j_{1}', \dots, j_{k+1}\}\cap \{i_1', \dots, i_{r+1}', j_{1}', \dots, j_{r+1}' \}=\emptyset$$
  with $j_{k+2}=j'_{k+2}$ admits cardinality (note that in both tuples the term $j_{k+2}$ agrees) 
 \begin{align*}t(n,m,k,r) &= n (n-1) \cdots (n-r-k-1)m(m-1) \cdots (m-r-k-2).
 \end{align*}
 Its complement intersected with $\mathcal{R}$, call it $ \mathcal{T}^c\cap \mathcal{R}$, admits cardinality upper bounded by %
 \begin{align*}
     & n^{k+r+2}m^{k+r+3} - t(n,m,k,r) \\
    \leq  \; & n^{k+r+2}m^{k+r+2} (k+r+2)^2 + n^{k+r+1} m^{k+r+3} (k+r+1)^2 ,
 \end{align*}
 which yields that 
 \begin{equation}
     \label{cardinalofT}
     \text{ $t(n,m,k,r) = (1-o(1)) n^{k+r+2}m^{k+r+3}$ if $k,r \leq  \lceil\log\{(n\,m)/(n+m)\}\rceil$.}
 \end{equation}
 Now, if $\left(\begin{array}{c}
      (\mathbf{i}, \mathbf{i}')  \\
     (\mathbf{j}, \mathbf{j}')
 \end{array}\right)\in \mathcal{T}$, since $\X_1, \dots, \X_n\sim \mP$ and $\Y_1,\dots, \Y_m\sim \mQ$ are independent, the value of the expectation
 $$ \E\left[ f_k^{(1)}( \X_{\mathbf{i}},\Y_{\mathbf{j}})f_r^{(1)}( \X_{\mathbf{i}'},\Y_{\mathbf{j}'})\right]$$
 does not depend on $\X_{\mathbf{i}},\Y_{\mathbf{j}},  \X_{\mathbf{i}'},\Y_{\mathbf{j}'}$, but only on $ j_{k+2}=j'_{k+2}$. That is, the value of the precedent display is 
 \begin{align*}
\mathbb{M}_{k,r}&=\E\left[f_k(\X_{\mathbf{i}},\Y_{\mathbf{j}}) f_r(\{ \X_{i}\}_{i=k+2}^{k+r+3},\{ \Y_{k+3}, \dots \Y_{k+r+4} , \Y_{k+2}\})\right]\\
     &=\E \left[(\AP\AQ)^{k} \AP \bar\eta_{\x})(\Y_1)(\AP\AQ)^{r} \AP \bar\eta_{\x})(\Y_1)\right].   
 \end{align*}
Hence, 
\begin{equation}
    \label{eq:varianceIndep}
    \sum_{\substack{k=0\\ r=0}}^N \frac{1}{m^{k+r+4} n^{k+r+2}} \sum_{\left(\begin{array}{c}
      (\mathbf{i}, \mathbf{i}')  \\
     (\mathbf{j}, \mathbf{j}')
 \end{array}\right)\in \mathcal{T}}
 \E[ f_k^{(1)}( \X_{\mathbf{i}},\Y_{\mathbf{j}})f_r^{(1)}( \X_{\mathbf{i}'},\Y_{\mathbf{j}'})] 
 = \sum_{\substack{k=0\\ r=0}}^N \frac{t(m,n,r,k)}{m^{k+r+4} n^{k+r+2}}\mathbb{M}_{k,r}.
\end{equation}
 Recall by Lemma \ref{lem:regularity} that $\vert f_k(\X_{\mathbf{i}},\Y_{\mathbf{j}}) \vert \leq \cK(c)^{2k}$ for $\cK(c)>1$  and thus
\begin{align*}
 \sum_{\substack{k=0\\ r=0}}^N \frac{1}{m^{k+r+4} n^{k+r+2}} &\sum_{\left(\begin{array}{c}
      (\mathbf{i}, \mathbf{i}')  \\
     (\mathbf{j}, \mathbf{j}')
 \end{array}\right)\in \mathcal{T}^c\cap \mathcal{R}}
 \E[ f_k^{(1)}( \X_{\mathbf{i}},\Y_{\mathbf{j}})f_r^{(1)}( \X_{\mathbf{i}'},\Y_{\mathbf{j}'})]\\
 & \leq \left| \sum_{\substack{k=0\\ r=0}}^N \frac{m^{k+r+3} n^{k+r+2} - t(m,n,r,k)}{m^{k+r+3} n^{k+r+2}} \cK(c)^{2(k+r)}\right|\\
 & \leq   \left(\frac{1}{n\, m}  +\frac{1}{m^2}\right)N^2 \sum_{\substack{k=0\\ r=0}}^N \cK(c)^{2(k+r)} =  \left(\frac{1}{n\, m}  +\frac{1}{m^2}\right)N^2 \left(\frac{\cK(c)^{2N} - 1}{\cK(c)^2-1}\right)^2 \\ &\leq  \left(\frac{1}{n\, m}  +\frac{1}{m^2}\right)N^2 2\left(\frac{\cK(c)^{4N} + 1}{\cK(c)^4- 2\cK(c)+1}\right).
\end{align*}
Hence, on the one hand, by choosing $N= N(n,m) = \Nnm$ we find that 
\begin{equation*}
 \sum_{\substack{k=0\\ r=0}}^N \frac{1}{m^{k+r+4} n^{k+r+2}} \sum_{\left(\begin{array}{c}
      (\mathbf{i}, \mathbf{i}')  \\
     (\mathbf{j}, \mathbf{j}')
 \end{array}\right)\in \mathcal{T}^c\cap \mathcal{R}}
 \E[ f_k^{(1)}( \X_{\mathbf{i}},\Y_{\mathbf{j}})f_r^{(1)}( \X_{\mathbf{i}'},\Y_{\mathbf{j}'})]=o\left(\frac{n+m}{mm}\right).
\end{equation*}
On the other hand,
\begin{align*}
&\sum_{\substack{k=0\\ r=0}}^N \frac{1}{m^{k+r+4} n^{k+r+2}} \sum_{\left(\begin{array}{c}
      (\mathbf{i}, \mathbf{i}')  \\
     (\mathbf{j}, \mathbf{j}')
 \end{array}\right)\in \mathcal{T}}
 \E[ f_k^{(1)}( \X_{\mathbf{i}},\Y_{\mathbf{j}})f_r^{(1)}( \X_{\mathbf{i}'},\Y_{\mathbf{j}'})] \\
&= \frac{1}{m^2} \sum_{\substack{k = 0\\ r = 0}}^N\E \left[(\AP\AQ)^{k} \AP \bar\eta_{\x})(\Y_1)(\AP\AQ)^{r} \AP \bar\eta_{\x})(\Y_1)\right]+ o\left(\frac{n+m}{mm}\right)\\
& = \Var  \left[  \frac{1}{ m} \sum_{j=1}^m F^{(1)}_{N,y}(\Y_j)\right] + o\left(\frac{n+m}{mm}\right),
\end{align*}
 where the first equality is consequence of \eqref{eq:varianceIndep} and \eqref{cardinalofT} while the others are consequence of \eqref{varianceofProjection}.
Therefore, \eqref{variancesame} holds via
\begin{align}
     \Var[\U_{m,N}^2]&=\sum_{\substack{k=0\\ r=0}}^N \frac{1}{m^{k+r+4} n^{k+r+2}} \sum_{\left(\begin{array}{c}
      (\mathbf{i}, \mathbf{i}')\notag  \\ \notag 
     (\mathbf{j}, \mathbf{j}')
 \end{array}\right)\in \mathcal{T}^c\cap \mathcal{R}}
 \E[ f_k^{(1)}( \X_{\mathbf{i}},\Y_{\mathbf{j}})f_r^{(1)}( \X_{\mathbf{i}'},\Y_{\mathbf{j}'})]
 \\& \qquad \qquad+\sum_{\substack{k=0\\ r=0}}^N \frac{1}{m^{k+r+4} n^{k+r+2}} \sum_{\left(\begin{array}{c}
      (\mathbf{i}, \mathbf{i}') \\
     (\mathbf{j}, \mathbf{j}')
 \end{array}\right)\in \mathcal{T}}
 \E[ f_k^{(1)}( \X_{\mathbf{i}},\Y_{\mathbf{j}})f_r^{(1)}( \X_{\mathbf{i}'},\Y_{\mathbf{j}'})]\notag\\
\label{eq:varianceApproximation} &= \Var  \left[  \frac{1}{ m} \sum_{j=1}^m F^{(1)}_{N,y}(\Y_j)\right] + o\left(\frac{n+m}{mm}\right).
\end{align}

Moreover, in case $\Var[F^{(1)}_{\infty,y}(\Y)] = 0$, it follows that $\lim_{N\rightarrow \infty} \Var[F^{(1)}_{N,y}(\Y)] = 0$. Hence, as  equality \eqref{eq:varianceApproximation} still remains valid, the convergence in \eqref{eq:ApproxUnm} follows. 
\end{proof}

\end{document}